\documentclass{amsart}
\usepackage{color}
\usepackage[colorlinks,linkcolor=blue,citecolor=blue,urlcolor=blue]{hyperref}
\usepackage{mathrsfs}
\usepackage{enumerate}
\usepackage{graphicx}

\allowdisplaybreaks

\newcommand{\eq}{\begin{equation}}
\newcommand{\en}{\end{equation}}
\newcommand{\lb}[1]{\label{#1}}

\makeatletter

\newcommand{\Rmnum}[1]{\expandafter\@slowromancap\romannumeral #1@}
\makeatother

\newtheorem{thm}{Theorem}[section]
\newtheorem{prop}[thm]{Proposition}
\newtheorem{lem}[thm]{Lemma}
\newtheorem{cor}[thm]{Corollary}
\theoremstyle{definition}
\newtheorem{remark}[thm]{Remark}
\newtheorem{defn}[thm]{Definition}
\newtheorem{hypothesis}[thm]{Hypothesis}
\newtheorem{notation}[thm]{Notation}

\numberwithin{equation}{section}

\newcommand{\reals}{\mathbb{R}}

\newcommand{\8}{\infty}
\newcommand{\ra}{\rightarrow}
\renewcommand{\P}{\mathbb{P}}
\newcommand{\Q}{\mathbb{Q}}
\newcommand{\E}{\mathbb{E}}

\newcommand{\eps}{\varepsilon}
\renewcommand{\and}{ \quad \text{and} \quad }

\begin{document}

\title[Lipschitz minorants]{Lipschitz minorants of \\
Brownian Motion and L\'evy processes}

\author{Joshua Abramson}
\address{J.~Abramson --- \textit{email:} josh@stat.berkeley.edu; \textit{address:} Department of Statistics, University of California, 367 Evans Hall, Berkeley, CA 94720-3860}

\author{Steven N. Evans}
\thanks{S.N.E. was supported in part by N.S.F.\ Grant DMS-0907630} 
\address{S.~N.~Evans --- \textit{email:} evans@stat.berkeley.edu; \textit{address:} Department of Statistics, University of California, 367 Evans Hall, Berkeley, CA 94720-3860}

\subjclass[2010]{60G51, 60G55, 60G17, 60J65}

\keywords{Fluctuation theory, regenerative set, subordinator, 
last exit decomposition, abrupt process, global minimum, 
$c$-convexity, Pasch-Hausdorff envelope}

\begin{abstract} 
For $\alpha > 0$, the $\alpha$-Lipschitz minorant of a function $f:
\reals \to \reals$ is the greatest function $m : \reals \to
\reals$ such that $m \leq f$ and $|m(s)-m(t)| \le \alpha |s-t|$ 
for all $s,t \in \reals$,  should
such a function exist.  
If $X=(X_t)_{t \in \reals}$ is a real-valued L\'evy process
that is not pure linear drift with slope $\pm \alpha$, then
the sample paths of $X$ have an $\alpha$-Lipschitz minorant
almost surely if and only if $| \E[X_1] | <
\alpha$.  Denoting the minorant by $M$, 
we investigate properties of the random closed set
$\mathcal{Z} := \{ t \in \reals : M_t = X_t \wedge X_{t-} \}$, which,
since it is regenerative and stationary, has the distribution of the
closed range of some subordinator ``made stationary'' in a suitable sense.  
We give conditions for the contact set $\mathcal{Z}$
to be countable or to  have zero Lebesgue measure, and
we obtain formulas that characterize the L\'evy measure
of the associated subordinator.
We study the limit of $\mathcal{Z}$ as $\alpha \ra \infty$ and
find for the so-called abrupt L\'evy processes introduced by Vigon
that this limit is the set of local infima of $X$.
When $X$ is a Brownian motion with drift $\beta$ such that $|\beta| <
\alpha$, we calculate explicitly the densities of various random
variables related to the  minorant.  
\end{abstract}

\maketitle

\section{Introduction}
\label{sec:intro}

A function $g: \reals \to \reals$ is $\alpha$-Lipschitz for some $\alpha > 0$ 
if $|g(s) - g(t)| \le \alpha |s - t|$ for all $s,t \in \reals$.
If $\Gamma$ is a set of $\alpha$-Lipschitz functions from $\reals$ to
$\reals$ such that
$\sup\{g(t_0) : g \in \Gamma\} < \infty$ for some $t_0 \in \reals$,
then the function $g^*: \reals \to \reals$ defined by
$g^*(t) = \sup\{g(t) : g \in \Gamma\}$, $t \in \reals$, is $\alpha$-Lipschitz.
Also, if $f: \reals \to \reals$ is an arbitrary function, then
the set of $\alpha$-Lipschitz functions dominated by $f$ is non-empty if
and only if $f$ is bounded below on compact intervals and satisfies
$\liminf_{t \to -\infty} f(t) - \alpha t > - \infty$ and 
$\liminf_{t \to +\infty} f(t) + \alpha t > - \infty$.
Therefore, in this case there is a unique greatest $\alpha$-Lipschitz function
dominated by $f$, and  we call this function the
{\em $\alpha$-Lipschitz minorant} of $f$.  

Denoting the
$\alpha$-Lipschitz minorant of $f$ by $m$, an explicit formula for $m$ is
\eq
\lb{mformula}
\begin{split}
m(t) & = \sup \{ h \in \reals : h - \alpha|t-s| \leq f(s)  \text{ for all } s \in \reals \} \\
& = \inf \{f(s) + \alpha |t-s| : s \in \reals\}. \\
\end{split}
\en
For the sake of completeness, we present a proof of these equalities
in Lemma~\ref{L:minorant_explicit}.
The first equality says that for each $t \in \reals$ we construct
$m(t)$ by considering the set of
``tent'' functions $s \mapsto h - \alpha|t-s|$ that have a peak of height
$h$ at the position $t$ and are dominated by $f$, and then taking
the supremum of those peak heights -- see Figure~\ref{fig:height}.  
The second equality is simply
a rephrasing of the first.

The property that the pointwise supremum of a suitable family of $\alpha$-Lipschitz
functions is also $\alpha$-Lipschitz is reminiscent of the fact
that the pointwise supremum of a suitable family of convex functions is also
convex, and so the notion of the $\alpha$-Lipschitz minorant of a function
is analogous to that of the convex minorant.  Indeed, there is a well-developed
theory of abstract or generalized convexity that subsumes both of these concepts
and is used widely in nonlinear optimization, particularly
in the theory of optimal mass transport -- see 
\cite{MR0348591, MR0452694, MR955448, 
 MR2027382}, Section 3.3 of \cite{MR1619170} and
Chapter 5 of \cite{MR2459454}. Lipschitz minorants have also been
studied in convex analysis for their Lipschitz regularization
and Lipschitz extension properties, and in this area are known
as Pasch-Hausdorff envelopes
\cite{MR593233,MR600082, MR1491362,MR2496900}.

Furthermore, the second expression in \eqref{mformula}
can be thought of as producing a function analogous to
the smoothing of the function $f$ by an integral kernel (that is, a function of the
form $t \mapsto \int_\reals K(|t-s|) f(s) \, ds$ for some suitable kernel $K: \reals \to \reals$)
where one has taken the ``min-plus'' or ``tropical'' point of view
and replaced the algebraic operations of $+$ and $\times$
by, respectively, $\wedge$ and $+$, so that integrals are replaced by infima.
Note that if $f$ is a continuous function that possesses an 
$\alpha_0$-Lipschitz minorant for some $\alpha_0$ (and hence an 
$\alpha$-Lipschitz minorant for all $\alpha \ge \alpha_0$), then the
$\alpha$-Lipschitz minorants converge pointwise monotonically up to $f$ as $\alpha \to +\infty$.
Standard methods in optimization theory involve approximating
a general function by a Lipschitz function and then determining approximate optima
of the original function by finding optima of its Lipschitz approximant
\cite{MR1168181, MR1168182, MR1274246, MR2217473}.

We investigate here the stochastic process $(M_t)_{t \in \reals}$
obtained by taking the $\alpha$-Lipschitz minorant of the sample path of a
real-valued L\'evy process $X=(X_t)_{t \in \reals}$ for which the
$\alpha$-Lipschitz minorant almost surely exists, a condition that turns
out to be equivalent to $| \E [X_1] | < \alpha$ when $X_0 = 0$ (excluding
the trivial case where $X_t = \pm \alpha t$ for $t \in \reals$)
-- see Proposition~\ref{P:existence}.  See Figure~\ref{fig:bmexample}
for an example of the minorant of a Brownian sample path.
Our original motivation for this undertaking was the 
abovementioned analogy between
$\alpha$-Lipschitz minorants and convex minorants
and the rich (and growing) literature on convex minorants of
Brownian motion and L\'evy processes in general 
\cite{MR714964, MR733673, MR770946, MR1145458, MR1747095, MR1891744,suidan,ECP2011-38,pitmanbravo}.

In particular, we study
properties of the {\em contact set}
$\mathcal{Z} := \{ t \in \reals : M_t = X_t \wedge X_{t-} \}$.  
This random set is clearly stationary and,
as we show in Theorem~\ref{thm:Zregen}, 
it is also regenerative.  Consequently, its
distribution is that of the closed range of a subordinator ``made
stationary'' in a suitable manner.  For a broad class
of L\'evy processes we are able to identify the associated subordinator
in the sense that we can determined its Laplace exponent
-- see Theorem~\ref{thm:bigthm}.

We then consider the Lebesgue measure of the random set $\mathcal{Z}$
in Theorem~\ref{thm:Zmeasure} and Remark~\ref{R:Zmeasure}. 
If the paths of the L\'evy process have either
unbounded variation or bounded variation 
with drift $d$ satisfying $|d |> \alpha$, then
the associated subordinator has zero drift, and hence 
the random set $\mathcal{Z}$ has zero
Lebesgue measure almost surely. Conversely, if the paths of the L\'evy process 
have bounded variation and
drift $d$ satisfying $|d| < \alpha$, then the associated
subordinator has positive
drift, and hence the random set $\mathcal{Z}$ has infinite Lebesgue measure
almost surely. In Theorem~\ref{thm:newLambda} we give conditions under
which the L\'evy measure of the subordinator associated to the set
$\mathcal{Z}$ has finite total mass, which implies that $\mathcal{Z}$
is a discrete set in the case where it has zero Lebesgue measure. Using
the methodologies developed to investigate the Lebesgue measure of
$\mathcal{Z}$ we give in Theorem~\ref{thm:localExtrema} an interesting
result relating to the local behavior of a L\'evy process at its
local extrema.

If for the moment we write $\mathcal{Z}_\alpha$ instead of $\mathcal{Z}$
to stress the dependence on $\alpha$, then it is clear
that $\mathcal{Z}_{\alpha'} \subseteq \mathcal{Z}_{\alpha''}$ for $\alpha' \le \alpha''$.
We find in Theorem~\ref{thm:Zlimit}
that if the L\'evy process is {\em abrupt}, that is, its paths
have unbounded variation and ``sharp'' local extrema in a suitable sense
(see Definition~\ref{def:abrupt} for a precise definition), then the set
$\bigcup_\alpha \mathcal{Z}_\alpha$ 
is almost surely the set of local infima of 
the L\'evy process.

\begin{figure}
\begin{center}
\includegraphics[width=\textwidth]{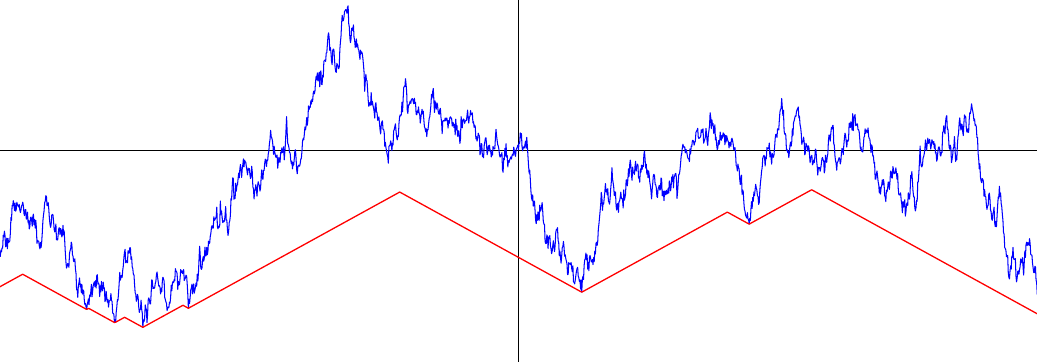}
\end{center}
\caption{A typical Brownian motion sample path and its associated Lipschitz minorant.}
\label{fig:bmexample}
\end{figure}

Lastly, when the L\'evy process is a Brownian motion
with drift, we can compute explicitly the distributions of a number
of functionals of the  $\alpha$-Lipschitz minorant process.
In order to describe these results, we first note
that it follows from Lemma~\ref{L:sawtooth} below 
that the graph of the $\alpha$-Lipschitz
minorant $M$ over one of the connected components
of the complement of $\mathcal{Z}$ is almost surely a ``sawtooth''
that consists of a line of slope $+\alpha$
followed by a line of slope $-\alpha$.  
Set 
$G := \sup \{ t < 0 : t \in \mathcal{Z} \}$, 
$D := \inf \{ t > 0 : t \in \mathcal{Z} \}$, and put $K: = D-G$.
Let $T$ be the unique $t \in [G,D]$ such that $M(t) = \max\{M(s) : s \in [G,D]\}$.
That is, $T$ is place where the peak of the sawtooth occurs.  Further,
let $H := X_T - M_T$ be the distance between the Brownian path and
the $\alpha$-Lipschitz minorant at the time where the peak occurs.

The following theorem summarizes a series of
results that we establish in  Section~\ref{sec:bm}.

\begin{thm}
\label{thm:bmsummary}
Suppose that $X$ is a Brownian motion with drift $\beta$, where
$|\beta|<\alpha$. Then, the following hold.
\begin{itemize}
\item[(a)]
The L\'evy measure $\Lambda$ of the subordinator associated to the
contact set $\mathcal{Z}$ has finite mass and 
is characterized up to a scalar multiple by
\[
\frac{\int_{\reals_+} 1 - e^{-\theta x} \, \Lambda(dx)}
{\int_{\reals_+} x \, \Lambda(dx)}
=  
\frac{4 (\alpha^2 - \beta^2) \theta}
{
\left(\sqrt{2 \theta + (\alpha - \beta)^2} + \alpha - \beta \right)
\left(\sqrt{2 \theta + (\alpha + \beta)^2} + \alpha + \beta \right)
}
\] 
\item[(b)]
When $\beta = 0$ the measure $\Lambda$ is absolutely continuous
with respect to Lebesgue measure with 
\[
\frac{\Lambda(d x)}{\Lambda(\reals_+)} 
 =
\frac{2 \alpha }{ \sqrt{2 \pi} }
\left[ x^{-\frac{1}{2}} e^{- \frac{\alpha^2 x}{2}} - 2 \alpha^2 \Phi(- \alpha x^{\frac{1}{2}}) \right] \, dx,
\] 
where $\Phi$ is the standard normal cumulative distribution function
(that is, $\Phi(z) := \int_{-\infty}^{z} \frac{1}{\sqrt{2 \pi}} e^{-\frac{t^2}{2}} \, dt$).
\item[(c)] 
The distribution of $T$ is characterized by
\[
\E[e^{-\theta T}] = 8 \alpha (\alpha^2-\beta^2) \frac{1}{\theta}
\left( \frac{1}{ \sqrt{(\alpha+\beta)^2 - 2 \theta} + 3 \alpha - \beta
  } - \frac{1}{ \sqrt{(\alpha-\beta)^2 + 2 \theta} + 3 \alpha + \beta
  } \right)
\]
for $- \frac{(\alpha-\beta)^2}{2} \le \theta \le
\frac{(\alpha+\beta)^2}{2} $.  Also,
\[
\P\{T>0\} = \frac{1}{2} \left( 1 + \frac{\beta}{\alpha} \right).
\]
\item[(d)]
The random variable
$H$ has a ${\mathrm{Gamma}}(2, 4 \alpha)$ distribution; that is,
the distribution of $H$ is absolutely continuous with
respect to Lebesgue measure with density
$h \mapsto (4 \alpha)^2 h e^{-4 \alpha h}$, $h \ge 0$.
\end{itemize}
\end{thm}

The rest of this article is organized as follows. In
Section~\ref{sec:main} we provide precise definitions and give some
preliminary results relating to the nature of the contact set. 
In Section~\ref{sec:identify} we describe the subordinator associated
with the contact set, and in Section~\ref{sec:limit} we describe the
limit of the contact set as $\alpha \ra \infty$.
In order to prove Theorem~\ref{thm:bigthm} we need some preliminary
results relating to the future infimum of a L\'evy process, which we
give in Section~\ref{sec:prelim}, and then we prove
Theorem~\ref{thm:bigthm} in Section~\ref{sec:proofs}.
In Section~\ref{sec:bm} we cover the special case when $X$ is a two
sided Brownian motion with drift in detail.
Finally, in Section~\ref{sec:facts} we give some basic facts about the
$\alpha$-Lipschitz minorant of a function that are helpful throughout
the paper.

\section{Definitions and Preliminary Results}
\label{sec:main}

\subsection{Basic definitions}
\label{sec:defs}

Let $X = (X_t)_{t \in \reals}$ be a real-valued L\'evy process.  
That is, 
$X$ has {\em c\`adl\`ag} sample paths, $X_0=0$, 
and $X_t-X_s$ is independent of 
$\{X_r : r \le s\}$ with the same distribution as $X_{t-s}$ for
$s,t \in \reals$ with $s<t$. 

The L\'evy-Khintchine
formula says that for $t \geq 0$ the characteristic function of $X_t$ is given by
$\E[e^{i \theta X_t}] = e^{-t \Psi(\theta)}$ for $\theta \in \reals$,
where
\[
\Psi(\theta) = - i a \theta + \frac{1}{2} \sigma^2 \theta^2 +
\int_\reals(1 - e^{i \theta x} + i \theta x 1_{ \{ |x| < 1 \} } )
\, \Pi(dx) 
\]
with $a \in \reals$, $\sigma \in \reals_+$, and $\Pi$ a $\sigma$-finite
measure concentrated on $\reals \setminus \{0\}$ 
satisfying $\int_\reals (1 \wedge x^2) \, \Pi(dx) <
\infty$. We call $\sigma^2$ the {\em infinitesimal variance} of the
Brownian component of $X$ and $\Pi$ the
{\em L\'evy measure} of $X$. 

The sample paths of $X$ have bounded variation almost surely
if and only if $\sigma = 0$ and
$\int_\reals (1 \wedge |x| ) \, \Pi(dx) < \infty$. In this case $\Psi$
can be rewritten as
\[
\Psi(\theta) = - i d \theta + \int_\reals (1-e^{i \theta x} ) \, \Pi(dx).
\]
We call $d \in \reals$ the drift coefficient.
For full details of these definitions see \cite{bertoin}.

We will often need the result of Shtatland \cite{shtatland}
that if $X$ has paths of bounded variation with drift $d$, then
\begin{equation}
\label{eq:shtatland_statement}
\lim_{t \downarrow 0} t^{-1} X_t = d \quad \text{a.s.}
\end{equation}
The counterpart of Shtatland's result when $X$
has paths of unbounded variation is Rogozin's result
\begin{equation}
\label{eq:Rogozin_small_time}
\liminf_{t \downarrow 0} t^{-1} X_t = - \infty
\and
\limsup_{t \downarrow 0} t^{-1} X_t = + \infty \quad \text{a.s.}
\end{equation}

For the sake of reference, we also record here a regularity
criterion due to Rogozin \cite{rogozin} 
(see also \cite[Proposition
VI.11]{bertoin}) that we use frequently:
\begin{equation}
\label{eq:rogozin_regularity}
\begin{split}
& \text{zero is regular for $(-\infty,0]$} \\
& \qquad \Longleftrightarrow \\
& \int_0^1 t^{-1} \P\{X_t \le 0\} \, dt = \infty. \\
\end{split}
\end{equation}
Of course, \eqref{eq:rogozin_regularity} has an obvious
analogue that determines when zero is regular for $[0,\infty)$.

\subsection{Existence of a minorant}

\begin{prop}
\label{P:existence}
Let $X$ be a L\'evy process.
The $\alpha$-Lipschitz minorant of $X$ exists almost surely if and
only if  either $\sigma = 0$, $\Pi=0$ and $|d| = \alpha$ (equivalently,
$X_t = \alpha t$ for all $t \in \reals$ or 
$X_t = -\alpha t$ for all $t \in \reals$), or
$\E[|X_1|] < \infty$ and $| \E [X_1]| < \alpha$.
\end{prop}

\begin{proof}
As we remarked in the Introduction,
the $\alpha$-Lipschitz minorant of a function $f: \reals \to
\reals$ exists if and only if $f$ is bounded below on
compact intervals and satisfies
$\liminf_{t \to -\infty} f(t) - \alpha t > - \infty$ and 
$\liminf_{t \to +\infty} f(t) + \alpha t > - \infty$.

Since the sample paths of a L\'evy process are
almost surely bounded on compact intervals, we need necessary
and sufficient conditions for 
$\liminf_{t \to -\infty} X_t - \alpha t > - \infty$ and 
$\liminf_{t \to +\infty} X_t + \alpha t > - \infty$
to hold almost surely.
This is equivalent to requiring that
\eq
\lb{conditions}
\limsup_{t \ra +\infty} X_t - \alpha t < +\infty \quad \text{a.s.} 
\qquad \text{and} \qquad  
\liminf_{t \ra +\infty}  X_t + \alpha t > -\infty \quad \text{a.s.}
\en

It is obvious that 
that the two conditions in \eqref{conditions} hold if $\sigma = 0$, $\Pi=0$ and $|d| = \alpha$.
It is clear from the strong law of large numbers that they also hold
if $\E[|X_1|] < \infty$ and $| \E [X_1]| < \alpha$.

Consider the converse.
Writing $x^+ := x \vee 0$ and $x^- := - (x \wedge 0)$ for $x \in \reals$,
the strong law of large numbers precludes any case where either 
$\E [X_1^+] = +\infty$ and $\E [X_1^-] < + \infty$ 
or 
$\E [X_1^+] < +\infty$ and $\E [X_1^-] = +\infty$. 
A result of Erickson \cite[Chapter 4, Theorem 15]{doney}
rules out the possibility $\E [X_1^+] = \E [X_1^-] = +\infty$,
and so $\E[|X_1|] < \infty$.  It now follows from
the strong law of large numbers that $\lim_{t \to \infty} t^{-1} X_t = \E[X_1]$
and so $|\E[X_1]| \le \alpha$.  Suppose that $X_t$ is non-degenerate for $t \ne 0$ (that is,
that $\sigma \ne 0$ or $\Pi \ne 0$).  Then,
$\limsup_{t \to \infty} X_t - \E[X_1] t = +\infty$ a.s.
and
$\liminf_{t \to \infty} X_t - \E[X_1] t = -\infty$ a.s.
(see, for example, \cite[Corollary 9.14]{MR1876169}),
and so $|\E[X_1]| < \alpha$ in this case. 
\end{proof}

\begin{hypothesis}
\label{H:standing}
From now on we assume, unless we note otherwise,
that the L\'evy process $X = (X_t)_{t \in \reals}$
has the properties:
\begin{itemize}
\item
$X_0 = 0$;
\item
$X_t$ is non-degenerate for $t \ne 0$;
\item
$\E[|X_1|]<\infty$;
\item 
$| \E [X_1]| < \alpha$. 
\end{itemize}
\end{hypothesis}

\begin{notation}
As in the Introduction, let 
$M = (M_t)_{t \in \reals}$ be the $\alpha$-Lipschitz
minorant of $X$.  Put
$\mathcal{Z} = \{ t \in \reals : M_t = X_t \wedge X_{t-} \}$.
\end{notation}

\subsection{The contact set is regenerative}

It follows fairly directly from our standing assumptions 
Hypothesis~\ref{H:standing} that the random set $\mathcal{Z}$
is almost surely unbounded above and below. (Alternatively,
it follows even more easily from Hypothesis~\ref{H:standing}
that $\mathcal{Z}$ is non-empty almost surely. We show below that
$\mathcal{Z}$ is stationary, and any non-empty stationary random set
is necessarily almost surely unbounded above and below.)

We now show that the contact set $\mathcal{Z}$ is stationary and that it is
also regenerative in the sense of Fitzsimmons and Taksar \cite{fitztaksar}. 
For simplicity, we specialize the definition in \cite{fitztaksar} somewhat
as follows by only considering random sets defined on probability spaces
(rather than general $\sigma$-finite measure spaces).

Let $\Omega^0$ denote the class of closed subsets of $\reals$. For $t
\in \mathbb{R}$ and $\omega^0 \in \Omega^0$, define
\[
d_t(\omega^0) := \inf \{ s>t: s \in \omega^0 \}, \quad r_t(\omega^0) :=
d_t(\omega^0) - t,
\]
and
\[
\tau_t(\omega^0) = \text{\textbf{cl}} \{ s-t:s \in \omega^0 \cap (t,\infty) \}
= \text{\textbf{cl}} \left( (\omega^0-t) \cap (0,\infty) \right) .
\]
Here $\text{\textbf{cl}}$ denotes closure and we adopt the
convention $\inf \emptyset = + \infty$. 
Note that 
$t \in \omega^0$ if and only if 
$\lim_{s \uparrow t} r_s(\omega^0) = 0$, and so
$\omega^0 \cap (-\infty,t]$ can be reconstructed from
$r_s(\omega^0)$, $s \le t$, for any $t \in \mathbb{R}$.
Set
$\mathcal{G}^0 = \sigma \{ r_s : s \in \reals \}$ and $\mathcal{G}_t^0 =
\sigma \{ r_s:s \leq t \}$. Clearly, $(d_t)_{t \in \reals}$ is an
increasing c\`adl\`ag process adapted to the filtration
$(\mathcal{G}_t^0)_{t \in \reals}$, and $d_t \ge t$ for all $t \in \reals$.

A \emph{random set} is a measurable mapping $S$ from a measurable
space $(\Omega,\mathcal{F})$ into
$(\Omega^0,\mathcal{G}^0)$.

\begin{defn}
\label{def:regenset}
A probability measure $\Q$ on $(\Omega^0,\mathcal{G}^0)$
is regenerative with regeneration law $\Q^0$ if
\begin{enumerate}[(i)]
  \item $\Q\{d_t = +\infty\} = 0$, for all $t \in \reals$; 
  \item for all $t \in \reals$ and
    for all $\mathcal{G}^0$-measurable nonnegative functions $F$,
\eq
\lb{eq:regen_law}
\Q \left [F(\tau_{d_t}) \, | \, \mathcal{G}_{t+}^0 \right] = \Q^0[F], 
\en
where we write $\Q[\cdot]$ and $\Q^0[\cdot]$ for expectations
with respect to $\Q$ and $\Q^0$.
A random set $S$ defined on a probability space 
$(\Omega,\mathcal{F}, \P)$ is a regenerative set if the
push-forward of $\P$ by the map $S$ (that is,
the distribution of $S$) is a regenerative probability measure.
\end{enumerate}
\end{defn}

\begin{remark}
\label{R:zero_enough}
Suppose that the probability measure $\Q$ on $(\Omega^0,\mathcal{G}^0)$
is stationary; that is, if $S^0$ is the identity map on $\Omega^0$,
then the random set $S^0$ on $(\Omega^0,\mathcal{G}^0, \Q)$
has the same distribution as $u+S^0$ for any $u \in \reals$ or, equivalently,
that the process $(r_t)_{t \in \reals}$ has the same distribution as
$(r_{t-u})_{t \in \reals}$ for any $u \in \reals$.  Then,
in order to check conditions (i) and (ii) of Definition~\ref{def:regenset}
it suffices to check them for the case $t=0$.
\end{remark}

The probability measure $\Q^0$ is itself regenerative. It assigns all of
its mass to the collection of closed subsets of $\reals_+$.
As remarked in \cite{fitztaksar}, it is well known that 
any regenerative probability measure with this
property  arises as the distribution of 
a random set of the form  $\text{\textbf{cl}}\{Y_t : Y_t > Y_0, \, t \ge 0\}$,
where $(Y_t)_{t \ge 0}$ is a subordinator (that is, a non-decreasing,
real-valued L\'evy process) with $Y_0 = 0$ -- see \cite{maisonregen2,maisonregen1}.
Note that $\text{\textbf{cl}}\{Y_t : Y_t > Y_0, \, t \ge 0\}$ has
the same distribution as $\text{\textbf{cl}}\{Y_{ct} : Y_{ct} > Y_{c0}, \, t \ge 0\}$,
and the distribution of the subordinator associated with a regeneration law
can at most be determined up to linear time change (equivalently,
the corresponding drift and L\'evy measure can at most be determined  up
to a common constant multiple).  It turns out that the distribution
of the subordinator is unique except for this ambiguity -- again see 
\cite{maisonregen2,maisonregen1}.

We refer the reader to \cite{fitztaksar} for a description of the
sense in which a stationary regenerative probability measure $\Q$ with
regeneration law $\Q^0$ can be regarded as $\Q^0$ ``made
stationary''. Note that if $\Lambda$ is the L\'evy measure of the
subordinator associated with $\Q$ in this way, then, by
stationarity, it must be the case that $\int_{\reals_+} y \, \Lambda(dy) < \infty$.

\begin{thm}
\label{thm:Zregen}
The random (closed) set $\mathcal{Z}$ is stationary and regenerative.
\end{thm}

\begin{proof} 
It follows from Lemma~\ref{L:sawtooth} that
$\mathcal{Z}$ is a.s.\ closed.

We next show that $\mathcal{Z}$ is stationary.
Note for $u \in \reals$ that $u + \mathcal{Z}
= \{t \in \reals : X_{(t-u)} \wedge X_{(t-u)-} = M_{(t-u)}\}$.
Define $(\breve X_t)_{t \in \reals}$ by $\breve X_t = X_{t-u} - X_{(-u)}$ 
for $t \in \reals$ and let $\breve M$ be the $\alpha$-Lipschitz minorant of $\breve X$.
Note that $\breve M_t = M_{t-u} - X_{(-u)}$ for $t \in \reals$.  Therefore,
$u + \mathcal{Z} = \{t \in \reals : \breve X_t \vee \breve X_{t-} = \breve M_t\}$
and hence $u + \mathcal{Z}$ has the same distribution as $\mathcal{Z}$
because $\breve X$ has the same distribution as $X$.

We now show that $\mathcal{Z}$ is regenerative.
For $t \in \reals$ set
\[
D_t := \inf \{ s > t : X_s \wedge X_{s-} = M_s \} 
= d_t \circ \mathcal{Z},
\]
\[
R_t := D_t - t,
\]
\[
S_t := 
\inf \left\{ s > t : X_s \wedge X_{s-} - \alpha (s-t) 
\leq 
\inf\{ X_u  - \alpha (u-t) : u \le t\}  \right\},
\]
and 
$\mathcal{F}_t := \bigcap_{s > t} \sigma\{X_u : u \le s\}$.

We claim that $X_{S_t} \le X_{S_t-}$ \ a.s.  Suppose to the contrary that
the event $A := \{X_{S_t} > X_{S_t-}\}$ has positive probability.
On the event $A$, $X_s > X_{S_t} - (X_{S_t} - X_{S_t-})/2$
for $s \in (S_t, S_t + \delta)$ for some (random) $\delta > 0$.
Hence,  $X_{s-} \ge X_{S_t} - (X_{S_t} - X_{S_t-})/2$
for $s \in (S_t, S_t + \delta)$, and so 
$X_s \wedge X_{s-} - \alpha(s - t) 
> X_{S_t} \wedge X_{S_t-} - \alpha(S_t-t) = X_{S_t-} - \alpha(S_t-t)$
for $s \in (S_t, S_t + \delta)$ on the event $A$ provided
$\delta$ is sufficiently small.  It follows that, on the event $A$,
$X_{S_t-} - \alpha(S_t-t) \le \inf\{ X_u  - \alpha (u-t) : u \le t\}$
and 
$X_s \wedge X_{s-} - \alpha(s-t) 
> \inf\{ X_u  - \alpha (u-t) : u \le t\}$ for $t \le s < S_t$.  
Define a non-decreasing
sequence of stopping times $\{U_n\}_{n \in \mathbb{N}}$ by 
\[
U_n := \inf\left\{s > t : X_s \wedge X_{s-} - \alpha(s-t)
\le 
\inf\{ X_u  - \alpha (u-t) : u \le t\} + \frac{1}{n}\right\},
\]
and set $U_\infty := \sup_{n \in \mathbb{N}} U_n$.
We have shown that, on the event $A$, 
$U_n < S_t$ for all $n \in \mathbb{N}$ and $U_\infty = S_t$.
By the quasi-left-continuity of $X$, 
$\lim_{n \to \infty} X_{U_n} = X_{U_\infty}$ \ a.s.  In particular,
$X_{S_t} = X_{S_t-}$ \ a.s. on the event $A$, and so
$A$ cannot have positive probability.

Lemma \ref{lem:recipe} now gives that
\[
D_t = \inf \left\{ s \ge S_t : X_t \wedge X_{t-} + \alpha (s-S_t) 
=
\inf\{ X_u + \alpha (u-S_t) : u \ge S_t \} \right\}
\]
almost surely.

We have already remarked that $\mathcal{Z}$ is almost surely unbounded 
above and below, and hence condition (i) of Definition~\ref{def:regenset} holds.
By Remark~\ref{R:zero_enough}, in order to check condition (ii) of 
Definition~\ref{def:regenset}, it suffices to consider the case $t=0$.

For notational simplicity, set $S := S_0$ and $D := D_0$
-- see Figure~\ref{fig:recipe} for two illustrations of the
construction of $S$ and $D$ from a sample path. 
For a random time $U$, let  $\mathcal{F}_{U}$ be the $\sigma$-field generated
by random variables of the form $\xi_U$, where $(\xi_t)_{t \in \reals}$
is some optional process for the filtration $(\mathcal{F}_t)_{t \in \reals}$
(cf. \ Millar \cite{millar_77, millarpostmin}).
It follows from Corollary~\ref{C:flow} (where we are thinking intuitively of
removing the process to the right of $D$ rather than to the right of
zero) that 
$\bigcap_{\epsilon > 0} \sigma\{R_s : s \le \epsilon\} \subseteq \mathcal{F}_{D}$.

Put
\[
\tilde {X} = (\tilde{X}_s)_{s \ge 0} := \left( (X_{S+s} - X_S) + \alpha s
\right)_{s \geq 0}.
\]
By the strong Markov property at the stopping time $S$ and the spatial homogeneity of $X$, the process
$\tilde X$ is independent of $\mathcal{F}_S$ with the same distribution as the L\'evy process
$(X_t + \alpha t)_{t \ge 0}$.
Suppose for the L\'evy process $(X_t + \alpha t)_{t \ge 0}$ that zero is regular for the interval $(0,\infty)$.
A result of Millar
\cite[Proposition $2.4$]{millarzeroone} implies that almost surely
there is a unique time $\tilde{T}$ such that $\tilde{X}_{\tilde{T}} =
\inf\{\tilde{X}_s : s \ge 0\}$ and that if $\bar T$ is
such that $\tilde{X}_{\bar T -} = \inf\{\tilde{X}_s : s \ge 0\}$,
then $\bar T = \tilde T$.  Thus,
$\tilde{T} 
= \sup \{ t \ge 0 : \tilde{X}_t \wedge \tilde{X}_{t-} = \inf\{\tilde{X}_s : s \ge 0\} \}$
and
$D = S + \tilde T$.
Combining this observation with the main result of Millar
\cite{millarpostmin} (see Remark~\ref{rem:millar} below) and the fact
that $\tilde{X}_{\tilde{T}} = \inf\{\tilde{X}_s : s \ge 0\}$
gives that $(\tilde{X}_{\tilde{T}+t})_{t \ge 0}$
is conditionally independent of $\mathcal{F}_D$ 
given $\tilde{X}_{\tilde{T}}$. Thus, again by the spatial homogeneity of $\tilde X$,
$(\tilde{X}_{\tilde{T}+t} - \tilde{X}_{\tilde{T}})_{t \ge 0}$
is independent of $\mathcal{F}_D$.  This establishes 
condition (ii) of Definition~\ref{def:regenset} for $t=0$.

If zero is not regular for the interval $(0,\infty)$ for the L\'evy process
$(X_t + \alpha t)_{t \ge 0}$, then zero is necessarily regular for 
the interval $(0,\infty)$
for the L\'evy process $(X_{-t-} + \alpha t)_{t \ge 0}$ because 
this latter process
the same distribution as 
$(-(X_t + \alpha t) + 2 \alpha t)_{t \ge 0}$.
The argument above then establishes that the random
set $-\mathcal{Z}$ is regenerative.  It follows from 
\cite[Theorem 4.1]{fitztaksar} that
$\mathcal{Z}$ is regenerative with the same distribution
as $-\mathcal{Z}$.
\end{proof}

\begin{remark}
\label{rem:millar}
A key ingredient in the proof of Theorem~\ref{thm:Zregen} was the
result of Millar from \cite{millarpostmin} which says that, under suitable
conditions, the future evolution of a c\`adl\`ag strong Markov process
after the time it attains its global minimum is conditionally
independent of the past up to that time given the value
of the process and its left limit at that time.  That result
follows in turn from results in \cite{MR0334335} on last exit decompositions
or results in \cite{MR0297019} on analogues of the strong Markov
property at general coterminal times.  We did not apply 
Millar's result directly; rather, we considered
a random time $D = D_0$ that was the last time after a stopping
time that a strong Markov process attained its infimum over times
greater than the stopping time and combined Millar's result
with the strong Markov property at the stopping time.  
An alternative route would have been to observe that the random
time $D$  is a {\em randomized coterminal time}
in the sense of \cite{millar_77} for a suitable strong Markov process.
\end{remark}

\section{Identification of the associated subordinator}
\label{sec:identify}

Let $Y = (Y_t)_{t \ge 0}$ be ``the'' subordinator associated with the
regenerative set $\mathcal{Z}$. Write $\delta$ and $\Lambda$
for the drift coefficient and L\'evy measure of $Y$.  
Recall that these quantities are unique up to a common scalar
multiple. The closed range of $Y$
either has zero Lebesgue measure almost surely or infinite
Lebesgue measure almost surely according to whether
$\delta$ is zero or positive  \cite[Chapter 2, Theorem
3]{doney}.  Consequently, the same
dichotomy holds for the contact set $\mathcal{Z}$,
and the following result gives necessary and sufficient
conditions for each alternative.

\begin{thm}
\label{thm:Zmeasure}
If $\sigma=0$, $\Pi(\reals) < \infty$, and $|d|=\alpha$, then
the Lebesgue measure of $\mathcal{Z}$ is almost surely infinite.
If $X$ is not of this form, then the Lebesgue measure of $\mathcal{Z}$
is almost surely zero if and only if zero is regular for the interval
$(-\infty,0]$ for at least one of the L\'evy processes 
$(X_t + \alpha t)_{t\ge 0}$ and $(-X_t+\alpha t)_{t \ge 0}$.  
\end{thm}

\begin{proof} 
Suppose first that $\sigma=0$, $\Pi(\reals) < \infty$ and $|d|=\alpha$.
In this case, the paths of $X$ are piecewise linear with slope $d$.
Our standing assumption $|\E[X_1]| < \alpha$ and the
strong law of large numbers give 
$\lim_{t \to -\infty} t^{-1} X_t = \lim_{t \to +\infty} t^{-1} X_t = \E[X_1]$.
It is now clear that $\mathcal{Z}$ has positive Lebesgue measure
with positive probability and hence infinite Lebesgue measure almost surely.

Suppose now that $X$ is not of this special form. 
It suffices by Fubini's theorem and the stationarity
of $\mathcal{Z}$ to show that
$\P\{ 0 \in \mathcal{Z} \} > 0$ if and only if zero is not regular for
$(-\infty,0]$ for both of the L\'evy processes 
$(X_t + \alpha t)_{t \ge 0}$ and
$(-X_t+\alpha t)_{t \ge 0}$.

Set 
$I^- := \inf\{X_t - \alpha t : t \leq 0\}$
and
$I^+ := \inf\{X_t + \alpha t : t \geq 0\}$.
Recall from \eqref{mformula} that $M_0 = I^{-} \wedge I^+$.
Therefore, 
\[
\begin{split}
\P\{0 \in \mathcal{Z}\}  
& = \P\{I^{-} \wedge I^+ = X_0 \wedge X_{0-} = 0\} \\
& = \P\{I^- = I^+ = 0\} \\
& = \P\{I^- = 0\} \P\{I^+ = 0\}, \\
\end{split}
\]
and so $\P\{0 \in \mathcal{Z}\} > 0$ if and only if $\P\{I^- = 0\} > 0$ and $\P\{I^+ = 0\} > 0$.

Note that $I^-$ has the same distribution as 
$\inf\{-X_t + \alpha t : t \geq 0\}$.
From the formulas of Pecherskii and Rogozin
\cite{pechrogo} (or \cite[Theorem VI.5]{bertoin}),
\eq
\lb{pechrogo2}
\E[e^{\theta I^-}] = \exp \left( \int_0^\8 \int_{(-\8,0]} (e^{\theta x} - 1)t^{-1} \P\{-X_{t} + \alpha t \in dx\} \, dt\right)
\en
and
\eq
\lb{pechrogo1}
\E[e^{\theta I^+}] = \exp \left( \int_0^\8 \int_{(-\8,0]} (e^{\theta x} - 1)t^{-1} \P\{X_t + \alpha t \in dx\} \, dt \right).
\en
Taking the limit as $\theta \ra \infty$ and applying monotone
convergence in \eqref{pechrogo2} and in \eqref{pechrogo1} gives
\eq
\lb{pechrogo4}
\P\{I^- = 0\} = \exp \left( - \int_0^\8 t^{-1} \P\{-X_{t} + \alpha t < 0\} \, dt \right)
\en 
and
\eq
\lb{pechrogo3}
\P\{I^+ = 0\} = \exp \left( - \int_0^\8 t^{-1} \P\{X_t + \alpha t < 0\} \, dt \right).
\en

Since we are assuming that it is not the case that
$\sigma = 0$,  $\Pi(\reals) < \infty$ and $|d|=\alpha$, 
we have $\P\{X_t + \alpha t = 0\} = \P\{-X_{t} +
\alpha t = 0\} = 0$ for all $t>0$. 
Moreover, by our standing assumption $|\E[X_1] | < \alpha$ it certainly follows
that both $X_t + \alpha t$ and $-X_{t} + \alpha t$ drift to $+\infty$. 
Hence,
by a result of Rogozin \cite{rogozin} 
(or see \cite[Theorem VI.12]{bertoin})
\eq
\lb{reg2}
\int_1^\infty t^{-1} \P\{ X_t + \alpha t \leq 0 \} \, dt < \infty 
\and
\int_1^\infty t^{-1} \P\{ - X_{t} + \alpha t \leq 0 \} \, dt < \infty.
\en

The result now follows from \eqref{eq:rogozin_regularity}
which implies that zero is not regular for the interval
$(-\infty,0]$ for both $(-X_{t}+\alpha t)_{t \ge 0}$
and $(X_t + \alpha t)_{t \ge 0}$ 
 if and only if
\[
\int_0^1 t^{-1} \P\{ -X_{t} + \alpha t \leq 0 \} \, dt < \infty
\and
\int_0^1 t^{-1} \P\{ X_t + \alpha t \leq 0 \} \, dt < \infty.
\]
\end{proof} 

\begin{remark}
\label{R:Zmeasure}
\begin{itemize}
\item[(i)]
Note that zero is regular for the interval $(-\infty,0]$ for both $(X_t +
\alpha t)_{t \ge 0}$ and $(-X_{t}+\alpha t)_{t \ge 0}$ when $X$ has
paths of unbounded variation, since then $\liminf_{t \ra 0} t^{-1} X_t =
-\infty$ by \eqref{eq:Rogozin_small_time}.
\item[(ii)]
If $X$ has paths of bounded variation and drift coefficient $d$, then $\lim_{t
  \downarrow 0} t^{-1}(X_t+ \alpha t) = d+ \alpha$ and $\lim_{t \downarrow 0}
t^{-1}(-X_{t}+ \alpha t) = -d + \alpha$ 
by \eqref{eq:shtatland_statement}. Thus, if
$|d| < \alpha$, then zero is regular for $(-\infty,0]$ for
neither  $(X_t + \alpha t)_{t \ge 0}$ or $(-X_{t}+\alpha t)_{t \ge
  0}$, whereas if $|d| > \alpha$, 
then zero is regular for $(-\infty,0]$ for exactly one of those two processes.
\item[(iii)]
If $X$ has paths of bounded variation and $|d| = \alpha$, 
then an integral condition due to 
Bertoin involving the
L\'evy measure $\Pi$  determines whether zero
is regular for the interval $(-\infty,0]$ 
for whichever of the processes $(X_t +\alpha t)_{t \ge 0}$ 
or $(-X_{t}+\alpha t)_{t \ge 0}$ has zero drift
coefficient \cite{bertoin97}.
\end{itemize}
\end{remark}

\begin{remark}
\label{R:size-biasing}
Recall the notation 
$G = \sup \{ t < 0 : t \in \mathcal{Z} \}$, 
$D = \inf \{ t > 0 : t \in \mathcal{Z} \}$ and  $K = D-G$ (note that $D = d_0 \circ \mathcal{Z}$). If the Lebesgue measure of
$\mathcal{Z}$ is almost surely zero (equivalently when $\delta = 0$
\cite[Chapter 2, Theorem 3]{doney}), then $0 \notin \mathcal{Z}$
and $G < 0 < D$, and the distribution of $K$ is obtained by
size-biasing the L\'evy measure $\Lambda$; that is,
\begin{equation}
\label{E:K_size_bias_Lambda}
\P\{K \in dx\} = \frac{x \, \Lambda(dx)}{\int_{\reals_+} y \, \Lambda(dy)}
\end{equation}
(recall that $\int_{\reals_+} y \, \Lambda(dy) < \infty$ since
$\mathcal{Z}$ is stationary).

If the Lebesgue measure of $\mathcal{Z}$ is positive almost surely
(and hence $\delta > 0$),
then $\P\{K=0\} > 0$ and we see by multiplying together 
\eqref{pechrogo4} and \eqref{pechrogo3} that 
\begin{equation}
\label{p0}
\begin{split}
\P\{K=0\} 
& = \exp \left( - \int_0^\infty t^{-1} 
\left( \P\{X_t + \alpha t < 0\}+ \P\{-X_{t} + \alpha t < 0\} \right) 
\, dt \right) \\
& = \exp \left(- \int_0^\infty t^{-1} \P\{X_t \notin [-\alpha t, \alpha t]\} \, dt \right). \\
\end{split}
\end{equation}
In this latter case, the conditional distribution of $K$ 
given $K > 0$ is the size-biasing of $\Lambda$. 
The relationship between $\delta$ and $\Lambda$ is
easily deduced since $\P \{ K = 0 \}$ is the expected proportion of the real line
that is covered by the range of the subordinator associated with
$\mathcal{Z}$. Thus,
\[
\P \{ K = 0 \} = \frac{ \delta }{ \delta + \int_{\reals_+} y
  \Lambda(dy) } .
\]
\end{remark}

\begin{remark}
\label{R:Lambda_finite_discussion}
Theorem~\ref{thm:Zmeasure} and Remark~\ref{R:Zmeasure} provide
conditions for deciding whether the Lebesgue measure of the contact
set $\mathcal{Z}$ is zero almost surely or positive almost surely,
i.e. whether $\delta = 0$ or $\delta > 0$. In
Theorem~\ref{thm:newLambda} we provide conditions for deciding whether
$\Lambda(\reals_+) < \infty$ or $\Lambda(\reals_+) = \infty$ for the
case $\delta = 0$ and $\Pi(\reals) = \infty$. Since $\delta = 0$, these
conditions determine whether the contact set $\mathcal{Z}$ is a
discrete set or not.  On the way, we provide in
Propositions~\ref{prop:localExtrema_unbounded},
\ref{prop:localExtrema_bounded}, and \ref{prop:localExtrema_zero}
descriptions of the local behavior of a L\'evy process at its
local extrema.

Clearly, if $\sigma = 0$ and $\Pi(\reals) < \infty$, then
$\Lambda(\reals_+) < \infty$.  Consider
the remaining case $\sigma =0$,
$\Pi(\reals) = \infty$, and $\delta > 0$.  We claim that
$\Lambda(\reals_+) = \infty$. To see this, suppose to the contrary
that $\Lambda(\reals_+)<\infty$, then there almost surely exists $t_1
< t_2$ such that $X_t \wedge X_{t-} = M_t$ for all $t_1 < t < t_2$. 
Because $t \mapsto M_t$ is continuous, the paths of $X$ 
cannot jump between times $t_1$ and
$t_2$.  However, when $\Pi(\reals) = \infty$ the jump times
of $X$ are almost surely dense in $\reals$.
\end{remark}

Write
\begin{equation}
\label{E:def_minima}
\mathcal{M}^- := \bigcup_{\epsilon > 0} 
\{t \in \reals : X_t \wedge X_{t-} 
= \inf\{X_s : s \in (t-\epsilon,t+\epsilon)\}\}
\end{equation}
for the set of local infima of the path of $X$ and 
\begin{equation}
\label{E:def_maxima}
\mathcal{M}^+ := \bigcup_{\epsilon > 0} 
\{t \in \reals : X_t \wedge X_{t-} 
= \sup\{X_s : s \in (t-\epsilon,t+\epsilon)\}\}
\end{equation}
for the set of local suprema.  The following result is essentially due
to Vigon \cite{abrupt}. 

\begin{prop}
\label{prop:localExtrema_unbounded}
Let $X$ be a L\'evy process with paths of unbounded variation. Then,
$X_t = X_{t-}$ for all $t \in \mathcal{M}^-$ and all $t \in
\mathcal{M}^+$ almost surely. Moreover, for any $r>0$, $\liminf_{\eps
  \downarrow 0} \eps^{-1}(X_{t + \eps} - X_t) \geq r$ for all $t \in
\mathcal{M}^-$ almost surely if and only if
\begin{equation}
\label{eq:r_plus}
\int_0^1 t^{-1} \P\{X_t \in [0, rt]\} \, dt < \infty ,
\end{equation}
and $\limsup_{\eps \downarrow 0} \eps^{-1}(X_{t + \eps} - X_t)
\leq -r$  for all $t \in \mathcal{M}^+$ almost surely  if and only if
\begin{equation}
\label{eq:r_minus}
\int_0^1 t^{-1} \P\{X_t \in [-rt, 0]\} \, dt < \infty .
\end{equation}
\end{prop}
\begin{proof}
We show the equivalence involving local infima. The
equivalence involving local suprema then follows by a time reversal
argument.

Let $(X^q_t)_{t \ge 0}$ be a copy of $(X_t)_{t \ge 0}$
killed at an independent exponential time $\xi$ with parameter $0 < q <
\infty$. Define
\[
\rho := \arg \inf_{0 < t < \xi} X^q_t \wedge X^q_{t-} 
\quad \text{ and } \quad
\sigma := \arg \sup_{0 < t < \xi} X^q_t \wedge X^q_{t-} .
\]
By a localization argument such as the one indicated in the proof
of \cite[Theorem 1.3]{abrupt}, it is sufficient to show
\begin{equation}
\label{eq:global_equal}
X_\rho = X_{\rho-}
\end{equation}
and
\begin{equation}
\label{eq:global}
\liminf_{\eps \downarrow 0} \eps^{-1}(X_{\rho + \eps} - X_\rho) \geq r  
\text{ if and only if }
\eqref{eq:r_plus} \text{ holds}.
\end{equation}
Because $X$ has paths of unbounded variation, 
$\liminf_{t \downarrow 0} t^{-1} X_t = -\infty$ and
$\limsup_{t \downarrow 0} t^{-1} X_t = +\infty$
almost surely by \eqref{eq:Rogozin_small_time}, and hence zero is regular
for both $(-\infty,0]$ and $[0,\infty)$.  Equation \eqref{eq:global_equal}
then follows from \cite[Theorem $3.1$]{millarzeroone}.
The inequality \eqref{eq:global} 
is exactly \cite[Proposition $3.6$]{abrupt}.

\end{proof}

\begin{prop}
\label{prop:localExtrema_bounded}
Let $X$ be any L\'evy process with paths of bounded variation with
drift $d \neq 0$. Then, $X_t \neq X_{t-}$ for all $t \in \mathcal{M}^-$
and all $t \in \mathcal{M}^+$ almost surely. Moreover, $\lim_{\eps
  \downarrow 0} \eps^{-1}(X_{t + \eps} - X_t) = d$ for all $t \in
\mathcal{M}^-$ and all $t \in \mathcal{M}^+$ almost surely.
\end{prop}

\begin{proof}
Using the same notation and arguments as in 
Proposition~\ref{prop:localExtrema_unbounded}, it
suffices to show that $X_\rho \neq X_{\rho-}$ almost surely and that
\begin{equation}
\label{eq:global2}
\lim_{\eps \downarrow 0} \eps^{-1}(X_{\rho + \eps} - X_\rho) = d.
\end{equation}

A result of Millar states that any L\'evy
process for which zero is not regular for $[0,\infty)$ must jump out
of its global infimum and that any L\'evy process for which zero is
not regular for $(-\infty,0]$ must jump into its global infimum -- see
\cite[Theorem $3.1$]{millarzeroone}. By \eqref{eq:shtatland_statement},
$\lim_{t \downarrow 0} t^{-1} X_t = d$ almost surely, 
and so one of these alternatives must hold.
Hence, in either case, $X_\rho \neq X_{\rho-}$. 

Moreover, the fact that $\rho$ is a jump time of $X$
implies \eqref{eq:global2}. To see this, we argue as in
\cite{pitmanbravo}.
For $\delta>0$, let $0 < J_1^\delta < J_2^\delta < \ldots$ be the
successive nonnegative times at which $X$ has jumps of size greater
than $\delta$ in absolute value. The strong Markov property applied at
the stopping time $J_i^\delta$ and \eqref{eq:shtatland_statement} 
gives that
\[
\lim_{\eps \downarrow 0} \eps^{-1}
(X_{J_i^\delta+\eps}-X_{J_i^\delta}) = d.
\]
Hence, at any random time $V$ such that $X_{V} \ne X_{V-}$ almost
surely we have
\[
\lim_{\eps \downarrow 0} \eps^{-1} (X_{V+\eps}-X_{V}) = d .
\]
\end{proof}

\begin{prop}
\label{prop:localExtrema_zero}
Let $X$ be a L\'evy process with paths of bounded variation,
drift $d=0$, and  $\Pi(\reals) = \infty$. 
Then, $\lim_{\eps \downarrow 0} \eps^{-1}(X_{t + \eps} - X_t) = 0$ for all $t \in
\mathcal{M}^-$ and all $t \in \mathcal{M}^+$ almost surely. Moreover,
\begin{enumerate}[(i)]
\item If zero is not regular for $[0,\infty)$, then $X_t > X_{t-}$ for all $t \in
\mathcal{M}^-$ and $X_t = X_{t-}$ for all $t \in \mathcal{M}^+$ almost
surely.
\item If zero is not regular for $(-\infty,0]$, then $X_t = X_{t-}$ for all $t \in
\mathcal{M}^-$ and $X_t < X_{t-}$ for all $t \in \mathcal{M}^+$ almost
surely.
\item If zero is regular for both $(-\infty,0]$ and $[0,\infty)$,
then $X_t = X_{t-}$ for all $t \in \mathcal{M}^-$ and all $t \in \mathcal{M}^+$ almost
surely.
\end{enumerate}
\end{prop}
\begin{proof}
Note that since $\Pi(\reals) = \infty$, zero must be regular for at
least one of $(-\infty,0]$ and $[0,\infty)$. Results (i), (ii) and
(iii) are direct consequences of \cite[Theorem 3.1]{millarzeroone}.

Arguing as in Proposition \ref{prop:localExtrema_bounded}, we get that
$\lim_{\eps \downarrow 0} \eps^{-1}(X_{t + \eps} - X_t) = 0$ for any
time $t$ such that $X_t \neq X_{t-}$. Using the notation of
Propositions~\ref{prop:localExtrema_unbounded}
and \ref{prop:localExtrema_bounded},
suppose that $X_\rho = X_{\rho-}$  and
$\liminf_{\eps \downarrow 0} \eps^{-1}(X_{\rho + \eps} - X_\rho)>
\gamma > 0$. Then, for all $0 < \gamma' < \gamma$ the time $\rho$ is
the time of a local infimum of the process $(X_t - \gamma' t)_{t \ge
  0}$. Since the drift coefficient of this modified process is less
than zero for all such $\gamma'$, the path of $X$ must jump at time
$\rho$, which is a contradiction. Hence, $\liminf_{\eps \downarrow 0}
\eps^{-1}(X_{\rho + \eps} - X_\rho) = 0$.
\end{proof}

\begin{thm}
\label{thm:newLambda}
Let $X$ be a L\'evy process that satisfies our standing assumptions
Hypothesis~\ref{H:standing} and $\Pi(\reals) = \infty$.
Then, $\Lambda(\reals_+) < \infty$ if and only if
\begin{equation}
\label{eq:finint}
\int_0^1 t^{-1} \P\{X_t \in [-\alpha t, \alpha t]\} \, dt < \infty .
\end{equation}
\end{thm}
\begin{proof}
We break the proof up into the consideration of a number of cases.
All of the cases except the first
rely on the facts that $D$ is the time of a
local infimum of the process $(X_t + \alpha t)_{t \ge 0}$ and that $G$
is the time of a local infimum of the process $(X_{-t} + \alpha t)_{t \ge 0}$.

\medskip\noindent
{\bf Case 1:} 
{\em The process $X$ has paths of bounded variation almost surely 
and $|d| < \alpha$.} 

Suppose that $0 \le d < \alpha$.  By \eqref{eq:shtatland_statement},
$\lim_{t \downarrow 0} t^{-1} X_t = d$.  Therefore,
zero is regular for $(-\infty,0]$ for the modified process
$(X_t - \alpha t)_{t \ge 0}$ but not for the modified process
$(X_t + \alpha t)_{t \ge 0}$.
Rogozin's regularity criterion \eqref{eq:rogozin_regularity} gives that
\[
\int_0^1 t^{-1} \P\{X_t - \alpha t \le 0\} \, dt = \infty
\]
and
\[
\int_0^1 t^{-1} \P\{X_t + \alpha t \le 0\} \, dt < \infty.
\]
Hence,
\[
\begin{split}
& \int_0^1 t^{-1} \P\{X_t \in [-\alpha t, \alpha t] \, dt \\
& \quad =
\int_0^1 t^{-1} \P\{X_t \le \alpha t \} \, dt
-
\int_0^1 t^{-1} \P\{X_t \le -\alpha t \} \, dt \\
& \quad = \infty,
\end{split}
\]
and the inequality \eqref{eq:finint} fails.
Note by Theorem~\ref{thm:Zmeasure}
that $\delta > 0$, and hence $\Lambda(\reals_+) = \infty$
-- see Remark~\ref{R:Lambda_finite_discussion}.
The proof for $-\alpha <  d \le 0$ is similar.

\medskip\noindent
{\bf Case 2:} 
{\em The process $X$ has paths of bounded variation almost surely 
and  $|d| > \alpha$.}

Suppose that $d > \alpha$.  By \eqref{eq:shtatland_statement},
$\lim_{t \downarrow 0} t^{-1} X_t = d$, and so zero is not
regular for $(-\infty,0]$ for the modified process
$(X_t - \alpha t)_{t \ge 0}$.  It follows from
\eqref{eq:rogozin_regularity} that
\[
\int_0^1 t^{-1} \P\{X_t - \alpha t \le 0\} \, dt < \infty. 
\]
Hence, 
\[
\int_0^1 t^{-1} \P\{X_t \le \alpha t \} \, dt < \infty 
\]
and the
inequality \eqref{eq:finint} certainly holds. 
It follows from
Proposition \ref{prop:localExtrema_bounded} that 
$\liminf_{\eps \downarrow 0} \eps^{-1}((X_{D +  \eps}+\alpha \eps)-X_D) > 2 \alpha$ a.s. 
Thus, if
\[
D' : = \inf \{ t > D : t \in \mathcal{Z} \}, 
\]
then $D'>D$ a.s. The regenerative property of
$\mathcal{Z}$ implies that $\mathcal{Z}$ is discrete, 
and hence $\Lambda(\reals_+)<\infty$.  The proof for
$d<-\alpha$ is similar.
%


\medskip\noindent
{\bf Case 3:}
{\em The process $X$ has paths of bounded variation almost surely,
$d=-\alpha$, and zero is not regular for $(-\infty,0]$ for the
modified process $(X_t + \alpha t )_{t \ge 0}$.}

By \eqref{eq:shtatland_statement},
$\lim_{\epsilon \downarrow 0}\eps^{-1} (X_t - \alpha t) = -2 \alpha$,
and so zero is not regular for $[0,\infty)$ for the
modified process $(X_t - \alpha t )_{t \ge 0}$.
It follows from \eqref{eq:rogozin_regularity} that
\[
\int_0^1 t^{-1} \P\{X_t + \alpha t \le 0\} \, dt < \infty 
\]
and
\[
\int_0^1 t^{-1} \P\{X_t - \alpha t \ge 0\} \, dt < \infty. 
\]
Hence,
\[
\begin{split}
& \int_0^1 t^{-1} \P\{X_t \in [-\alpha t, \alpha t] \, dt \\
& \quad =
\int_0^t t^{-1} \, dt
- \int_0^1 t^{-1} \P\{X_t \le -\alpha t \} \, dt
- \int_0^1 t^{-1} \P\{X_t \ge \alpha t \} \, dt \\
& \quad = \infty,
\end{split}
\]
and the inequality \eqref{eq:finint} fails.
Theorem~\ref{thm:Zmeasure}
implies that $\delta > 0$ and hence $\Lambda(\reals_+) =
\infty$ -- see Remark~\ref{R:Lambda_finite_discussion}.

\medskip\noindent
{\bf Case 4:}
{\em The process $X$ has paths of bounded variation almost surely,
$d=-\alpha$, and zero is regular for both $(-\infty,0]$ 
and $[0,\infty)$ for the
modified process $(X_t + \alpha t )_{t \ge 0}$.}

As in Case 3,
\[
\int_0^1 t^{-1} \P\{X_t - \alpha t \ge 0\} \, dt < \infty. 
\]
Also,
\[
\int_0^1 t^{-1} \P\{X_t + \alpha t \le 0\} \, dt = \infty 
\]
and
\[
\int_0^1 t^{-1} \P\{X_t + \alpha t \ge 0\} \, dt = \infty. 
\]
Hence,
\[
\begin{split}
& \int_0^1 t^{-1} \P\{X_t \in [-\alpha t, \alpha t] \, dt \\
& \quad =
\int_0^1 t^{-1} \P\{X_t \ge - \alpha t \} \, dt
-
\int_0^1 t^{-1} \P\{X_t \ge  \alpha t \} \, dt \\
& \quad = \infty,
\end{split}
\]
and inequality \eqref{eq:finint} fails.
Note that $\delta = 0$
by Theorem~\ref{thm:Zmeasure}, so if $\Lambda(\reals_+) < \infty$,
then $M_t = X_D + \alpha (t-D)$ for $0 \le t \le \eps$ for
some $\eps > 0$, but
$\liminf_{\eps \downarrow 0} \eps^{-1}((X_{D +  \eps}+\alpha \eps)-X_D)  = 0$ a.s.
by Proposition \ref{prop:localExtrema_zero}. 
Thus, we must have $\Lambda(\reals_+) = \infty$.

\medskip\noindent
{\bf Case 5:}
{\em The process $X$ has paths of bounded variation almost surely,
$d=-\alpha$, zero is regular for $(-\infty,0]$ 
and not regular for $[0,\infty)$ for the modified
process $(X_t + \alpha t )_{t \ge 0}$.}

Similarly to Case 4,
\[
\int_0^1 t^{-1} \P\{X_t - \alpha t \ge 0\} \, dt < \infty, 
\]
\[
\int_0^1 t^{-1} \P\{X_t + \alpha t \le 0\} \, dt = \infty, 
\]
and
\[
\int_0^1 t^{-1} \P\{X_t + \alpha t \ge 0\} \, dt < \infty. 
\]
In particular,
\[
\int_0^1 t^{-1} \P\{X_t \ge - \alpha t\} \, dt < \infty, 
\] 
and the inequality \eqref{eq:finint} holds.
Proposition \ref{prop:localExtrema_zero} gives that
$X_D > X_{D-}$ a.s. Thus, $D'>D$ a.s. The regenerative property of
$\mathcal{Z}$ implies that $\mathcal{Z}$ is discrete, 
and hence $\Lambda(\reals_+)<\infty$.

\medskip\noindent
{\bf Case 6:}
{\em The process $X$ has paths of bounded variation almost surely and
$d=\alpha$.}

This is handled by considering the behavior at $G$ for
the time reversed process in the manner of Cases 3,4, and 5.

\medskip\noindent
{\bf Case 7:}
{\em The process $X$ has paths of unbounded variation almost surely
and
$\int_0^1 t^{-1} \P\{X_t \in [-\alpha t, (\alpha+\gamma) t]\} \, dt <
\infty$ for some $\gamma > 0$.}

Proposition
\ref{prop:localExtrema_unbounded} gives that
$\liminf_{\eps \downarrow 0} \eps^{-1}((X_{D + \eps}+\alpha \eps)-X_D)
\geq 2 \alpha + \gamma$ a.s., which, as in Case 2, implies that $D' > D$
a.s. and hence $\Lambda(\reals_+) < \infty$.

\medskip\noindent
{\bf Case 8:}
{\em The process $X$ has paths of unbounded
variation, \eqref{eq:finint} holds, but 
$\int_0^1 t^{-1} \P\{X_t \in [-\alpha t, (\alpha+\gamma) t]\} \, dt =
\infty$ for every $\gamma > 0$.}

To get to the desired result that
$\Lambda(\reals_+) < \infty$ we introduce a new technique involving
convex minorants of L\'evy processes. 

As before, let $(X^q_t)_{t \ge 0}$ be a copy of $(X_t)_{t \ge 0}$ but
killed at an independent exponential time $\xi$ with parameter $0 < q <
\infty$, and let $\rho = \arg \inf_{0 < t < \xi} X_t \wedge X_{t-}$. 

By results of Pitman and Uribe Bravo on the convex minorant of the path of a
L\'evy process \cite[Corollary 2]{pitmanbravo}, the linear segments
of the convex minorant of the process $(X_t + \alpha t)_{0 \leq t \leq \xi}$ 
define a Poisson point process on $\reals_+ \times \reals$,
where a point at $(t, x)$ represents a segment with length $t$ and
increment $x$. The intensity measure of the Poisson point process is
the measure on $\reals_+ \times \reals$ given by
\begin{equation}
\label{eq:intensity_expression}
e^{-qt} t^{-1} \P \{ X_t + \alpha t \in dx \} dt.
\end{equation}
In order to recreate the convex minorant from the point process, the
segments are arranged in increasing order of slope. Note that the
convex minorant after time $\rho$ can be recreated by only piecing
together the segments of positive slope.

Let $\mathcal{I}$ be the infimum of the slopes of all segments of
the convex minorant of $(X_t + \alpha t)_{0 \leq t \leq \xi}$ that have positive
slope. Under the assumption \eqref{eq:finint}, it follows from
\eqref{eq:intensity_expression} that
\[
\P \{ \mathcal{I} \geq 2 \alpha \} = \exp \left( - \int_0^\infty
  e^{-qt} t^{-1} \P\{ X_t + \alpha t \in [0, 2 \alpha t] \} \, dt \right) > 0 .
\]
Thus, with positive probability, there exists $\eps>0$ such that
$(X_{\rho + t} + \alpha t)-X_\rho \geq 2 \alpha t$ for all $0 \leq t
\leq \eps$.  Hence by Millar's zero-one law at the infimum of a L\'evy
process, such an $\eps$ exists almost surely.  By the almost sure uniqueness
of the value of the infimum of a L\'evy process that is not a compound
Poisson process with zero drift \cite[Proposition VI.4]{bertoin},
almost surely there exists $\eps>0$ such that $(X_{\rho + t} + \alpha t)-X_\rho
> 2 \alpha t$ for all $0 < t \leq \eps$.

Using the same localization argument as before, this behavior extends
to all local infima almost surely, and hence is true at time $D$ for
the process $(X_t + \alpha t)_{t \ge 0}$, which allows us to conclude that
$D' > D$ a.s. Since $\delta = 0$, this implies $\Lambda(\reals_+) < \infty$.
\end{proof}

\begin{remark}
  An example of a process satisfying our standing assumptions for
  which \eqref{eq:finint} fails for all $\alpha >0$ is given by truncating
  the L\'evy measure of the symmetric Cauchy process to remove all
  jumps with magnitude greater than $m$, so that the L\'evy measure
  becomes $1_{|x| \le m}x^{-2} \, dx$. To see this, first note that
  \eqref{eq:finint} fails for the symmetric Cauchy process because, by
  the self-similarity properties of this process, the probability that
  it lies in an interval of the form $(at, bt)$ at time $t > 0$ does
  not depend on $t$ and $\int_0^1 t^{-1} \, dt = \infty$.  Then
  observe that the difference between the probabilities that the
  truncated process and the symmetric Cauchy processes lie in some
  interval at time $t$ is at most the probability that the symmetric
  Cauchy process has at least one jump of size greater than $m$ before
  time $t$.  The latter probability is $1-e^{-\lambda t}$ with
  $\lambda = 2 \int_m^\infty x^{-2} \, dx < \infty$, and $\int_0^1
  t^{-1}(1-e^{-\lambda t}) \, dt < \infty$.
\end{remark}

\begin{remark}
If $X$ is a symmetric $\beta$-stable process for $1 < \beta \le 2$,
then \eqref{eq:finint} holds for all $\alpha >0$.  To see this,
first note that $X_1$ has a bounded density. Hence, by scaling,
$\P\{X_t \in [-\alpha t, \alpha t]\} = \P\{t^{1/\beta} X_1 \in
[-\alpha t, \alpha t]\} \le c t^{1-1/\beta}$
for some constant $c$ depending on $\alpha$, and then observe
that $\int_0^1 t^{-1} t^{1-1/\beta} \, dt < \infty$. Further cases
where \eqref{eq:finint} holds for all $\alpha >0$ are discussed in
Remark~\ref{rem:abrupt}.
\end{remark}

The technique introduced at the end of the previous proof allows a
strengthening of Propositions~\ref{prop:localExtrema_unbounded}
and~\ref{prop:localExtrema_bounded}, which we state only for local
infima but has a clear counterpart for local suprema.  The result also
covers much of the information from Proposition
\ref{prop:localExtrema_zero} but does not strengthen it.

\begin{thm}
\label{thm:localExtrema}
Let $X$ be an L\'evy process such that $\sigma \neq 0$ or
$\Pi(\reals) < \infty$.
Define
\[
r^* := \sup \{ r \geq 0 : \int_0^1 t^{-1} \P\{X_t \in [0, rt]) dt < \infty \}
\]
Then,
\[
\liminf_{\eps \downarrow 0} \eps^{-1}(X_{t + \eps} - X_t \wedge X_{t-}) = r^*
\]
for all $t \in \mathcal{M}^-$ almost surely. Moreover, define, for $r 
\geq 0$ and $t \geq 0$,
\[
T_t^{(r)} := \inf \{ s > 0 :  X_{t + s} - X_t \wedge X_{t-} \leq r s \} .
\]
If $\int_0^1 t^{-1} \P\{X_t \in [0, r^*t]\} \, dt < \infty$, then $T_t^{(r^*)} > 0$ for all $t \in \mathcal{M}^-$ almost surely.
\end{thm}

\begin{proof}
As usual, we need only show that the given properties hold at time
$\rho = \arg \inf_{0 < t < \xi} X^q_t \wedge X^q_{t-}$. 
Suppose first that $\int_0^1 t^{-1} \P\{X_t \in [0, r^*t]\} \, dt < \infty$. 
Let $\mathcal{I}$ be the infimum of the slopes of all segments of
the convex minorant of $(X_t)_{0 \leq t \leq \xi}$ that have positive
slope. It follows from \eqref{eq:intensity_expression} that
\[
\P \{ \mathcal{I} \geq r^* \} = \exp \left( - \int_0^\infty
  e^{-qt} t^{-1} \P\{ X_t \in [0, r^* t] \} \, dt \right) > 0 .
\]
Thus, with positive probability, there exists $\eps>0$ such that
$X_{\rho + t}-X_\rho \geq r^* t$ for all $0 \leq t
\leq \eps$.  Hence, by Millar's zero-one law at the infimum of a L\'evy
process, such an $\eps$ exists almost surely.  By the almost sure uniqueness
of the value of the infima of a L\'evy process that is not a compound
Poisson process with zero drift \cite[Proposition VI.4]{bertoin},
there exists almost surely $\eps>0$ such that $X_{\rho+t}-X_\rho
> r^* t$ for all $0 < t \leq \eps$. Hence, $T_\rho^{(r^*)} > 0$ a.s.

For any $0 \leq r < r^*$  we have 
$\int_0^1 t^{-1} \P\{X_t \in [0, rt]\} \, dt < \infty$.
Applying the above argument gives that $T^{(r)}_\rho = 0$ almost
surely, and thus
\[
\liminf_{\eps \downarrow 0} \eps^{-1}(X_{t + \eps} - X_t \wedge
X_{t-}) \geq r
\]
for all $t \in \mathcal{M}^-$ almost surely, for all $0 \leq r < r^*$.

For any $r > r^*$  we have 
$\int_0^1 t^{-1} \P\{X_t \in [0, rt]\} \, dt = \infty$, and hence $\P\{
\mathcal{I} \geq r\} = 0$. Since the convex minorant of $(X_t)_{0 \le t
  \le \xi}$ almost surely contains linear segments with positive slope
less that or equal to $r$, it follows that $T_t^{(r)} = 0$ almost
surely. Hence, for all $r > r^*$,
\[
\liminf_{\eps \downarrow 0} \eps^{-1}(X_{t + \eps} - X_t \wedge
X_{t-}) \leq r
\]
for all $t \in \mathcal{M}^-$ almost surely.
\end{proof}

\begin{remark}
\label{rem:howtotest}
The value of $r^*$ is infinite when $X$ has non-zero Brownian
component or is a stable process with stability parameter in the
interval $(1,2]$ -- see the discussion of \emph{abrupt} processes in
Section~\ref{sec:limit}. Vigon provides in an unpublished work
\cite{vigon_unpublished} a 
practical method for determining whether the integral in
Theorem~\ref{thm:localExtrema} is finite or not for processes with paths of
unbounded variation:
\[
\int_0^\infty t^{-1} e^{-qt} \P\{X_t \in [at, bt]\} \, dt = \frac{1}{2 \pi}
\int_a^b \left( \int_0^\infty \mathfrak{R} \frac{1}{\Psi(-u) + iur} \,
  du \right) \, dr,
\]
where $\Psi$ is as defined in Section~\ref{sec:defs}.
\end{remark}


In Section~\ref{sec:proofs} we prove the following result, which
characterizes the subordinator associated with $\mathcal{Z}$ when $X$
has paths of unbounded variation and satisfies certain extra
conditions. 
In Corollary~\ref{cor:boundeddensities} we show that
these extra conditions hold
 when $X$ has non-zero Brownian component.
Note that the conclusion $\delta = 0$ in the result
follows from Remark~\ref{R:Zmeasure}(i).

\begin{thm}
\label{thm:bigthm}
Let $X$ be a L\'evy process with paths of unbounded variation almost
surely that satisfies our standing assumptions
Hypothesis~\ref{H:standing}. Suppose further that $X_t$ has absolutely
continuous distribution for all $t \neq 0$, and that 
the densities of the
random variables $\inf_{t \ge 0} \{ X_t + \alpha t \}$ and $\inf_{t
\ge 0} \{ X_{-t} + \alpha t \}$ are square integrable.
Then, $\delta = 0$ and $\Lambda$ is characterized by
\[
\begin{split}
& \frac{\int_{\reals_+} (1-e^{-\theta x}) \, \Lambda( dx )} 
{\int_{\reals_+} x \, \Lambda( dx )}\\
& \quad = 4 \pi \alpha 
\int_{-\infty}^\infty
\Biggl\{
\exp \biggl(
 \int_0^\infty 
t^{-1} 
\E\Bigl[
\left(e^{i z X_t - i z \alpha t}   - 1\right)
\mathbf{1}\{X_t \ge + \alpha t\} \\
& \quad \qquad +
\left(e^{i z X_t + i z \alpha t}   - 1\right)
\mathbf{1}\{X_t \le - \alpha t\}
\Bigr]
\, dt 
\biggr) \\
& \qquad - 
\exp \biggl(
 \int_0^\infty 
t^{-1} 
\E\Bigl[
\left(e^{- \theta t + i z X_t - i z \alpha t}   - 1\right)
\mathbf{1}\{X_t \ge + \alpha t\} \\
& \qquad \qquad +
\left(e^{- \theta t + i z X_t + i z \alpha t}   - 1\right)
\mathbf{1}\{X_t \le - \alpha t\}
\Bigr]
\, dt 
\biggr)
\Biggr\}
\, dz \\
\end{split}
\]
for $\theta \ge 0$, and, moreover, $\Lambda(\reals_+) < \infty$.
\end{thm}

Note that the existence of the densities of the infima in the
hypothesis (ii) of
Theorem~\ref{thm:bigthm} comes from the assumption that $X_t$ has absolutely
continuous distribution for all $t \neq 0$ -- see
Lemma~\ref{lem:abscont}.

When the conditions of Theorem~\ref{thm:bigthm} are not satisfied, we
are able to give a characterization of $\Lambda$ as a limit of
integrals in the following way. Let $X^\eps = X + \eps B$, with $B$ a
(two-sided) standard Brownian motion independent of $X$, and let
$\Lambda^\eps$ be the L\'evy measure of the subordinator associated
with the contact set for $X^\eps$. Then in the case $\delta = 0$ we
have the representation
\[
\frac{\int_{\reals_+} (1-e^{-\theta x}) \, \Lambda( dx )}
{\int_{\reals_+} x \, \Lambda(dx)}
=
\lim_{\eps \downarrow 0} 
\frac{\int_{\reals_+} (1-e^{-\theta x}) \, \Lambda^\eps( dx )}
{\int_{\reals_+} x \, \Lambda^\eps(dx)} .
\]
See Lemma \ref{lem:continuousmapping} in Section
\ref{sec:proofs} for details of this limit and
\eqref{eq:limit_representation} for a proof of the above equality.

Theorem~\ref{thm:newLambda} together with the conclusion $\Lambda(\reals_+) < \infty$ of
Theorem~\ref{thm:bigthm} result in the following.
\begin{cor}
Suppose the conditions of Theorem~\ref{thm:bigthm} are satisfied, then
\[
\int_0^1 t^{-1} \P\{X_t \in [-\alpha t, \alpha t]\} \, dt < \infty.
\]
\end{cor}

\section{The limit of the contact set for increasing slopes}
\label{sec:limit}

We now investigate how $\mathcal{Z}$ changes as $\alpha$ increases. 
For the sake of clarity, let $X$ be a fixed L\'evy process with $X_0=0$ and
$\E[|X_1|] < \infty$. Write $M^{(\alpha)} = (M^{(\alpha)}_t)_{t \in
  \reals}$ for the $\alpha$-Lipschitz minorant of $X$ for 
$\alpha > |\E[X_1]|$, and put
$\mathcal{Z}_\alpha 
:= 
\{ t \in \reals : X_t \wedge X_{t-} = M^{(\alpha)}_t  \}$. 
For $|\E[X_1]| < \alpha' \leq \alpha''$, 
we have $M^{(\alpha')}_t \le M^{(\alpha'')}_t \le X_t$
for all $t \in \reals$ (because any $\alpha'$-Lipschitz function
is also $\alpha''$-Lipschitz), and so $\mathcal{Z}_{\alpha'} \subseteq
\mathcal{Z}_{\alpha''}$. We note in passing that $\mathcal{Z}_{\alpha'}$ is
\emph{regeneratively embedded} in $\mathcal{Z}_{\alpha''}$ in the
sense of Bertoin \cite{bertoin_regen}. 

If $X$ has paths of bounded variation and drift coefficient $d$, then
$|d| < \alpha$ for all $\alpha$ large enough. Since $\lim_{t
  \downarrow 0} t^{-1} X_t = - \lim_{t \downarrow 0} t^{-1} X_{-t} =
d$, the law of large numbers implies that
\[
\lim_{\alpha \ra \infty} \P \{ 0 \in \mathcal{Z}_\alpha \}
=
\lim_{\alpha \ra \infty} \P \{ \inf_{t \ge 0} (X_t + \alpha t) =
\inf_{t \le 0} (X_t - \alpha t) = 0 \} 
= 
1,
\]
and thus the set $\bigcup_{\alpha > | \E [X_1] | } Z_\alpha$ has full
Lebesgue measure.

We now consider the case where $X$ has paths of unbounded variation.
In order to state our result, we need to recall the definition of the
so-called \emph{abrupt} L\'evy processes introduced by Vigon
\cite{abrupt}.  Recall from \eqref{E:def_minima} that $\mathcal{M}^-$
is the set of local infima of the path of $X$, and that as noted in
\cite{abrupt}, if the paths of $X$ have unbounded variation, then
almost surely $X_{t-} = X_t$ for all $t \in \mathcal{M}^-$.

\begin{defn}
\label{def:abrupt}
A L\'evy process $X$ is {\em abrupt} if its paths have unbounded variation 
and almost surely for all $t \in \mathcal{M}^-$
\[
\limsup_{\eps \uparrow 0} \frac{X_{t+\eps}-X_{t-}}{\eps} = -\infty 
\and
\liminf_{\eps \downarrow 0} \frac{X_{t+\eps}-X_t}{\eps} = +\infty.
\]
\end{defn}

\begin{remark}
\label{R:abrupt}
An equivalent definition may be made in terms of local
suprema \cite[Remark 1.2]{abrupt}: a L\'evy process $X$ 
with paths of unbounded variation is abrupt if almost surely for any $t$ 
that is the time of a local supremum,
\[
\liminf_{\eps \uparrow 0} \frac{X_{t+\eps}-X_{t-}}{\eps} = +\infty 
\and
\limsup_{\eps \downarrow 0} \frac{X_{t+\eps}-X_t}{\eps} = -\infty.
\]
\end{remark}

\begin{remark}
\label{rem:abrupt}
A L\'evy process $X$ with paths of unbounded variation is abrupt
 if and only if 
\begin{equation}
\label{eq:abruptcondition}
 \int_0^1 t^{-1} \P\{ X_t \in [at,bt]\}  \, dt < \infty,
 \quad \forall a<b,
\end{equation}
(see \cite[Theorem 1.3]{abrupt}).
Examples of abrupt L\'evy processes include stable processes 
with stability parameter in the interval $(1,2]$, processes with
non-zero Brownian component, and any processes that creep upwards or
downwards. An example of an unbounded variation process that is not
abrupt is the symmetric Cauchy process. 
\end{remark}

\begin{remark}
The analytic condition given in Remark~\ref{rem:abrupt}
\eqref{eq:abruptcondition} for a L\'evy process $X$ to be abrupt has
an interpretation in terms of the smoothness of the convex
minorant of $X$ over a finite interval. The results of Pitman and
Uribe Bravo \cite{pitmanbravo} imply that the number of segments of
the convex minorant of $X$ over a finite interval with slope between
$a$ and $b$ is finite for all $a < b$ if and only if
\eqref{eq:abruptcondition} holds.
\end{remark}

We now return to the question of the limit of $\mathcal{Z}_\alpha$.

\begin{thm}
\label{thm:Zlimit}
Let $X$ be a L\'evy process with $X_0=0$ and $|\E[X_1]| <
\infty$. Then
$ \bigcup_{\alpha > | \E [X_1] |} \mathcal{Z}_\alpha \supseteq
\mathcal{M}^-$. Furthermore, if $X$ is abrupt, then 
$ \bigcup_{\alpha > | \E [X_1] |} \mathcal{Z}_\alpha =  \mathcal{M}^-$.
\end{thm}

\begin{proof}
Suppose that $t \in \mathcal{M}^-$ so that there
exists $\epsilon > 0$ such that 
$\inf\{X_s : t-\epsilon < s < t+\epsilon\}  = X_t = X_{t-}$. 
Fix any $\beta > | \E [X_1] |$. Then, by the strong law of large numbers,
$\inf\{ X_s + \beta s : s \ge 0\} > - \infty$ and 
$\inf\{ X_s - \beta s : s \le 0 \} > - \infty$.
It is clear that if
$\alpha \in \reals$ is such that
\[
\alpha > - \frac{  
\inf\{ X_s + \beta s : s \ge 0 \}  
\, \vee \, 
\inf\{ X_s - \beta s : s \le 0 \},
}{\epsilon}
\] 
then $X_t = X_{t-} = M^{(\alpha)}_t$ and  $t \in
\mathcal{Z}_\alpha$. Hence $ \bigcup_{\alpha > | \E [X_1] |}
\mathcal{Z}_\alpha \supseteq \mathcal{M}^-$.

Now suppose that $X$ is abrupt, and let $t \in \mathcal{Z}_\alpha$ for
some $\alpha > | \E [X_1] |$.  Then, one of the
following three possibilities must occur:
\begin{enumerate}[(a)]
\item $X_t > X_{t-}$ and $\limsup_{\eps \uparrow 0} \eps^{-1} (X_{t+\eps}-X_{t-}) \leq
  \alpha$;
\item $X_{t-} > X_t$ and $\liminf_{\eps \downarrow 0} \eps^{-1} (X_{t+\eps}-X_t) \geq -
  \alpha$;
\item $X_{t-} = X_t$ and $\limsup_{\eps \uparrow 0} \eps^{-1} (X_{t+\eps}-X_{t-}) \leq
  \alpha$, $\liminf_{\eps \downarrow 0} \eps^{-1} (X_{t+\eps}-X_t)
  \geq -\alpha$.
\end{enumerate}
We discount options (a) and (b) by assuming that $t$ is a jump
time of $X$ and then showing that the $\liminf$ or $\limsup$ part of
the statements cannot occur. Our argument borrows heavily from the
proof of Property 2 in \cite[Proposition 1]{pitmanbravo}, which
itself is based on the proof of 
\cite[Proposition 2.4]{millarzeroone}, but is more detailed.

Arguing as in the proof of Proposition~\ref{prop:localExtrema_bounded}, for
$\delta>0$, let $0 < J_1^\delta < J_2^\delta < \ldots$ be the
successive nonnegative times at which $X$ has jumps of size greater
than $\delta$ in absolute value. The strong Markov property applied at
the stopping time $J_i^\delta$ and \eqref{eq:Rogozin_small_time} 
gives that
\[
\liminf_{\eps \downarrow 0} \eps^{-1}
(X_{J_i^\delta+\eps}-X_{J_i^\delta}) =-\infty \and \limsup_{\eps
  \downarrow 0} \eps^{-1} (X_{J_i^\delta+\eps}-X_{J_i^\delta}) =
+\infty.
\]
Hence, at any random time $V$ such that $X_{V} \ne X_{V-}$ almost
surely we have
\[
\liminf_{\eps \downarrow 0} \eps^{-1} (X_{V+\eps}-X_{V}) = -\infty,
\]
and, by a time reversal,
\[
\limsup_{\eps \uparrow 0} \eps^{-1} (X_{V+\eps}-X_{V-}) = + \infty.
\]
Thus, neither of the possibilities (a) or (b) hold, and so (c) must
hold. It then follows from Theorem~\ref{thm:allt} below that $X$ must
have a local infimum or supremum at $t$.  However, $X$ cannot have a
local supremum  at $t$ by Remark~\ref{R:abrupt}, and so $X$ must have a
local infimum at $t$.
\end{proof}

The key to proving Theorem~\ref{thm:Zlimit} in the abrupt case was the
following theorem that describes the local behavior of an abrupt
L\'evy process at arbitrary times.  This result is an immediate
corollary of \cite[Theorem 2.6]{abrupt} once we use the fact that
almost surely the paths of a L\'evy processes cannot have both points
of increase and points of decrease \cite{fourati}.

\begin{thm}
\label{thm:allt}
Let $X$ be an abrupt L\'evy process. Then,
almost surely for all $t$ one of the following possibilities must
hold:
\begin{enumerate}[(i)]
  \item $\limsup_{\eps \uparrow 0} \eps^{-1} (X_{t+\eps}-X_{t-}) = +\infty$
    and $\liminf_{\eps \downarrow 0} \eps^{-1} (X_{t+\eps}-X_t) = -\infty $;
  \item $\limsup_{\eps \uparrow 0} \eps^{-1} (X_{t+\eps}-X_{t-}) < +\infty$
    and $\liminf_{\eps \downarrow 0} \eps^{-1} (X_{t+\eps}-X_t) = -\infty $;
  \item $\limsup_{\eps \uparrow 0} \eps^{-1} (X_{t+\eps}-X_{t-}) = +\infty$
    and $\liminf_{\eps \downarrow 0} \eps^{-1} (X_{t+\eps}-X_t) > -\infty $;
  \item $X$ has a local infimum or supremum at $t$.
\end{enumerate}
\end{thm}

\begin{remark}
  Theorem \ref{thm:Zlimit}  shows that the
  $\alpha$-Lipschitz minorant provides a method for ``sieving out''
  a certain discrete set of times of
  local infima of an abrupt process.  This method 
  has the property that if we let $\alpha \ra \infty$, then eventually
  we collect all the times of local infima. 
  Alternative methods for sieving out the local minima
  of Brownian motion are discussed in
  \cite{pitmanneveu1,pitmanneveu2}. One method is to take all
  local infima times $t$ such that $X_{t+s} - X_t > 0$ for all $s \in (
  -h, h)$ for some fixed $h$, and then let $h \ra 0$. 
  Another is to take all local infima times $t$ such that $X_{s_+} -
  X_t \ge h$ for some time $s_+ \in (0, \inf \{ s > 0 : X_s - X_t = 0
  \})$ and such that $X_{s_-} - X_t \ge h$ for some time $s_- \in (0,
  \inf \{ s < 0 : X_s - X_t = 0 \})$, and then again let $h \ra
  0$. This work is extended to Brownian motion with
  drift in \cite{faggionato}.
\end{remark}

\section{Future infimum of a L\'evy process}
\label{sec:prelim}

For future use, we collect together in this section some preliminary
results concerning the distribution of the infimum of a L\'evy process
$(Z_t)_{t \ge 0}$ and the time at which the infimum is attained.

Let $Z = (Z_t)_{t \ge 0}$ be a L\'evy process such that 
$Z_0 = 0$. 
Set $\underline{Z}_t := \inf\{Z_s : 0 \le s \le t\}$, $t \ge 0$.
If $Z$ is not a compound Poisson process
(that is, either $Z$ has a non-zero Brownian component or the
L\'evy measure of $Z$ has infinite total mass or the L\'evy measure
has finite total mass but there is a non-zero drift coefficient),
then 
\begin{equation}
\label{E:minima_distinct}
\P\{\exists 0 \le s < t < u : 
\underline{Z}_s = \underline{Z}_t = Z_t \wedge Z_{t-} = \underline{Z}_u\} = 0
\end{equation}
-- see, for example, \cite[Proposition VI.4]{bertoin}.
Hence, almost surely for each $t \ge 0$
there is a unique time 
$U_t$ such that $Z_{U_t} \wedge Z_{U_t-} =
\underline{Z}_t$.
If, in addition,
$\lim_{t \ra \infty} Z_t = +\infty$,
then almost surely there is a unique time 
$U_\infty$ such that $Z_{U_\infty} \wedge Z_{U_\infty-} =
\underline{Z}_\infty := \inf\{Z_s : s \ge 0\}$.

\begin{lem}
\label{lem:abscont}
Let $Z$ be a L\'evy process such that $Z_0 = 0$,
$Z_t$ has an absolutely
continuous distribution for each $t>0$, and  $\lim_{t \ra
  \infty} Z_t = +\infty$.  Then, the distribution of $(U_\infty,\underline{Z}_\infty)$ 
restricted to $(0,\infty) \times (-\infty,0]$ is
absolutely continuous with respect to Lebesgue measure. 
Moreover, $\P\{(U_\infty, \underline{Z}_\infty) = (0,0)\} > 0$ 
if and only if 
zero is not regular for $(-\infty,0)$.
\end{lem}

\begin{proof}
Because the random variable
$Z_t$ has an absolutely continuous distribution
for each $t>0$, it follows from
\cite[Theorem 2]{pitmanbravo} 
that for all $t > 0$ the restriction of the
distribution of the random vector $(U_t,\underline{Z}_t)$ is
absolutely continuous with respect to Lebesgue measure on 
the set $(0,t] \times (-\infty,0]$.
Observe that 
\[
\P\left\{\exists s : (U_t, \underline{Z}_t) = (U_\infty, \underline{Z}_\infty) \; \forall t \ge s \right\} = 1.
\]
Thus, if $A \subseteq (0,\infty) \times
(-\infty,0]$ is Borel with zero Lebesgue measure, then
\[
\P \left\{ (U_\infty,\underline{Z}_\infty) \in A \right\} 
= 
\lim_{t \ra \infty} \P \left\{ (U_t,\underline{Z}_t) \in A \right\} = 0.
\]
The proof the claim concerning the atom at $(0,0)$ follows from the
above formula, the fact that $\P\{(U_t, \underline{Z}_t) = (0,0)\} >
0$ if and only if zero is not regular for the interval $(-\infty,0)$ 
\cite[Theorem 2]{pitmanbravo}, and the hypothesis that 
$\lim_{t \ra \infty} Z_t = +\infty$.
\end{proof}

\begin{remark}
Note that if the process $Z$ has a non-zero Brownian component, 
then the random variable $Z_t$ has an absolutely
continuous distribution for all $t>0$.  
Moreover, in this case zero is regular for the interval $(-\infty,0)$
\end{remark}

Let $\tau = (\tau_t)_{t \ge 0}$ 
be the local time at zero for the process
$Z - \underline{Z}$.  Write $\tau^{-1}$ for the inverse local time process.
Set $\underline{H}_t := \underline{Z}_{\tau^{-1}(t)}$
The process $\underline{H} :=
(\underline{H}_t)_{t \ge 0}$ is the 
{\em descending ladder height process}
for $Z$.  
If $\lim_{t \ra \infty} Z_t = +\infty$, 
then $\hat{\underline{H}} := -\underline{H}$ is
a subordinator killed at an independent
exponential time (see, for example, \cite[Lemma VI.2]{bertoin}).

For the sake of completeness, we include the following observation
that combines well-known results and probably already exists
in the literature -- it can be easily concluded from Theorem 19 and
the remarks at the top of page 172 of \cite{bertoin}.

\begin{lem}
\label{lem:boundeddensity}
Let $Z$ be a L\'evy process such that $Z_0 = 0$
and  $\lim_{t \ra \infty} Z_t = +\infty$.
Then, the distribution of random variable 
$\underline{Z}_\infty$ is absolutely continuous with a bounded density 
if and only if 
the (killed) subordinator
$\hat{\underline{H}}$ has a positive drift coefficient.
\end{lem}

\begin{proof}
Let $S = (S_t)_{t \ge 0}$ be an (unkilled) subordinator
with the same drift coefficient and L\'evy measure as 
$\hat{\underline{H}}$, so that $-\underline{Z}_\infty$
has the same distribution as $S_\zeta$, where $\zeta$
is an independent, exponentially distributed random time.
Therefore, for some $q>0$,
\[
\P\{-\underline{Z}_\infty \in A\} 
= \int_0^\infty q e^{-qt} \P\{S_t \in A\} \, dt
\]
for any Borel set $A \subseteq \reals$.  By a
result of Kesten for general L\'evy processes
(see, for example, \cite[Theorem II.16]{bertoin}) the
$q$-resolvent measure 
$\int_0^\infty e^{-qt} \P\{S_t \in \cdot\} \, dt$
of $S$ is absolutely continuous with a bounded density
for all $q>0$ (equivalently, for some $q>0$) if
and only if points are not essentially polar for $S$.
Moreover, points are not essentially polar 
for a L\'evy process with paths of bounded variation (and, in particular,
for a subordinator)
if and only if the process has a non-zero drift
coefficient  \cite[Corollary II.20]{bertoin}.
\end{proof}

\begin{cor}
\label{cor:boundeddensities}
Let $X$ be a L\'evy process that
satisfies our standing assumptions Hypothesis~\ref{H:standing} and
which has paths of unbounded variation almost surely. 
Then, the random variables
$\inf\{ X_t + \alpha t : t \ge 0 \}$  and 
$\inf\{ X_t - \alpha t : t \le 0 \}$ 
both have absolutely
continuous distributions with bounded densities if and only if 
$X$ has a non-zero Brownian component.
\end{cor}

\begin{proof} 
By Lemma \ref{lem:boundeddensity}, the distributions in question
are absolutely continuous with bounded densities
if and only if the drift coefficients of the
descending ladder processes for the two L\'evy processes
$(X_t + \alpha t)_{t \ge 0}$ 
and 
$(-X_{t} + \alpha t)_{t \ge 0}$ are non-zero. 
By the results of \cite{MR0321198} (see also \cite[Theorem
VI.19]{bertoin}), this occurs if and only if both 
$(X_t + \alpha t)_{t \ge 0}$ and $(-X_{t} + \alpha t)_{t \ge 0}$ 
have positive probability of creeping down across $x$ for some (equivalently, all) $x<0$, where we recall that a
L\'evy process creeps down across $x < 0$ if the first passage time in $(-\infty, x)$ is not a jump time for the path of the process. Equivalently, 
both densities exist and are bounded if and only if 
the L\'evy process
$(X_t + \alpha t)_{t \ge 0}$ creeps downwards and 
the L\'evy process
$(X_t - \alpha t)_{t \ge 0}$ creeps upwards, where
the latter notion is defined in the obvious way.

A result of Vigon \cite{MR1875147} 
(see also \cite[Chapter 6, Corollary 9]{doney}) states that when
the paths of $X$ have unbounded variation, $(X_t + \alpha t)_{t \ge 0}$ creeps
downward if and only if $X$ creeps downward, and hence, in turn, if and
only if $(X_t - \alpha t)_{t \ge 0}$ creeps downwards. A similar
result applies to creeping upwards. 

Thus, both  densities exist and
are bounded if and only if $X$ creeps downwards and upwards. 
This occurs if and only if 
the ascending and descending ladder processes of $X$ have
positive drifts \cite[Theorem VI.19]{bertoin}, which happens if and
only if $X$ has a 
non-zero Brownian component \cite[Chapter 4, Corollary
4(i)]{doney} (or see the remark after the proof of 
\cite[Theorem VI.19]{bertoin}).
\end{proof}

\section{The complementary interval straddling zero}
\label{sec:proofs}

\subsection{Distributions in the case of a non-zero Brownian
  component}
\label{sec:proofs1}

Suppose that $X = (X_t)_{t \in \reals}$ is a L\'evy process 
that satisfies our standing assumptions Hypothesis~\ref{H:standing}.  
Also, suppose until further notice
that $X$ has a non-zero Brownian component.

Recall that $M =
(M_t)_{t \in \mathbb{R}}$ is the $\alpha$-Lipschitz minorant of $X$
and
$\mathcal{Z}$ is the stationary regenerative set 
$\{t \in \reals : X_t \wedge X_{t-} = M_t\}$.  Recall also that
$K = D - G$, where
$
G = \sup\{ t < 0 : X_t \wedge X_{t-} = M_t \} = \sup\{t < 0 : t \in \mathcal{Z}\}
$
and
$
D = \inf\{ t > 0 : X_t \wedge X_{t-} = M_t \} = \inf\{t > 0 : t \in \mathcal{Z}\}
$.  Lastly, recall that $T$ is the unique $t \in [G,D]$ such that
$M_t = \max \{M_s : s \in [G, D]\}$.

Let $f^+$ (respectively, $f^-$) 
be the joint density of the 
random variables we denoted
by $(U_\infty,\underline{Z}_\infty)$
in Lemma~\ref{lem:abscont} in the case where the
L\'evy process $Z$ is $(X_t + \alpha t)_{t \ge 0}$
(respectively, $(-X_t + \alpha t)_{t \ge 0}$).

\begin{prop}
\label{prop:tlrdensity2}
Let $X$ be a L\'evy process that
satisfies our standing assumptions Hypothesis~\ref{H:standing}.  
Suppose, moreover,
that $X$ has a non-zero Brownian component.  
Set $L := T - G$ and $R := D - T$. Then,
the random vector $(T,L,R)$ has a distribution that
is absolutely continuous with respect to Lebesgue measure
with joint density 
\[
(\tau, \lambda, \rho) 
\mapsto 
2  \alpha \int_{-\infty}^0  f^-(\lambda,h) f^+(\rho,h) \, dh,
\quad  \text{$\lambda,\rho > 0$ and $\tau - \lambda < 0 < \tau + \rho$.}
\]
Therefore, $(T,G,D)$ also has 
an absolutely continuous distribution
with joint density 
\[
(\tau, \gamma, \delta) 
\mapsto 
2  \alpha \int_{-\infty}^0  f^-(\tau-\gamma,h) f^+(\delta-\tau,h) \, dh, \quad \text{$\gamma < 0 < \delta$ and $\gamma < \tau < \delta$,}
\]
and $K$ has an absolutely continuous distribution with density
\[
\kappa \mapsto 
2  \alpha \kappa \int_0^\kappa \int_{-\infty}^0 f^-(\xi,h) f^+(\kappa-\xi,h) \, dh \, d \xi, \quad \kappa > 0.
\] 
\end{prop}

\begin{proof}
Observe that $X$ is abrupt and so,
by Theorem~\ref{thm:newLambda} and Remark~\ref{rem:abrupt}, $\mathcal{Z}$ 
is a stationary discrete random set with intensity 
$$\left( \frac{\int_{\reals_+} x \Lambda(dx)}{\Lambda(\reals_+)}
\right)^{-1} = \frac{\Lambda(\reals_+)}{\int_{\reals_+} x \Lambda(dx)}
< \infty.$$ Hence, the set of times of peaks of the $\alpha$-Lipschitz minorant
$M$ is also a stationary discrete random set with the same finite
intensity.  The point process consisting of a single point at
time $T$ is included in the set of times of peaks of $M$, and so for $A$ a
Borel set with Lebesgue measure $|A|$ we have 
\begin{eqnarray*}
\P\{T \in A\} & \leq & 
\P \left\{ \text{at least one peak of $M$ at a time $t \in A$} \right\} \\
& \leq  & \E \left[ \text{number of times of peaks in $A$} \right] \\
& = & \frac{ \Lambda(\reals_+) }{ \int_{\reals_+} x \Lambda(dx)} \,
|A| . \\
\end{eqnarray*}
Thus, the distribution of $T$ is absolutely continuous with respect to
Lebesgue measure with density bounded above by 
$\Lambda(\reals_+)/\int_{\reals_+} x \Lambda(dx)$.

It follows from the observations made in
the proof of Theorem~\ref{thm:Zregen} about the nature of the global
infimum of the process $\tilde{X}$ that under our
hypotheses, almost surely 
$X_G = X_{G-} = M_T - \alpha |G-T| = M_T - \alpha L$, $X_D = X_{D-} = M_T - \alpha |D-T| = M_T - \alpha R$,
and $X_t \wedge X_{t-} > M_T - \alpha |t-T|$ for $t \notin \{ G, D \}$.
Thus,
\[
0
=
\inf\{ X_{T+t} - (M_T - \alpha t) : t \ge 0 \} 
=
X_{T+R} - (M_T - \alpha R)
\] 
and
\[
\begin{split}
0
& =
\inf\{ X_{T+t} - (M_T + \alpha t) : t \le 0 \} \\
& =
\inf\{ X_{T-t} - (M_T - \alpha t) : t \ge 0 \}
=
X_{T-L} - (M_T - \alpha L). \\
\end{split}
\]
Consequently,
\begin{equation}
\label{E:TLR_char}
\begin{split}
X_{T-L} - X_T + \alpha L
& =
\inf\{ X_{T-t} - X_T + \alpha t : t \ge 0 \} \\
& =
\inf\{X_{T+t} - X_T + \alpha t : t \ge 0 )\}
=
X_{T+R} - X_T + \alpha R. \\
\end{split}
\end{equation}
Conversely, $(T,L,R)$ is the unique triple with
$T-L < 0 < T+R$ such that \eqref{E:TLR_char} holds.

Fix $\tau \in \mathbb{R}$ and $\lambda,\rho \in \mathbb{R}_+$
such that $\tau - \lambda < 0 < \tau + \rho$.  Set
\[
\begin{split}
Z_t^- & := X_{\tau - t} - X_\tau + \alpha t, \quad t \ge 0, \\
\underline{Z}^-   & := \inf\{Z_t^- : t \ge 0\}, \\
U^-   & := \inf\{t \ge 0: Z_t^-   = \underline{Z}^-\}. \\
\end{split}
\]
For $0 < \Delta \tau < \rho$ set
\[
\begin{split}
Z_t^+ & := X_{t + \tau + \Delta \tau} - X_{\tau + \Delta \tau} + \alpha t, \quad t \ge 0, \\
\underline{Z}^+    &  := \inf\{Z_t^+ : t \ge 0\}, \\
U^+   & := \inf\{t \ge 0: Z_t^+ = \underline{Z}^+\}. \\
\end{split}
\]

From \eqref{E:TLR_char} we have  that
\begin{equation}
\label{upper_LTR_bd}
\begin{split}
& \mathbb{P}
\{
T \in [\tau, \tau+\Delta \tau],\,
L > \lambda, \,
R >\rho
\} \\
& \quad \le
\mathbb{P}
(
\{
U^- > \lambda - \Delta \tau, \,
U^+ > \rho - \Delta \tau
\} \\
& \qquad \cap
\{
\exists 0 \le s \le \Delta \tau: 
X_\tau + \underline{Z}^- + \alpha s = X_{\tau + \Delta \tau} + \underline{Z}^+ + \alpha (\Delta \tau - s)
\}
). \\
\end{split}
\end{equation}

Similarly,
\begin{equation}
\label{lower_LTR_bd}
\begin{split}
& \mathbb{P}
\{
T \in [\tau, \tau+\Delta \tau],\,
L > \lambda, \,
R >\rho
\} \\
& \quad \ge
\mathbb{P}
(
\{
U^- > \lambda, \,
U^+ > \rho
\} \\
& \qquad \cap
\{
\exists 0 \le s \le \Delta \tau: 
X_\tau + \underline{Z}^- + \alpha s = X_{\tau + \Delta \tau} + \underline{Z}^+ + \alpha (\Delta \tau - s)
\} \\
& \qquad \cap
\{
\inf\{
X_{\tau+s} - (X_\tau  + \underline{Z}^- + \alpha s)
: 
0 \le s \le \Delta \tau
\}
> 0
\} \\
& \qquad \cap
\{
\inf\{
X_{\tau+\Delta \tau - s} - (X_{\tau+\Delta \tau}  + \underline{Z}^+ + \alpha s)
: 
0 \le s \le \Delta \tau
\}
> 0
\}
) \\
& \quad \ge
\mathbb{P}
(
\{
U^- > \lambda, \,
U^+ > \rho
\} \\
& \qquad \cap
\{
\exists 0 \le s \le \Delta \tau: 
X_\tau + \underline{Z}^- + \alpha s = X_{\tau + \Delta \tau} + \underline{Z}^+ + \alpha (\Delta \tau - s)
\}
) \\
& \qquad -
\mathbb{P}
(
\{
\exists 0 \le s \le \Delta \tau: 
X_\tau + \underline{Z}^- + \alpha s = X_{\tau + \Delta \tau} + \underline{Z}^+ + \alpha (\Delta \tau - s)
\} \\
& \qquad \quad \cap
\{
\inf\{
X_{\tau+s} - (X_\tau  + \underline{Z}^- + \alpha s)
: 
0 \le s \le \Delta \tau
\}
\le 0
\}
) \\
& \qquad -
\mathbb{P}
(
\{
\exists 0 \le s \le \Delta \tau: 
X_\tau + \underline{Z}^- + \alpha s = X_{\tau + \Delta \tau} + \underline{Z}^+ + \alpha (\Delta \tau - s)
\} \\
& \qquad \quad \cap
\{
\inf\{
X_{\tau+\Delta \tau - s} - (X_{\tau+\Delta \tau}  + \underline{Z}^+ + \alpha s)
: 
0 \le s \le \Delta \tau
\}
\le 0
\}
). \\
\end{split}
\end{equation}

Observe that
\[
\begin{split}
& \mathbb{P}
(
\{
\exists 0 \le s \le \Delta \tau: 
X_\tau + \underline{Z}^- + \alpha s = X_{\tau + \Delta \tau} + \underline{Z}^+ + \alpha (\Delta \tau - s)
\} \\
& \qquad  \cap
\{
\inf\{
X_{\tau+s} - (X_\tau  + \underline{Z}^- + \alpha s)
: 
0 \le s \le \Delta \tau
\}
\le 0
\}
) \\
& \quad =
\mathbb{P}
(
\{
\{
\exists 0 \le s \le \Delta \tau: 
(\underline{Z}^+ - \underline{Z}^-)  + (X_{\tau + \Delta \tau} - X_\tau) =  2 \alpha s - \alpha \Delta \tau
\} \\
& \qquad \quad \cap
\{
\underline{Z}^- \ge \inf_{0 \le s \le \Delta \tau} (X_{\tau+s} - X_\tau - \alpha s)
\}
) \\
& \quad = 
\mathbb{P}
(
\{
\{
(\underline{Z}^+ - \underline{Z}^-)  + (X_{\tau + \Delta \tau} - X_\tau)  \in  [-\alpha \Delta \tau, \alpha \Delta \tau]
\} \\
& \quad \qquad \cap
\{
\underline{Z}^- \ge \inf_{0 \le s \le \Delta \tau} (X_{\tau+s} - X_\tau - \alpha s)
\}
). \\
\end{split}
\]
By Corollary~\ref{cor:boundeddensities},
the independent random variables $\underline{Z}^-$ and $\underline{Z}^+$ 
have densities bounded by some constant $c$.  
Moreover, they are independent of 
$(X_{\tau+s})_{0 \le s \le \Delta \tau}$.  Conditioning on $\underline{Z}^-$
and $(X_{\tau+s})_{0 \le s \le \Delta \tau}$, we see that the last probability 
is, using $| \cdot |$ to denote Lebesgue measure, at most
\[
\begin{split}
& \E
[
c
|
[
\underline{Z}^-  - (X_{\tau + \Delta \tau} - X_\tau) - \alpha \Delta \tau,
\underline{Z}^-  - (X_{\tau + \Delta \tau} - X_\tau) + \alpha \Delta \tau
] 
| \\
& \qquad \times
\mathbf{1}
\{
\underline{Z}^- \ge \inf_{0 \le s \le \Delta \tau} (X_{\tau+s} - X_\tau - \alpha s)
\}
] \\
& \quad =
2 c \alpha \Delta \tau
\P
\{
\underline{Z}^- \ge \inf_{0 \le s \le \Delta \tau} (X_{\tau+s} - X_\tau - \alpha s)
\}. \\
\end{split}
\]
Consequently, 
\begin{equation}
\label{Zminus_ot}
\begin{split}
& \mathbb{P}
(
\{
\exists 0 \le s \le \Delta \tau: 
X_\tau + \underline{Z}^- + \alpha s = X_{\tau + \Delta \tau} + \underline{Z}^+ + \alpha (\Delta \tau - s)
\} \\
& \qquad  \cap
\{
\inf\{
X_{\tau+s} - (X_\tau  + \underline{Z}^- + \alpha s)
: 
0 \le s \le \Delta \tau
\}
\le 0
\}
) \\
& \quad =
o(\Delta \tau) \\
\end{split}
\end{equation}
as $\Delta \tau \downarrow 0$.

The same argument shows that 
\begin{equation}
\label{Zplus_ot}
\begin{split}
& \mathbb{P}
(
\{
\exists 0 \le s \le \Delta \tau: 
X_\tau + \underline{Z}^- + \alpha s = X_{\tau + \Delta \tau} + \underline{Z}^+ + \alpha (\Delta \tau - s)
\} \\
& \qquad  \cap
\{
\inf\{
X_{\tau+\Delta \tau - s} - (X_{\tau+\Delta \tau}  + \underline{Z}^+ + \alpha s)
: 
0 \le s \le \Delta \tau
\}
\le 0
\}
) \\
& \quad =
o(\Delta \tau) \\
\end{split}
\end{equation}
as $\Delta \tau \downarrow 0$.

Now,
\[
\begin{split}
& \mathbb{P}
(
\{
U^- > \lambda, \,
U^+ > \rho
\} \\
& \qquad \cap
\{
\exists 0 \le s \le \Delta \tau: 
X_\tau + \underline{Z}^- + \alpha s = X_{\tau + \Delta \tau} + \underline{Z}^+ + \alpha (\Delta \tau - s)
\}
) \\
& \quad = \mathbb{P}
(
\{
U^- > \lambda,  \,
U^+ > \rho
\} \\
& \qquad \cap
\{
\exists 0 \le s \le \Delta \tau: 
(\underline{Z}^+ - \underline{Z}^-)  + (X_{\tau + \Delta \tau} - X_\tau) =  2 \alpha s - \alpha \Delta \tau
\}
) \\
& \quad =
\mathbb{P}
(
\{
U^- > \lambda,\,
U^+ > \rho
\} \\
& \quad \qquad \cap
\{
(\underline{Z}^+ - \underline{Z}^-)  + (X_{\tau + \Delta \tau} - X_\tau) \in [- \alpha \Delta \tau, + \alpha \Delta \tau]
\}
). \\
\end{split}
\]

The random vectors
$(U^-,\underline{Z}^-)$ and $(U^+, \underline{Z}^+)$ are independent with
respective densities $f^-$ and $f^+$, and so the
joint density of $(U^-, U^+, \underline{Z}^+ - \underline{Z}^- )$ is
\[
(u,v,w) \mapsto \int_{-\infty}^\infty  f^-(u,h-w) f^+(v,h) \, dh.
\]
Thus, using the facts that the random variable
$\underline{Z}^+ - \underline{Z}^-$ is independent of
$X_{\tau + \Delta \tau} - X_\tau$ and the latter random variable has the same
distribution as $X_{\Delta \tau}$,
\[
\begin{split}
& \mathbb{P}
(
\{
U^- > \lambda,\,
U^+ > \rho 
\} \\
& \qquad \cap
\{
(\underline{Z}^+ - \underline{Z}^-)  + (X_{\tau + \Delta \tau} - X_\tau) \in [- \alpha \Delta \tau, + \alpha \Delta \tau]
\}
) \\
& \quad =
\int_{\lambda}^{\infty}  du \,
\int_{\rho}^{\infty}  dv \,
\int_{-\infty}^\infty dw \,
\int_{-\infty}^\infty dh  \\
& \qquad \times
\P\{-w- \alpha \Delta \tau < X_{\Delta \tau} < -w + \alpha \Delta \tau\} \,
f^-(u,h-w) f^+(v,h).  \\
\end{split}
\]

By Fubini's theorem,
\[
\begin{split}
& \int_{-\infty}^\infty dw \,
\P\{-w- \alpha \Delta \tau < X_{\Delta \tau} < -w + \alpha \Delta \tau\} \\
& \quad =
\E\left[
\int_{-\infty}^\infty dw \,
\mathbf{1} \{-X_{\Delta \tau} - \alpha \Delta \tau < w < -X_{\Delta \tau} + \alpha \Delta \tau\}
\right] \\
& \quad = \E\left[2 \alpha \Delta \tau\right]
= 2 \alpha \Delta \tau. \\
\end{split}
\]
Moreover, for any $\epsilon > \Delta \tau$,
\[
\begin{split}
& \int_{-\infty}^\infty dw \,
\P\{-w- \alpha \Delta \tau < X_{\Delta \tau} < -w + \alpha \Delta \tau\} 
\mathbf{1}\{|w| > \epsilon\} \\
& \quad =
\E\left[
\int_{-\infty}^\infty dw \,
\mathbf{1} \{-X_{\Delta \tau} - \alpha \Delta \tau < w < -X_{\Delta \tau} + \alpha \Delta \tau, \, |w| > \epsilon\}
\right] \\
& \quad =
\E\left[
(|X_{\Delta \tau}|-(\epsilon-\Delta \tau))_+ \wedge (2 \Delta \tau)
\right]. \\
\end{split}
\]
Note that $(\Delta \tau)^{-1} [(|X_{\Delta \tau}|-(\epsilon-\Delta \tau))_+ \wedge (2 \Delta \tau)] \le 2$ and that 
the random variable on the left of this inequality
converges to $0$ almost surely as $\Delta \tau \downarrow 0$.
Hence, by bounded convergence,
\[
\lim_{\Delta \tau \downarrow 0}
\int_{-\infty}^\infty dw \,
(\Delta \tau)^{-1}
\P\{-w- \alpha \Delta \tau < X_{\Delta \tau} < -w + \alpha \Delta \tau\} 
\mathbf{1}\{|w| > \epsilon\}
=
0.
\]

Furthermore,  the independent
random variables $\underline{Z}^-$ and $\underline{Z}^+$ 
both have bounded densities
by Corollary~\ref{cor:boundeddensities};  that is, the functions
$h \mapsto \int_0^\infty du \, f^-(u,h)$  
and
$h \mapsto \int_0^\infty dv \, f^+(v,h)$
both belong to $L^1 \cap L^\infty$. Therefore, the
functions
$h \mapsto \int_{\lambda}^\infty du \, f^-(u,h)$
and
$h \mapsto \int_{\rho}^\infty  dv \, f^+(v,h)$
both certainly belong to $L^1 \cap L^\infty$.

It now follows from the Lebesgue differentiation theorem that
\[
\begin{split}
& \lim_{\Delta \tau \downarrow 0}
(\Delta \tau)^{-1}
\int_{-\infty}^\infty dw \,
\P\{-w- \alpha \Delta \tau < X_{\Delta \tau} < -w + \alpha \Delta \tau\}
\int_{\lambda}^\infty du \,
f^-(u,h-w) \\
& \quad =
2 \alpha
\int_{\lambda}^\infty du \,
f^-(u,h) \\
\end{split}
\]
for Lebesgue almost every $h \in \reals$.  Moreover, 
the quantity on the left is bounded by 
$\sup_{h \in \reals} 2 \alpha
\int_{\lambda}^\infty du \,
f^-(u,h) < \infty$.  Therefore, by \eqref{upper_LTR_bd},
\eqref{lower_LTR_bd}, \eqref{Zminus_ot}, \eqref{Zplus_ot},
and bounded convergence,
\begin{equation}
\label{E:T_limit}
\begin{split}
& \lim_{\Delta \tau \downarrow 0}
(\Delta \tau)^{-1} 
\mathbb{P}
\{
T \in [\tau, \tau + \Delta \tau], \,
L > \lambda, \,
R > \rho
\} \\
& \quad =
2 \alpha \int_{\lambda}^\infty du \,
\int_{\rho}^\infty dv \,
\int_{-\infty}^\infty dh \, f^-(u,h) f^+(v,h). \\
\end{split}
\end{equation}

As we observed above, the measure
$\P\{T \in d\tau\}$ is absolutely continuous with
density bounded above by $\Lambda(\reals_+) < \infty$,
and so the same is certainly true of the measure
$
\P\{
T \in d\tau, \,
L > \lambda, \,
R > \rho
\}
$
for fixed $\lambda$ and $\rho$.
Therefore, by \eqref{E:T_limit} and the Lebesgue differentiation theorem,
\[
\begin{split}
& \P\{
T \in A, \,
L > \lambda, \,
R > \rho
\} \\
& \quad =
2 \alpha 
\int_{-\infty}^{\infty} d \tau \,
\int_{\lambda}^\infty du \,
\int_{\rho}^\infty dv \,
\int_{-\infty}^\infty dh \, f^-(u,h) f^+(v,h) 
\mathbf{1}\{\tau \in A\} \\
\end{split}
\]
for any Borel set $A \subseteq (-\rho,\lambda)$, 
and this establishes that $(T,L,R)$ has the claimed density.

The remaining two claims follow immediately.
\end{proof}

\begin{cor}
\label{cor:Klaplace}
Under the assumptions of Proposition~\ref{prop:tlrdensity2},
 \[
\begin{split}
\! \! \! \! \! \! \! \! 
\mathbb{E}[e^{-\theta K}] 
& =
- 4 \pi \alpha \frac{d}{d \theta}
\int_{-\infty}^\infty
\Biggl(
\exp \left\{
 \int_0^\infty dt \int_0^\infty [ e^{- \theta t + i z x}   - 1 ] t^{-1}
 \P\{X_t - \alpha t \in dx\}
 \right\} \\
& \qquad \qquad \qquad \qquad \qquad \times \exp \left\{
 \int_0^\infty dt \int_0^\infty [ e^{- \theta t - i z x}   - 1 ] t^{-1}
 \P\{- X_t - \alpha t \in dx\}
 \right\}
\Biggr) \, dz.\\
\end{split}
\]
\end{cor}
\vspace{-10pt}
\begin{proof}
From Proposition~\ref{prop:tlrdensity2},
\begin{equation}
\label{eq:Klaplace1}
\begin{split}
& \mathbb{E}[e^{-\theta K}] \\
& \quad = 2  \alpha \int_{-\infty}^0 \int_0^\infty \kappa \int_0^\kappa f^-(\kappa-\xi,h) f^+(\xi,h) e^{-\theta \kappa} \, d \xi \, d \kappa \, dh \\
& \quad = 2  \alpha \int_{-\infty}^0 \int_0^\infty \int_\xi^\infty \kappa f^-(\kappa-\xi,h) f^+(\xi,h)  e^{-\theta \kappa} \, d \kappa \, d \xi \, dh \\
& \quad = 2  \alpha \int_{-\infty}^0  \Biggl(
\int_0^\infty f^+(\xi,h) e^{-\theta \xi} \int_\xi^\infty ( \kappa - \xi) f^-(\kappa-\xi,h) e^{-\theta ( \kappa - \xi) } \, d \kappa \, d \xi  \\
& \qquad 
+ \int_0^\infty \xi f^+(h,\xi) e^{-\theta \xi} \int_\xi^\infty f^-(h,\kappa-\xi) e^{-\theta ( \kappa - \xi) } \, d \kappa \, d \xi 
 \Biggr) \, dh \\
& = 2  \alpha \int_{-\infty}^0  \Biggl(
\int_0^\infty f^+(\xi,h) e^{-\theta \xi} \int_0^\infty  \kappa f^-(\kappa,h) e^{-\theta \kappa } \, d \kappa \, d \xi  \\
& \qquad
+ \int_0^\infty \xi f^+(\xi,h) e^{-\theta \xi} \int_0^\infty f^-(\kappa,h) e^{-\theta  \kappa  } \, d \kappa \, d \xi 
 \Biggr) \, dh \\
&  \quad = - 2 \alpha \frac{d}{d \theta} \Biggl( \int_{-\infty}^0 
\Bigl( \int_0^\infty f^+(\xi,h) e^{-\theta \xi} \, d \xi \Bigr) 
\Bigl( \int_0^\infty f^-(\kappa,h) e^{-\theta \kappa } \, d \kappa \Bigr) 
 \, dh \Biggr). \\
\end{split}
\end{equation}
Viewing $\int_0^\infty f^+(\xi,h) e^{-\theta \xi} \, d \xi $ 
and $\int_0^\infty f^-(\kappa,h) e^{-\theta \kappa }  \, d \kappa $
as functions of $h$ that belong to
$L^1 \cap L^\infty \subset L^2$, we can use Plancherel's Theorem 
and then the Pecherskii-Rogozin formulas \cite[p. 28]{doney} again
to get that $\mathbb{E}[e^{-\theta K}]$ is 
 \[
\begin{split}
& - 2 \alpha \frac{d}{d \theta} 
 2 \pi \int_{-\infty}^\infty 
 \Biggl(
 \int_0^{\infty} \int_0^\infty f^+(\xi,-h) e^{izh - \theta \xi} \, d \xi \, dh \\
& \qquad \times \overline{
\int_0^{\infty} \int_0^\infty f^-(\kappa,-h) e^{izh - \theta \kappa} \, d \kappa \, dh
}
 \, dz \Biggr) \\
&  \quad = - 4 \pi \alpha \frac{d}{d \theta}
\int_{-\infty}^\infty
\Biggl(  
\exp \left\{
 \int_0^\infty dt \int_0^\infty [ e^{- \theta t + i z x}   -1 ] t^{-1}
 \P\{X_t - \alpha t \in dx\}
 \right\} \\
& \quad \qquad \times
\overline{
\exp \left\{
 \int_0^\infty dt \int_0^\infty [ e^{- \theta t + i z x}   -1 ] t^{-1}
 \P\{- X_t - \alpha t \in dx\}
 \right\}
}
\Biggr) \, dz \\
& \quad = 
- 4 \pi \alpha \frac{d}{d \theta}
\int_{-\infty}^\infty
\Biggl(
\exp \left\{
 \int_0^\infty dt \int_0^\infty [ e^{- \theta t + i z x}   - 1 ] t^{-1}
 \P\{X_t - \alpha t \in dx\}
 \right\} \\
& \quad \qquad \times \exp \left\{
 \int_0^\infty dt \int_0^\infty [ e^{- \theta t - i z x}   - 1 ] t^{-1}
 \P\{- X_t - \alpha t \in dx\}
 \right\}
\Biggr) \, dz .\\
\end{split}
\]

\vspace{-20pt}

\end{proof}

\subsection{Extension to more general L\'evy processes}

Corollary~\ref{cor:Klaplace} establishes 
Theorem~\ref{thm:bigthm} when $X$ has a non-zero Brownian component. 
The next few results allow us establish the latter result
for the class of L\'evy processes described in its statement.

Recall the definitions
\[
\begin{split}
  G &:= \sup\{ t < 0 : X_t \wedge X_{t-} = M_t \} = \sup\{t < 0 : t
  \in \mathcal{Z}\}
  ,  \\
  D &:= \inf\{ t > 0 : X_t \wedge X_{t-} = M_t \} = \inf\{t > 0 : t
  \in \mathcal{Z}\}
  , \\
  T &:= \arg \max \{M_t : G \le t \le D\} , \\
  S &:= \inf \{ t > 0 : X_t \wedge X_{t-} - \alpha t 
  \leq \inf\{ X_{s} - \alpha s : s \le 0\} \}. \\
\end{split}
\]
As in the proof of Theorem~\ref{thm:Zregen}, 
it follows from Lemma~\ref{lem:recipe} that almost surely
\[
D = \inf \{ t \ge S : X_t \wedge X_{t-} + \alpha (t-S) = \inf \{ X_u
+ \alpha (u-S) : u \ge S\} \} .
\]

\begin{prop}
\label{prop:order}
Suppose that $X$ is  L\'evy process satisfying
our standing assumptions Hypothesis~\ref{H:standing}. 
Then, $\P\{0 \notin \mathcal{Z}, \, S=0\} = 0$. In addition,
\begin{itemize}
\item[(a)]
If $X$ has paths of unbounded variation,  then $G<T<S<D$ a.s.
\item[(b)]
If $X$ has paths of bounded variation and drift
coefficient $d$ satisfying $d < - \alpha$, then $G < T<S<D$ a.s.,
and if $X$ has paths of bounded variation and drift
coefficient $d$ satisfying $d > \alpha$, then $G < T<S \leq D$ a.s.
\item[(c)]
If $X$ has paths of bounded variation and drift 
coefficient $d$ satisfying $|d| < \alpha$, then almost surely either $0 \in
\mathcal{Z}$ and $G = T = S = D = 0$, or $0 \notin \mathcal{Z}$ and $G
\leq T \leq S \leq D$. Furthermore, $T=S=D$ almost surely 
on the event $\{T=S\}$.
\end{itemize}
\end{prop}

\begin{proof}
Firstly, if $0 \notin \mathcal{Z}$, then
$\inf \{ X_u - \alpha u : u \leq 0 \} < 0$, and thus $S>0$ a.s.
on the event $\{0 \notin \mathcal{Z}\}$. 

(a) Suppose that $X$ has paths of unbounded variation.
We have from Theorem~\ref{thm:Zmeasure} 
(see Remark~\ref{R:Zmeasure} (i)) that
$0 \notin \mathcal{Z}$ almost surely.
It follows from \eqref{eq:Rogozin_small_time} that
at the stopping time $S$
\[
- \liminf_{\eps \downarrow 0} \eps^{-1}(X_{S+\eps} - X_S) = \limsup_{\eps \downarrow 0}
\eps^{-1}(X_{S+\eps} - X_S) = \infty,
\]
and hence it is not possible for the 
$\alpha$-Lipschitz minorant to meet the path of $X$ at time $S$. 
Thus, $T < S < D$ almost surely by Corollary~\ref{C:t_s_d}.
By time reversal, $G < T$ almost surely.  

(b) Suppose $X$ has paths of bounded variation and drift
coefficient $d$ satisfying $|d| > \alpha$, then
we have from Theorem~\ref{thm:Zmeasure} 
(see Remark~\ref{R:Zmeasure} (ii)) that
$0 \notin \mathcal{Z}$ almost surely. Therefore, by 
Corollary~\ref{C:t_s_d}, if $T = S$ then $T=S=D$.

Suppose that $d<-\alpha$.
It follows from \eqref{eq:shtatland_statement}
that
\[
\lim_{\eps \downarrow 0} \eps^{-1}(X_{S+\eps} - X_S) = d, \quad \text{a.s.}
\]
Thus, $S \notin \mathcal{Z}$ and, in particular,
$S < D$, so that $T < S < D$ a.s.

On the other hand, if $d > \alpha$, then the L\'evy process
$(X_t - \alpha t)_{t \ge 0}$ has
positive drift and so the associated
descending ladder process has zero drift coefficient \cite[p. 56]{doney}. 
In that case, for any $x < 0$ we have $X_V < x$ almost surely, where 
$V := \inf\{t \ge 0 : X_t - \alpha t \le x\}$
\cite[Theorem III.4]{bertoin}. Therefore,
\[
X_S - \alpha S < \inf \{ X_{u} - \alpha u : u \le 0 \}  \quad
\text{a.s.} 
\]
If $T=S$, then $T=S=D$ by Corollary~\ref{C:t_s_d}, and then 
\[
\begin{split}
X_S 
& = X_D \wedge X_{D-} \\
& = X_G \wedge X_{G-} + \alpha(D-G) \\
& = X_G \wedge X_{G-} + \alpha(S-G), \\
\end{split}
\]
which results in the contradiction
\[
X_G \wedge X_{G-} - \alpha G 
= X_S - \alpha S 
 < \inf \{ X_{u} - \alpha u : u \le 0 \}. 
\]
Thus, $T<S \leq D$ a.s. 

The results for $G$ now follow by a time reversal argument.

(c) Suppose $X$ has paths of bounded variation and drift
coefficient $d$ satisfying $|d| > \alpha$.
We know from Theorem~\ref{thm:Zmeasure} and
Remark~\ref{R:Zmeasure} that the subordinator associated with
$\mathcal{Z}$ has non-zero drift and so $\mathcal{Z}$
has positive Lebesgue measure almost surely.  The subset of points
of $\mathcal{Z}$ that are isolated on either the left or
the right is countable and hence has zero Lebesgue measure.
It follows from the stationarity of $\mathcal{Z}$ that
$G = T = S = D = 0$ almost surely  
on the event 
$\{0 \in \mathcal{Z}\}$. The
remaining statements can be read from Corollary~\ref{C:t_s_d}.
\end{proof}

\begin{lem}
\label{lem:continuousmapping}
Let $X$ be a L\'evy process that satisfies our standing assumptions
Hypothesis~\ref{H:standing}.
Suppose, moreover, that if $X$ has paths of bounded variation, then
$|d| \ne \alpha$.
For $\eps > 0$ set $X^\eps = X + \eps B$, where $B$ is a standard
Brownian motion on $\reals$, independent of $X$. Define
$G^\eps$, $D^\eps$ and $K^\eps = D^\eps - G^\eps$
to be the analogues of $G$, $D$ and $K$ with
$X$ replaced by $X^\eps$.  Then, $(G^\eps,D^\eps)$
converges almost surely to $(G,D)$ as $\eps \downarrow 0$,
and so $K^\eps$ converges almost surely to $K$ as $\eps \downarrow 0$.
\end{lem}
\begin{proof}
By symmetry, it suffices to show that 
$D^\eps$
converges almost surely to $D$ as $\eps \downarrow 0$. We first show
the convergence on the event $\{S>0\}$.

Let $S^\eps$ be the analogue of the stopping time $S$ with $X$
replaced by $X^\eps$.  
As we observed in the proof of Theorem~\ref{thm:Zregen}, 
$X_S - \alpha S 
= X_S \wedge X_{S-} - \alpha S 
\le \inf \{ X_u - \alpha u : u \leq 0 \}$.
If $X$ has paths of unbounded variation
or bounded variation with drift satisfying $d < \alpha$, then, since $S$
is a stopping time, 
$\liminf_{u \downarrow S} (X_u - X_S - \alpha(u-S))/(u-S) < 0$.
If $X$ has paths of bounded variation with drift satisfying $d  > \alpha$,
then by the remarks at the top of page 56 of \cite{doney},
the downwards ladder height process of the process
$(X_u - \alpha u)_{u \ge 0}$ (resp. $(-X_u + \alpha u)_{u \ge 0}$)
has zero drift (resp. non-zero drift).
By Lemma~\ref{lem:boundeddensity}, the distribution of 
$\inf \{ X_u - \alpha u : u \leq 0 \}$ is absolutely continuous with
a bounded density, and hence
\[
\P \left\{ X_S - \alpha S = \inf \{ X_u - \alpha u : u \leq 0 \}
\right\}
=
0
\]
by Fubini's theorem and the fact that the range of
a subordinator with zero drift has zero Lebesgue measure almost surely.

For all three of these cases, 
given any $\delta > 0$ we can, with probability one,
thus find a time $t \in (S,S+\delta)$ such that 
$$
X_t \wedge X_{t-} - \alpha t <  \inf \{ X_u - \alpha u : u \leq 0 \}.
$$
By the strong law of large numbers for the Brownian motion $B$,
\[
\lim_{\eps \downarrow 0} 
\inf\{X_{u}^\eps - \alpha u : u \le 0\}
=
\inf \{ X_u - \alpha u : u \leq 0 \}.
\]
Hence, $X_t^\eps \wedge X_{t-}^\eps - \alpha t \le \inf\{
X_{u}^\eps - \alpha u : u \le 0\}$ for $\eps$ sufficiently small, and so
$S^\eps \le S + \delta$ for such an $\eps$.  
Therefore, $\limsup_{\eps \downarrow 0}
S^\eps \le S$.

On the other hand, for any $\delta > 0$ we have
\[
\inf\left\{
X_t \wedge X_{t-} - \alpha t 
- \inf \{ X_u - \alpha u
: u \le 0 \} 
: t \in [0, (S - \delta)_+]
\right\} > 0. 
\]
Thus, $X_t^\eps
\wedge X_{t-}^\eps - \alpha t > \inf\{ X_{u}^\eps - \alpha u : u \le
0\}$ for all $t \in [0, (S- \delta)_+]$ for $\eps$ sufficiently small, so
that $S^\eps \ge (S - \delta)_+$. Therefore, $\liminf_{\eps \downarrow 0}
S^\eps \ge S$. Consequently, $\lim_{\eps \downarrow 0} S^\eps = S$.

Now, as a result of the uniqueness of the global infima of L\'evy
processes that are not compound Poisson processes with zero drift
\cite[Proposition VI.4]{bertoin}, and the law of large numbers applied
to $B$, we have
\[
\lim_{\eps \downarrow 0} 
\arg \inf_{u \ge S^\eps} \{ X_u^\eps + \alpha (u-S^\eps)\}
= 
\arg \inf_{u \ge S} \{ X_u + \alpha (u-S)\} .
\]
It follows readily that $D^\eps$ converges
to $D$ almost surely as $\epsilon \downarrow 0$
on the event $\{S>0\}$.

Suppose now that we are on the event 
$\{S=0\}$. Then, by Proposition~\ref{prop:order}, $0 \in \mathcal{Z}$
almost surely, and we may suppose that $X$ satisfies the conditions of
part (c) of that result, so that $G=T=S=D=0$ almost surely. 
Then, by the strong law of large numbers 
for the Brownian motion $B$, almost surely
$$
\lim_{\eps \downarrow 0} \inf_{u \le 0} \{ X_u^\eps - \alpha u \}
 = 
\lim_{\eps \downarrow 0} \inf_{u \ge 0} \{ X_u^\eps + \alpha u \}
=
0.
$$
Therefore, $D^\eps$ also converges
to $D$ almost surely as $\epsilon \downarrow 0$ on the event
$\{S = 0\}$.
\end{proof}

We are finally in a position to give the proof of Theorem \ref{thm:bigthm}. 
Suppose for the moment that $X$ has a non-zero Brownian component. 
It follows from Theorem \ref{thm:Zmeasure} that $\delta = 0$, and hence
from \eqref{E:K_size_bias_Lambda} we have that
\begin{equation}
\label{E:K_Laplace_to_Lambda}
\begin{split}
\frac{\int_{\reals_+} (1-e^{-\theta x}) \, \Lambda( dx )}
{\int_{\reals_+} x \, \Lambda(dx)}
& = 
\frac{\int_{\reals_+} \left( \int_0^\theta x e^{- \varphi x} \, d \varphi
\right)  \,  \Lambda (dx)}
{\int_{\reals_+} x \,  \Lambda(dx)}
\\
& = 
\int_0^\theta
\left(
\frac{ \int_{\reals_+} x e^{- \varphi x} \,  \Lambda(dx)}
{\int_{\reals_+} x \,  \Lambda(dx)} 
\right) 
\, d \varphi 
 = 
\int_0^\theta \E[e^{-\varphi K}] \, d \varphi.\\
\end{split}
\end{equation}

By Corollary \ref{cor:Klaplace}, this last integral is
\begin{equation}
\label{eq:bigintegral}
\begin{split}
& 4 \pi \alpha 
\int_{-\infty}^\infty
\Biggl\{
\exp \biggl(
 \int_0^\infty 
t^{-1} 
\E\Bigl[
\left(e^{i z X_t - i z \alpha t}   - 1\right)
\mathbf{1}\{X_t \ge + \alpha t\} \\
& \qquad +
\left(e^{i z X_t + i z \alpha t}   - 1\right)
\mathbf{1}\{X_t \le - \alpha t\}
\Bigr]
\, dt 
\biggr) \\
& \quad - 
\exp \biggl(
 \int_0^\infty 
t^{-1} 
\E\Bigl[
\left(e^{- \theta t + i z X_t - i z \alpha t}   - 1\right)
\mathbf{1}\{X_t \ge + \alpha t\} \\
& \quad \qquad +
\left(e^{- \theta t + i z X_t + i z \alpha t}   - 1\right)
\mathbf{1}\{X_t \le - \alpha t\}
\Bigr]
\, dt 
\biggr)
\Biggr\}
\, dz, \\
\end{split}
\end{equation}
as claimed in the theorem.

Now suppose $X$ has zero Brownian component, but satisfies the
conditions of Theorem~\ref{thm:bigthm}. Since $X$ has paths of
unbounded variation almost surely, it follows from
Remark~\ref{R:Zmeasure} that $\delta = 0$. Let $X^\eps = X + \eps
B$ and $K^\eps$ be as in Lemma \ref{lem:continuousmapping}, and let
$\Lambda^\eps$ be the L\'evy measure of the
subordinator associated with the set of points where $X^\eps$ meets
its $\alpha$-Lipschitz minorant. By Lemma
\ref{lem:continuousmapping} we know that $K^\eps \to K$ almost surely,
and thus since $\delta = 0$, arguing as in
\eqref{E:K_Laplace_to_Lambda} we have
\begin{equation}
\label{eq:limit_representation} 
\begin{split}
\frac{\int_{\reals_+} (1-e^{-\theta x}) \, \Lambda( dx )}
{\int_{\reals_+} x \, \Lambda(dx)}
& =
\int_0^\theta \E[e^{-\varphi K}] \, d \varphi \\
& =
\lim_{\eps \downarrow 0}
\int_0^\theta \E[e^{-\varphi K^\eps}] \, d \varphi
=
\lim_{\eps \downarrow 0} 
\frac{\int_{\reals_+} (1-e^{-\theta x}) \, \Lambda^\eps( dx )}
{\int_{\reals_+} x \, \Lambda^\eps(dx)} . \\
\end{split}
\end{equation}

Now, in the notation of of the proof of Corollary~\ref{cor:Klaplace} ,
it can be seen that the square integrability of the densities of
$\inf_{t \ge 0} \{ X_t + \alpha t \}$ and $\inf_{t \ge 0} \{ -X_t +
\alpha t \}$ implies that
\[
\int_{-\infty}^0 
\Bigl( \int_0^\infty f^+(\xi,h) e^{-\theta \xi} \, d \xi \Bigr) 
\Bigl( \int_0^\infty f^-(\kappa,h) e^{-\theta \kappa } \, d \kappa \Bigr) 
 \, dh 
<
\infty
\]
for all $\theta \ge 0$. Thus, by the same methods used in the proof of
Corollary~\ref{cor:Klaplace} from the last line of \eqref{eq:Klaplace1}
onwards, it follows that \eqref{eq:bigintegral} is finite. Then, since for
each fixed value of $z$ the integrand in \eqref{eq:bigintegral} is a
product of characteristic functions of certain infima, and hence not equal
to zero, we can apply Fubini's theorem to swap the order of the
integrals within the exponentials (here we are using the absolute
continuity of the distribution of $X_t$ for all $t>0$).
We now have that the integrand for each fixed value of $z$ with $X_t$
replaced by $X^\eps_t$ converges to the inegrand with just $X_t$ as
$\eps \ra 0$. Then, by finiteness of \eqref{eq:bigintegral}, we have that
\eqref{eq:bigintegral} with $X_t$ replaced by $X^\eps_t$ converges to
\eqref{eq:bigintegral}.

It remains to show that $\Lambda(\reals_+) < \infty$. 
We have from \eqref{eq:Klaplace1} that
\[
\begin{split}
\frac{\int_{\reals_+} (1-e^{-\theta x}) \, \Lambda( dx )}
{\int_{\reals_+} x \, \Lambda(dx)}
& =
2 \alpha
\Biggl[
\Biggl( \int_{-\infty}^0 
\Bigl( \int_0^\infty f^+(\xi,h)  \, d \xi \Bigr) 
\Bigl( \int_0^\infty f^-(\kappa,h) \, d \kappa \Bigr) 
 \, dh \Biggr) \\
& \quad -
\Biggl( \int_{-\infty}^0 
\Bigl( \int_0^\infty f^+(\xi,h) e^{-\theta \xi} \, d \xi \Bigr) 
\Bigl( \int_0^\infty f^-(\kappa,h) e^{-\theta \kappa } \, d \kappa \Bigr) 
 \, dh \Biggr)
\Biggr] \\
& \to
2 \alpha
\Biggl( \int_{-\infty}^0 
\Bigl( \int_0^\infty f^+(\xi,h)  \, d \xi \Bigr) 
\Bigl( \int_0^\infty f^-(\kappa,h) \, d \kappa \Bigr) 
 \, dh \Biggr)
 < \infty \\
\end{split}
\]
as $\theta \to \infty$,
and we conclude that $\Lambda(\reals_+) < \infty$. \qed

\section{Lipschitz Minorants of Brownian Motion}
\label{sec:bm}

\subsection{Williams' path decomposition for Brownian motion with drift}
\label{SS:Williams}

We recall for later use a path composition due to David Williams 
that describes the distribution of a Brownian motion
with positive drift in terms of the segment of the
path up to the time the process achieves its
global minimum and the segment of the path
after that time -- see \cite[p. 436]{rogerswills} or, for a concise
description, \cite[Section IV.5]{handbook}. 

For $\mu \in \reals$, let $Z^{(\mu)} = (Z_t^{(\mu)})_{t \ge 0}$
be a Brownian motion with drift $\mu$ started at zero.
Take $\beta > 0$ and let $E$ be a random variable that is
independent of $Z^{(-\beta)}$ and has an exponential
distribution with mean $(2 \beta)^{-1}$. Set
\[
T_E := \inf \{ t \ge 0 : Z^{(-\beta)} = - E \}. 
\]
Then, there is a diffusion $W = (W_t)_{t \ge 0}$ 
with the properties
\begin{enumerate}[(i)]
  \item  $W$ is independent of $Z^{(-\beta)}$ and $E$;
  \item  $W_0 = 0$;
  \item  $W_t > 0$ for all $t>0$ a.s.;
\end{enumerate}
such that if we define a process $(\tilde Z_t)_{t \ge 0}$ by
\eq
\lb{decomp}
\tilde Z_t :=
\begin{cases}
  Z_t^{(-\beta)}, & 0 \le t < T_E, \\
  Z_{T_E}^{(-\beta)} + W_{t-T_E}, & t \ge T_E, \\
\end{cases}
\en 
then $\tilde Z$ has the same distribution as
$Z^{(\beta)}$.
Thus, in particular, 
\eq
\lb{infdist}
- \inf\{Z^{(\beta)}_t  : t \ge 0 \} \sim \mathrm{Exp}(2 \beta)
\en
and the unique time that $Z^{(\beta)}$ achieves its
global minimum is distributed as $T_E$.
Recall also that 
\eq 
\lb{hittime} 
\E[\inf \{ t \ge 0 : Z_t^{(-\beta)} = h \}] = \frac{h}{\beta} 
\en 
for $h \le 0$ (see, for example, 
\cite[page 295, equation $2.2.0.1$]{handbook}).

\subsection{Random variables related to the Brownian Lipschitz minorant}

\begin{prop}
\label{P:K_Laplace_drift_BM}
Let $X$ be a Brownian motion with drift $\beta$, where $|\beta| <
\alpha$. Then, the distribution of $K$
is characterized by
\[
\mathbb{E}[e^{-\theta K}]
= \frac{8 \alpha ( \alpha^2 -\beta^2) \left(\frac{1}{\sqrt{2 \theta + ( \alpha +\beta)^2
   }}+\frac{1}{\sqrt{2 \theta  + ( \alpha -\beta)^2 }}\right)}{\left(\sqrt{2 \theta  +( \alpha +\beta)^2
   }+\sqrt{2 \theta  +( \alpha -\beta)^2}+2  \alpha \right)^2}
\] 
for $\theta \ge 0$, and hence $\Lambda$ is characterized by
\[
\frac
{\int_{\reals_+} (1-e^{-\theta x}) \, \Lambda(dx)}
{\int_{\reals_+} x \, \Lambda(dx)}
= \frac{4 (\alpha^2 - \beta^2) \theta}
{
\left(\sqrt{2 \theta + (\alpha - \beta)^2} + \alpha - \beta \right)
\left(\sqrt{2 \theta + (\alpha + \beta)^2} + \alpha + \beta \right)
}
\]
for $\theta \ge 0$.
\end{prop}

\begin{proof}
We have from \cite[page 269, equation $1.14.3(1)$]{handbook} that
\[
\int_{-\infty}^0 f^-(\xi,h) e^{-\theta \xi} \, d \xi = 2(\alpha-\beta)e^{h(\sqrt{2 \theta + (\alpha - \beta)^2}+(\alpha -\beta))}
\]
and
\[
\int_0^\infty f^+(\xi,h) e^{-\theta \xi} \, d \xi 
= 2(\alpha+\beta)e^{h(\sqrt{2 \theta + (\alpha + \beta)^2}+(\alpha+\beta))}.
\]
Thus, from \eqref{eq:Klaplace1},
\[
\begin{split}
\mathbb{E}[e^{-\theta K}]
&  
= - 2 \alpha \frac{d}{d \theta} \Biggl( \int_{-\infty}^0 
4(\alpha^2-\beta^2)e^{h(\sqrt{2 \theta + (\alpha + \beta)^2}+\sqrt{2 \theta + (\alpha - \beta)^2}+2\alpha)}
 \, dh \Biggr) \\
\\
&  
= 8 \alpha (\alpha^2-\beta^2) \frac{d}{d \theta} \left( 
\frac{1}{\sqrt{2 \theta + (\alpha + \beta)^2}+\sqrt{2 \theta + (\alpha - \beta)^2}+2\alpha}
 \right) \\
&
= \frac{8 \alpha ( \alpha^2 -\beta^2) \left(\frac{1}{\sqrt{2 \theta + ( \alpha +\beta)^2
   }}+\frac{1}{\sqrt{2 \theta  + ( \alpha -\beta)^2 }}\right)}{\left(\sqrt{2 \theta  +( \alpha +\beta)^2
   }+\sqrt{2 \theta  +( \alpha -\beta)^2}+2  \alpha \right)^2},
\\
\end{split}
\] 
as required.

Now, by \eqref{E:K_Laplace_to_Lambda},
\[
\begin{split}
\frac
{\int_{\reals_+} (1-e^{-\theta x}) \, \Lambda(dx)}
{\int_{\reals_+} x \, \Lambda(dx)}
& = \int_0^\theta \E[e^{-\varphi K}] \, d\varphi \\
& = 8 \alpha (\alpha^2-\beta^2)
\biggl[
\frac{1}{\sqrt{(\alpha + \beta)^2}+\sqrt{(\alpha - \beta)^2}+2\alpha} \\
& \quad -
\frac{1}{\sqrt{2 \theta + (\alpha + \beta)^2}+\sqrt{2 \theta + (\alpha - \beta)^2}+2\alpha}
\biggr] \\
& = 
8 \alpha (\alpha^2-\beta^2)
\biggl[
\frac{1}{4 \alpha}
 \\
& \quad - 
\frac{1}{\sqrt{2 \theta + (\alpha + \beta)^2}+\sqrt{2 \theta + (\alpha - \beta)^2}+2\alpha}
\biggr] \\
& = \frac{4 (\alpha^2 - \beta^2) \theta}
{
\left(\sqrt{2 \theta + (\alpha - \beta)^2} + \alpha - \beta \right)
\left(\sqrt{2 \theta + (\alpha + \beta)^2} + \alpha + \beta \right)
}\\
\end{split} 
\]
after a little algebra.
\end{proof}

\begin{remark}
There is an alternative way to verify that the Laplace transform
for $K$ presented in Proposition~\ref{P:K_Laplace_drift_BM} is correct.  Recall from the proof of Theorem~\ref{thm:Zregen}
that $D = S + \tilde T$, where the independent random variables $S$
and $\tilde T$ are defined by
\[
S
= 
\inf \left\{ s > 0 : X_s - \alpha s
=
\inf\{ X_u  - \alpha u : u \le 0\}  \right\}
\]
and
\[
\tilde{T} 
= \sup \{ t \ge 0 : \tilde{X}_t = \inf\{\tilde{X}_s : s \ge 0\} \}
\]
with
\[
(\tilde{X}_s)_{s \ge 0} := \left( (X_{S+s} - X_S) + \alpha s
\right)_{s \geq 0}.
\]

Set $I^- := \inf\{ X_u  - \alpha u : u \le 0\}$.
Because $(X_{-t} + \alpha t)_{t \ge 0}$ is a Brownian motion with
drift $\alpha - \beta$, we  know from Subsection~\ref{SS:Williams} that 
$-I^-$ has an exponential distribution with
mean $(2(\alpha - \beta))^{-1}$.  Now $(X_t - \alpha t)_{t \ge 0}$
is a Brownian motion with drift $\beta - \alpha$, and so, again from  Subsection~\ref{SS:Williams}, $S$ is distributed as the time until
this process achieves its global minimum.  It follows that
\[
\E[e^{-\theta S}]
=
\frac{2(\alpha-\beta)}{\sqrt{2 \theta + (\alpha - \beta)^2} + \alpha - \beta}
\]
and
\[
\E[e^{-\theta \tilde T}]
=
\frac{2(\alpha+\beta)}{\sqrt{2 \theta + (\alpha + \beta)^2} + \alpha + \beta}
\]
-- see, for example, \cite[page 266, equation $1.12.3(2)$]{handbook}.

By stationarity, $D$ has the same distribution as $U(D-G) = UK$, where
$U$ is an independent random variable that is uniformly distributed on $[0,1]$.
Thus, 
\[
\E[e^{-\theta D}]
=
\int_0^1 \E[e^{- u \theta K}] \, du
=
\E\left[\frac{1}{\theta K} \left(1 - e^{- \theta  K}\right)\right]
=
\frac{1}{\theta} 
\frac
{\int_{\reals_+} (1-e^{-\theta x} ) \, \Lambda(dx)}
{\int_{\reals_+} x \, \Lambda(dx)},
\]
and
\[
\begin{split}
\frac
{\int_{\reals_+} (1-e^{-\theta x}) \, \Lambda(dx)}
{\int_{\reals_+} x \, \Lambda(dx)}
& =
\theta \E[e^{-\theta D}] \\
& = 
\theta \E[e^{-\theta S}] \E[e^{-\theta \tilde T}] \\
& =
\frac{4 (\alpha^2 - \beta^2) \theta}
{
\left(\sqrt{2 \theta + (\alpha - \beta)^2} + \alpha - \beta \right)
\left(\sqrt{2 \theta + (\alpha + \beta)^2} + \alpha + \beta \right)
}.\\
\end{split} 
\]
This equality agrees with the one found in
Proposition~\ref{P:K_Laplace_drift_BM}.  Differentiating
the expression on the right with respect to $\theta$ and recalling the
observation \eqref{E:K_Laplace_to_Lambda}, we arrive at the
the expression for the Laplace transform of $K$ 
in Proposition~\ref{P:K_Laplace_drift_BM}.
\end{remark}

\begin{prop}
Let $X$ be a Brownian motion with zero drift. Then,
\[
\P\{K \in d \kappa\} =
\left(
\frac{4 \alpha^3 }{ \sqrt{2 \pi} } \kappa^{1/2} e^{- \alpha^2 \kappa /2} - 4 \alpha^4 \kappa \Phi(- \alpha \kappa^{1/2})
\right) \, d\kappa,\\
\]
where $\Phi$ is the standard normal cumulative distribution function.
Thus, 
\[
\frac{\Lambda(dx)}{\Lambda(\reals_+)}
=
\frac{2 \alpha }{ \sqrt{2 \pi} } x^{-1/2} e^{- \alpha^2 x /2} - 2 \alpha^2 \Phi(- \alpha x^{1/2}) 
\] 
\end{prop}

\begin{proof}
We have from \cite[page 269, equation $1.14.4(1)$]{handbook} that
\[
f^-(\xi,h) = f^+(\xi,h)
=
\frac{- 2 \alpha h }{ \sqrt{2\pi} \xi^{3/2} }  \exp \left\{ - \frac{( \alpha \xi -h)^2}{ 2 \xi} \right\}.
\]
Thus, by Proposition~\ref{prop:tlrdensity2},
\[
\begin{split}
\frac{\P\{ K \in d \kappa\} }{d\kappa}
& = 
\frac{ 4 \alpha^3 \kappa  e^{- \alpha^2 \kappa / 2 } }{ \pi }
\int_0^\kappa \int_{-\infty}^0 
\frac{ h^2 }{ \xi^{3/2} (\kappa-\xi)^{3/2} }  \exp \left\{ 2 \alpha h - \frac{ \kappa h^2}{ 2 \xi (\kappa - \xi)  } \right\} 
\, dh \, d \xi \\
& = 
\frac{ 4 \alpha^3  e^{- \alpha^2 \kappa / 2 } }{ \pi \kappa }
\int_{-\infty}^0 h^2 e^{ 2 \alpha h} 
\left( \int_0^1 
\frac{ 1 }{ \xi^{3/2} (1-\xi)^{3/2} }  \exp \left\{ - \frac{ h^2/2 \kappa }{  \xi (1 - \xi)  } \right\} \, d \xi \right)
\, dh.  \\
\end{split}
\] 
The change of variable $y = \frac{1}{\xi(1-\xi)} -4$ gives that
\[
\int_0^{1/2} \frac{ 1 }{ \xi^{3/2} (1-\xi)^{3/2} }  \exp \left\{ - \frac{ c }{  \xi (1 - \xi)  } \right\} \, d \xi = e^{-4c} \int_0^\infty z^{-1/2} e^{- c z} dz 
= \frac{e^{-4c} \sqrt{\pi}}{ \sqrt{c}}
\]
for any $c>0$,
and hence
\[
\begin{split}
\frac{\P\{ K \in d \kappa\}}{d\kappa} 
& =
\frac{ 4 \alpha^3 e^{- \alpha^2 \kappa / 2 } }{ \pi \kappa }
\int_{-\infty}^0 h^2 e^{2 \alpha h }
\frac{2 e^{-2h^2 / \kappa } \sqrt{\pi}}{ \sqrt{ h^2/2 \kappa }}
\, dh  \\
& =
- \frac{ 8 \sqrt{2} \alpha^3 e^{- \alpha^2  \kappa / 2 } }{ \sqrt{\pi \kappa} }
\int_{-\infty}^0 h  e^{2 \alpha h - 2h^2 / \kappa } 
\, dh.  \\
\end{split}
\] 
The further change of variable $z = 2 \kappa^{-1/2} h -  \alpha \kappa^{1/2} $ leads to
\[
\begin{split}
\frac{\P\{K \in d \kappa\}}{d\kappa}
& =
- 4 \alpha^3
\int_{-\infty}^{- \alpha \kappa^{1/2} } (\kappa^{1/2} z + \alpha \kappa  ) \frac{1}{ \sqrt{2 \pi} }  e^{- z^2 / 2  } 
\, dz  \\
& =
- 4 \alpha^3
\left( - \frac{\kappa^{1/2} }{ \sqrt{2 \pi} } e^{- \alpha^2 \kappa /2} + \alpha \kappa \Phi(- \alpha \kappa^{1/2} ) \right)  \\
& =
\frac{4 \alpha^3 }{ \sqrt{2 \pi} } \kappa^{1/2} e^{- \alpha^2 \kappa /2} - 4 \alpha^4 \kappa \Phi(- \alpha \kappa^{1/2}).\\
\end{split}
\] 

Because $\Lambda(dx)$ is proportional to 
$x^{-1} \P\{K \in dx\}$,
we need only find $\int_{\reals_+} x^{-1} \P\{K \in dx\}$
to establish the claim for $\Lambda$,
and this can be done
using methods of integration similar to those used in
Remark~\ref{R:K_density_integral} below to check
that the density of $K$ integrates to one.
\end{proof}

\begin{remark} 
\label{R:K_density_integral}
We can check directly that the density given for $K$ integrates to one. 
For the first term, we use the substitution $\eta = \alpha^2 \kappa/2$, and
for the second we use the substitution $\eta = \alpha^2 \kappa$ and
then change the order of integration to get that the integral
of the claimed density is
\[
\begin{split}
& \frac{4 }{ \Gamma(3/2) } \int_0^{\infty} \eta^{1/2} e^{- \eta} \, d \eta - 4 \int_0^\infty \eta \Phi(- \eta^{1/2}) \, d \eta \\
& \quad =
4 - \frac{4}{\sqrt{2 \pi} } \int_0^\infty \int_{\eta^{1/2}}^\infty \eta e^{-y^2/2} \, d y \, d \eta \\
& \quad =
4 - \frac{4}{\sqrt{2 \pi} } \int_0^\infty \left( \int_0^{y^2} \eta \, d \eta \right) e^{-y^2/2} \, d y \\
& \quad = 
4 - \frac{2}{\sqrt{2 \pi} } \int_0^\infty y^4 e^{- y^2/2} \, dy \\
& \quad =  
4 - \frac{3}{ \Gamma(5/2) } \int_0^\infty x^{3/2} e^{- x} \, \, dx \,
\, = \, 1. \\
\end{split}
\]
\end{remark}

\begin{prop}
\label{prop:Hlaplace}
Let $X$ be a Brownian motion with
drift $\beta$, where $|\beta| < \alpha$.
Recall that $T := \arg \max \{M_t : G \le t \le D\}$
and
$H := X_T - M_T $.
Then, $H$ has a  $\mathrm{Gamma}(2, 4 \alpha)$ distribution;
that is, the distribution of $H$ is absolutely continuous with
respect to Lebesgue measure with density
$h \mapsto (4 \alpha)^2 h e^{-4 \alpha h}$, $h \ge 0$.
Also,
\[
\P\{T>0\} = \frac{1}{2} \left( 1 + \frac{\beta}{\alpha} \right) ,
\]
and the distribution of $T$ is characterized by
\[
\E\left[e^{-\theta T}\right] = 8 \alpha (\alpha^2-\beta^2) \frac{1}{\theta} \left( \frac{1}{ \sqrt{(\alpha+\beta)^2 - 2 \theta} + 3 \alpha - \beta  } - \frac{1}{ \sqrt{(\alpha-\beta)^2 + 2 \theta} + 3 \alpha + \beta  } \right)
\]
for $- \frac{(\alpha-\beta)^2}{2} \le \theta \le \frac{(\alpha+\beta)^2}{2}$.
\end{prop}

\begin{proof} 
Consider the claim regarding the
distribution of $H$.  A slight elaboration of the proof of
Proposition~\ref{prop:tlrdensity2} shows, in the notation of that result,
that the random vector $(T,L,R,-H)$ has a distribution that is 
absolutely continuous with respect to Lebesgue measure with joint density
$(\tau,\lambda,\rho,\eta) \mapsto 2 \alpha f^-(\lambda,\eta) f^+(\rho,\eta)$,
$\lambda,\rho>0$, $\tau - \lambda < 0 < \tau + \rho$, $\eta < 0$.
Therefore,
\[
\P\{H \in dh\}
=
2  \alpha 
\int_0^\infty \int_0^\infty (\lambda + \rho) f^-(\lambda,-h) f^+(\rho,-h)
\, d\lambda d\rho d\eta.
\]

By \eqref{infdist},
\begin{equation}
\label{E:f-marg1}
\int_0^\infty f^-(\lambda, -h) \, d \lambda 
= 2(\alpha-\beta)e^{-2(\alpha-\beta) h}.
\end{equation}
Combining this with \eqref{hittime} gives
\begin{equation}
\label{E:f-marg2}
\int_{-\infty}^0 \lambda f^-(\lambda, -h) \, d \lambda 
= \frac{-\eta}{\alpha - \beta} \times 2(\alpha-\beta)e^{-2(\alpha-\beta) h} 
= 2 h \eta e^{-2(\alpha-\beta) h}.
\end{equation}
Similarly,
\begin{equation}
\label{E:f+marg1}
\int_0^\infty f^+(\rho, -h) \, d \rho 
= 2(\alpha+\beta)e^{-2(\alpha+\beta) h}
\end{equation}
and
\begin{equation}
\label{E:f+marg2}
\int_{-\infty}^0 \rho f^+(\rho, \eta) \, d \rho  
= 2 h e^{-2(\alpha+\beta) h}.
\end{equation}
Thus,
\[
\begin{split}
\P\{H \in dh\}
& =
2  \alpha 
\biggl[
2 h e^{-2(\alpha-\beta) h}
\times
2(\alpha+\beta)e^{-2(\alpha+\beta) h} \\
& \quad +
2(\alpha-\beta)e^{-2(\alpha-\beta) h}
\times
2 h e^{-2(\alpha+\beta) h}
\biggr] \, dh \\
& =
(4 \alpha)^2 h e^{-4 \alpha h} \, dh. \\
\end{split}
\]

Note that $T>0$ if and only if $I^+ > I^-$,
where
\[
I^+ := \inf\{X_t + \alpha t : t \ge 0\}
\quad
\text{ and }
\quad
I^- := \inf\{X_t - \alpha t : t \le 0\}.
\]
Recall from Subsection~\ref{SS:Williams} that the independent
random variables $I^+$ and $I^-$ are exponentially
distributed with respective means $(2(\alpha + \beta))^{-1}$
and $(2(\alpha - \beta))^{-1}$.  It follows that
\[
\P\{T > 0\} 
= 
\frac{2(\alpha + \beta)}{2(\alpha + \beta) + 2(\alpha - \beta)}
=
\frac{1}{2} \left( 1 + \frac{\beta}{\alpha} \right).
\] 

We can also derive this last result from Proposition~\ref{prop:tlrdensity2}
as follows.
\[
\begin{split}
\P\{T>0\}
& = 2 \alpha \int_{-\infty}^0 \int_0^\infty  \int_{-\infty}^0 \int_\tau^\infty
f^-(\tau-\gamma,h) f^+(\delta-\tau,h) \,
d \delta d \gamma  d\tau dh \\
& = 2 \alpha 
\int_{-\infty}^0 \int_0^\infty  \int_{-\infty}^0 
f^-(\tau-\gamma,h) \left( \int_0^\infty  f^+(\eta,h) \, d\eta \right)
dh  d \tau d \gamma \\
& = 2 \alpha \int_{-\infty}^0 
\left( \int_0^\infty \int_{-\infty}^0  f^+(\tau-\gamma,h) \, d \gamma \, d
  \tau \right)
\left( \int_0^\infty  f^+(\eta,h) \, d\eta \right)
\, dh\\
& = 2 \alpha \int_{-\infty}^0 
\left( \int_0^\infty \eta  f^-(\eta,h) \, d\eta \right)
\left( \int_0^\infty  f^+(\eta,h) \, d\eta \right)
\, dh.\\  
\end{split}
\]
Substituting in \eqref{E:f-marg2} and \eqref{E:f+marg1}, and then
evaluating the resulting straightforward integral establishes the result.

The Laplace transform of $T$ may be calculated using very similar
methods.
\end{proof}

\section{Some facts about Lipschitz minorants}
\label{sec:facts}

The following is a restatement of \eqref{mformula}  
accompanied by a proof.

\begin{lem} 
\label{L:minorant_explicit}
Suppose that the function $f: \reals \to \reals$ has $\alpha$-Lipschitz
minorant $m: \reals \to \reals$.  Then,
\[
\begin{split}
m(t) & = \sup \{ h \in \reals : h - \alpha|t-s| \leq f(s)  \text{ for all } s \in \reals \} \\
& = \inf \{f(s) + \alpha |t-s| : s \in \reals\}. \\
\end{split}
\]
\end{lem}

\begin{figure}
\begin{center}
\includegraphics[width=2.5in]{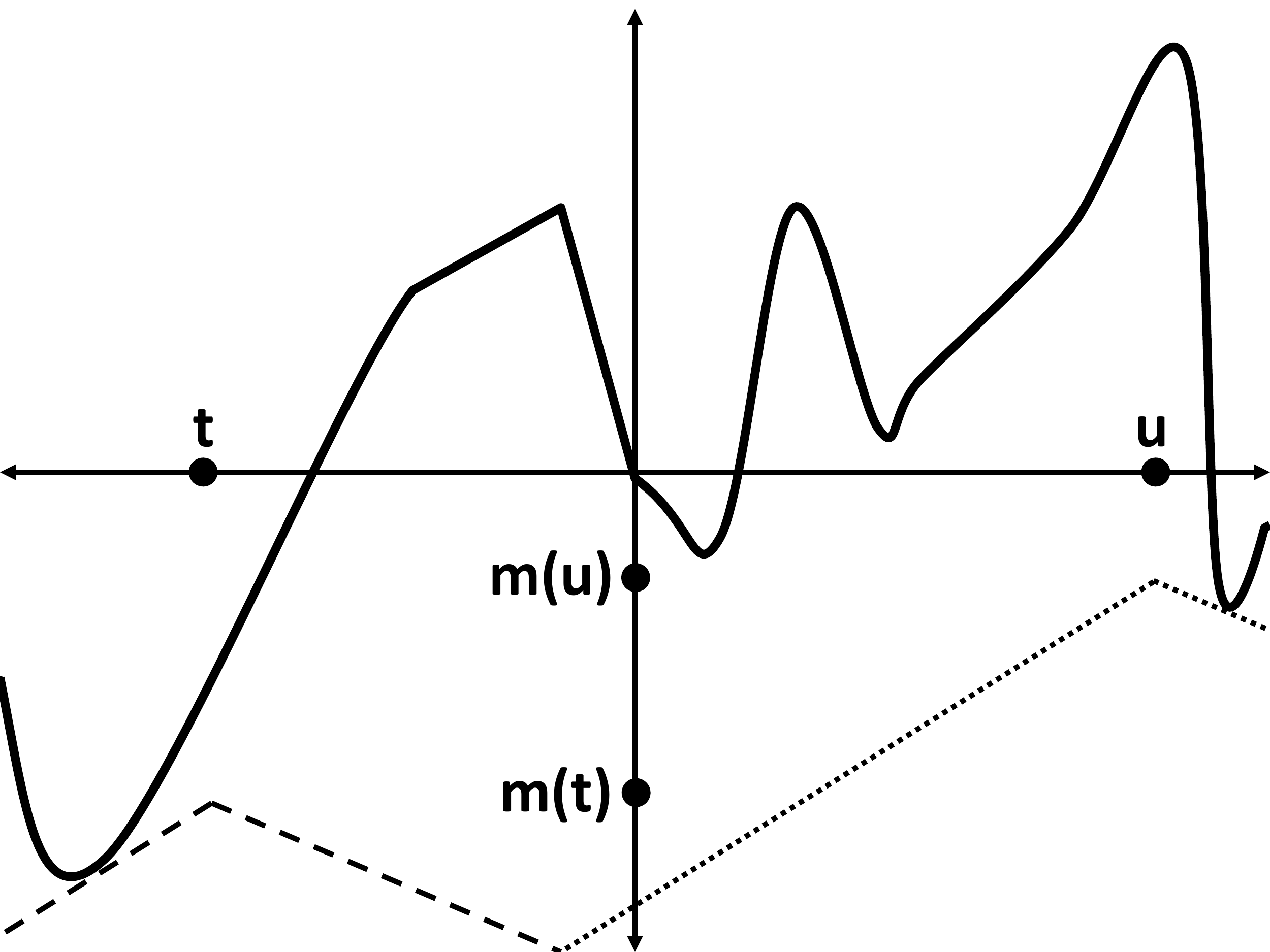}
\end{center}
\caption{Lemma~\ref{L:minorant_explicit} shows that the height of the
  $\alpha$-Lipschitz minorant of a function $f$ at a fixed time $t$ is given by $\sup \{ h \in \reals : h - \alpha|t-s| \leq f(s)  \text{ for all } s \in \reals \} $.}
\label{fig:height}
\end{figure}

\begin{proof}
Consider the first equality.  Fix $t \in \reals$. Because $m$ is $\alpha$-Lipschitz, if $h \le m(t)$, then
$h - \alpha |t-s| \le m(t) - \alpha |t-s| \le m(s) \le f(s)$
for all $s \in \reals$.  On the other hand, if $h > m(t)$, then
$s \mapsto (h - \alpha |t-s|) \vee m(s)$ is an $\alpha$-Lipschitz
function that dominates $m$ (strictly at $t$), and so 
$(h - \alpha |t-s|) \vee m(s) > f(s)$ for some $s \in \reals$.
This implies that $h - \alpha |t-s|  > f(s)$, since $m(s) \le f(s)$.
The second equality is simply a rephrasing of the first.
\end{proof}

We leave the proof of the following straightforward consequence of
Lemma~\ref{L:minorant_explicit} to the reader.

\begin{cor}
\label{C:flow}
Suppose that the function $f: \reals \to \reals$ has $\alpha$-Lipschitz
minorant $m: \reals \to \reals$.  Define functions
$f^\leftarrow: \reals \to \reals$
and
$f^\rightarrow: \reals \to \reals$
by
\[
f^\leftarrow(t)
:=
\begin{cases}
f(t),& t < 0, \\
m(0) - \alpha t,& t \ge 0,
\end{cases}
\]
and
\[
f^\rightarrow(t)
:=
\begin{cases}
m(0) + \alpha t,& t \le 0, \\
f(t),& t > 0.
\end{cases}
\]
Denote the $\alpha$-Lipschitz minorants of $f^\leftarrow$ and $f^\rightarrow$
by $m^\leftarrow$ and $m^\rightarrow$, respectively.  Then,
$m^\leftarrow(t) = m(t)$ for all $t \le 0$ and $m^\rightarrow(t) = m(t)$ for all $t \ge 0$.
\end{cor}

The next result says that if $f$ is a c\`adl\`ag function with
$\alpha$-Lipschitz minorant $m$, then on an open interval
in the complement of the closed set 
$\{t \in \reals: m(t) = f(t) \wedge f(t-)\}$ the graph
of the function $m$ is either a straight line or a ``sawtooth''.

\begin{lem}
\label{L:sawtooth}
Suppose that 
$f: \reals \to \reals$ be a c\`adl\`ag function with
$\alpha$-Lipschitz minorant $m : \reals \to \reals$.
The set $\{t \in \reals: m(t) = f(t) \wedge f(t-)\}$ is closed.
If $t' < t''$ 
are such that $f(t') \wedge f(t'-) = m(t')$,
$f(t'') \wedge f(t''-) = m(t'')$, and $f(t) \wedge f(t-) > m(t)$
for $t' < t < t''$, then, setting
$t^* = (f(t'') \wedge f(t''-) - f(t') \wedge f(t'-) + \alpha (t'' + t'))/(2 \alpha)$,
\[
m(t) 
= 
\begin{cases}
f(t') \wedge f(t'-) + \alpha(t - t'), 
& t' \le t \le t^*, \\
f(t'') \wedge f(t''-) + \alpha(t'' - t),
& t^* \le t \le t''. \\
\end{cases}
\]
\end{lem}

\begin{proof}
We first show that the set 
$\{t \in \reals: m(t) = f(t) \wedge f(t-)\}$ is closed
by showing that its complement is open.  Suppose
$t$ is in the complement, so that 
$f(t) \wedge f(t-) - m(t) =: \epsilon > 0$.  Because
$f$ is c\`adl\`ag and $m$ is continuous, there exists
$\delta > 0$ such that if $|s - t| < \delta$, then
$f(s) > f(t) \wedge f(t-) - \epsilon/3$ and
$m(s) < m(t) + \epsilon/3$. 
Hence, $f(s-) \ge f(t) \wedge f(t-) - \epsilon/3$
and $f(s) \wedge f(s-) - m(s) > \epsilon/3$
for $|s - t| < \delta$, showing that a neighborhood
of $t$ is also in the complement.

Turning to the second claim,
define a function $\tilde m : \reals \to \reals$ by
\[
\tilde m(t) 
:= 
\begin{cases}
f(t') \wedge f(t'-) + \alpha(t - t'), 
&  t \le t^*, \\
f(t'') \wedge f(t''-) + \alpha(t'' - t),
& t^* \le t. \\
\end{cases}
\]
That is, $\tilde m(t) = h^* - \alpha|t - t^*|$, where
\[
h^*
=
(f(t'') \wedge f(t''-) + f(t') \wedge f(t'-) + \alpha (t'' - t'))/2.
\]

Because 
$m(t') = \tilde m(t')$
$m(t'') = \tilde m(t'')$, and 
$m$ is $\alpha$-Lipschitz, we have $m(t) \le \tilde m(t)$
for $t \in [t',t'']$ and $m(t) \ge \tilde m(t)$ for 
$t \notin [t',t'']$.  Suppose for some $t_0 \in (t',t'')$
that $m(t_0) < \tilde m(t_0)$.  We must have that
$m(t_0) - \alpha|t'-t_0| \le m(t') \le f(t') \wedge f(t'-)$ 
and
$m(t_0) - \alpha|t''-t_0| \le m(t'') \le f(t'') \wedge f(t''-)$.
Moreover, both of these inequalities must be strict, because
otherwise we would conclude that $m(t_0) \ge \tilde m(t_0)$.

We can therefore choose $\epsilon > 0$ sufficiently small so that
$m(t_0) + \epsilon - \alpha |t - t_0| < f(t) \wedge f(t-)$ for
$t \in [t',t'']$.  This implies that 
$m(t_0) + \epsilon - \alpha |t - t_0| < \tilde m(t) \le m(t) \le
f(t) \wedge f(t-)$ for $t \notin [t',t'']$.  Thus,
$t \mapsto (m(t_0) + \epsilon - \alpha |t - t_0|) \vee m(t)$
is an $\alpha$-Lipschitz function that is dominated 
everywhere by $f$
and strictly dominates $m$ at the point $t_0$, contradicting
the definition of $m$.
\end{proof}

We have a recipe for finding 
$\inf \{ t>0 : f(t) \wedge f(t-) = m(t) \}$ when
$f$ is a c\`adl\`ag function with
$\alpha$-Lipschitz minorant $m$. Figure~\ref{fig:recipe} gives two
examples of how the recipe applies to different paths (note that the
value of $\alpha$ differs for the two examples).

\begin{lem}
\label{lem:recipe}
Let $f: \reals \to \reals$ be a c\`adl\`ag function with
$\alpha$-Lipschitz minorant $m : \reals \to \reals$. 
Set 
\[
\mathbf{d} := \inf \{ t>0 : f(t) \wedge f(t-) = m(t) \}, 
\]
\[
\mathbf{s} := 
\inf \left\{  
t > 0 : f(t) \wedge f(t-) -  \alpha t
\leq 
\inf\{ f(u) - \alpha u : u \leq 0 \}  
\right\},
\]
and
\[
\mathbf{e} := 
\inf \left\{ t \ge \mathbf{s} : f(t) \wedge f(t-) + \alpha (t-\mathbf{s}) 
= 
\inf \{ f(u) + \alpha (u-\mathbf{s}) : u \geq \mathbf{s}\}  \right\}.
\]
Suppose that
$f(\mathbf{s}) \le f(\mathbf{s}-)$.
Then, $\mathbf{e}=\mathbf{d}$.
\end{lem}

\begin{figure}
\begin{center}$
\begin{array}{cc}
\includegraphics[width=2.5in]{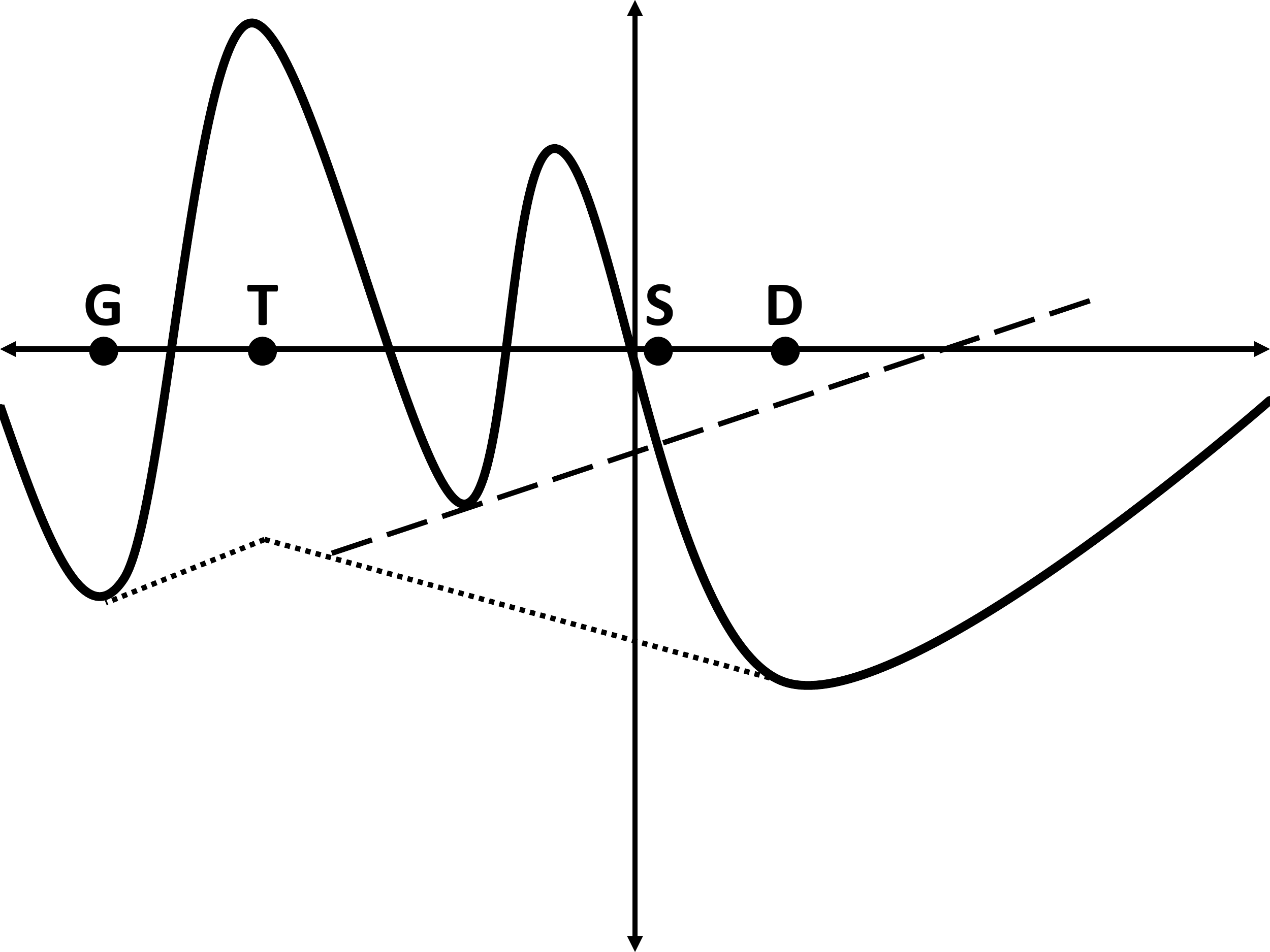} &
\includegraphics[width=2.5in]{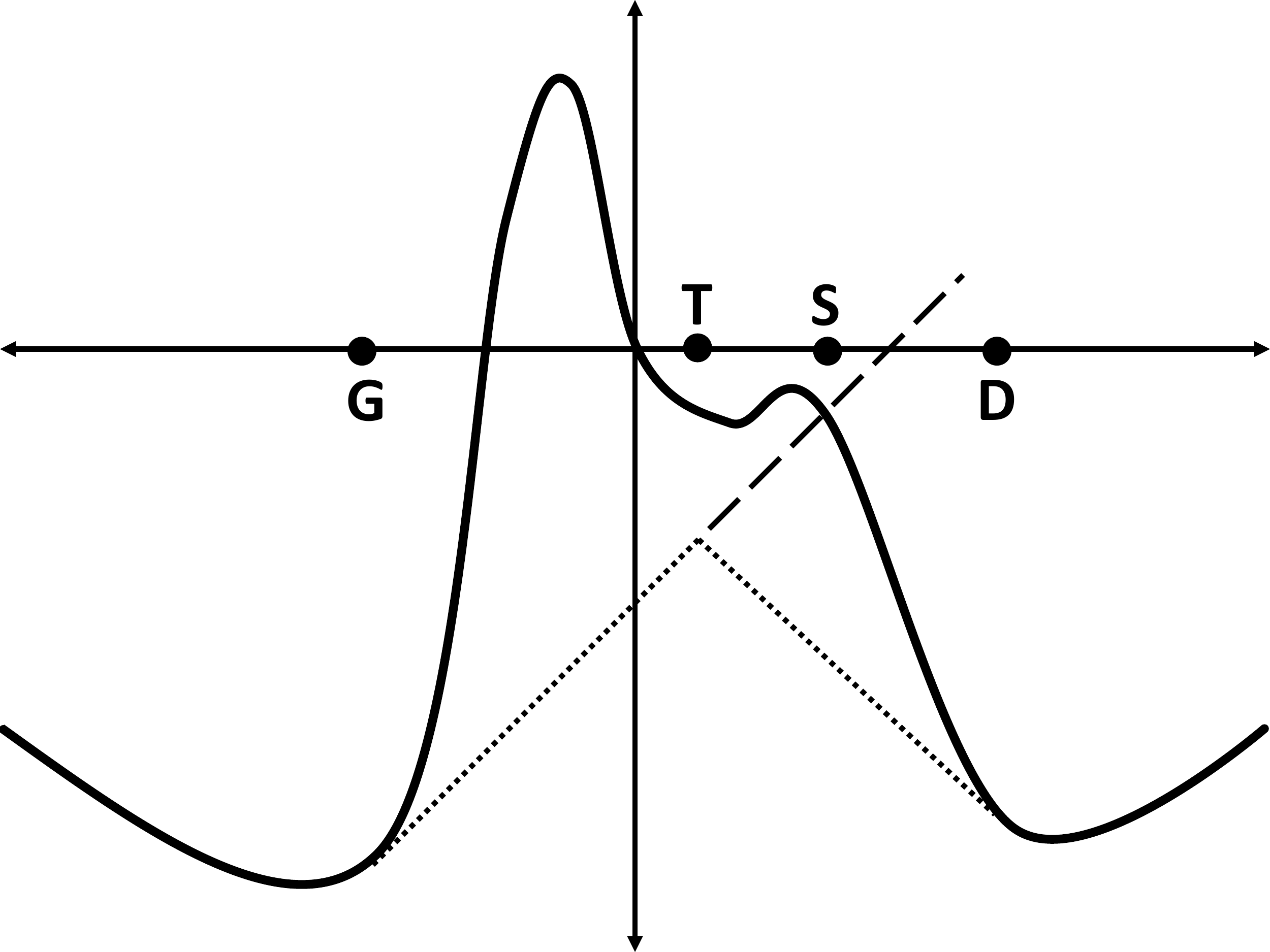}
\end{array}$
\end{center}
\caption{Two instances of the construction of
  Lemma~\ref{lem:recipe}.}
\label{fig:recipe}
\end{figure}

\begin{proof} 
It suffices to show the following:
\begin{equation}
\label{E:determin1}
f(t) \wedge f(t-) > m(t) \; \text{for} \; 0 < t < \mathbf{e},
\end{equation}
\begin{equation}
\label{E:determin2}
f(\mathbf{e}) \wedge f(\mathbf{e}-) \le m(\mathbf{e}),
\end{equation}
\begin{equation}
\label{E:determin3}
\mathbf{d} > 0 \Longrightarrow \mathbf{e} > 0.
\end{equation}

For $0 < t < \mathbf{s}$, it follows from the definition of $\mathbf{s}$ that
\[
\begin{split}
f(t) \wedge f(t-) 
& > 
\inf\{ f(u) - \alpha u : u \leq 0 \} + \alpha t\\
& = 
\inf\{ f(u) + \alpha (t-u) : u \leq 0 \} \\
& \ge
\inf\{ f(u) + \alpha |t-u| : u \in \reals \}
= m(t). \\
\end{split}
\]
For $\mathbf{s} \le t < \mathbf{e}$, it follows from the definition of $\mathbf{e}$ that
\[
f(t) \wedge f(t-) + \alpha (t-\mathbf{s}) > \inf \{ f(u) + \alpha (u-\mathbf{s}) : u \geq \mathbf{s}\},
\] 
and hence 
\[
\begin{split}
f(t) \wedge f(t-) 
& >
\inf\{f(u) + \alpha (u-\mathbf{s}) : u \ge \mathbf{s}\} - \alpha(t-\mathbf{s}) \\
& =  
\inf\{f(u) + \alpha (u-t) : u \ge \mathbf{s}\} \\
& \ge
\inf\{ f(u) + \alpha |t-u| : u \in \reals \}
= m(t). \\
\end{split}
\]
This completes the proof of \eqref{E:determin1}

Now 
$f(\mathbf{e}) \wedge f(\mathbf{e}-)  + \alpha(\mathbf{e} - \mathbf{s}) = \inf\{ f(u) + \alpha (u - \mathbf{s}) : u \ge \mathbf{s}\} $ ,
and so
$f(\mathbf{e}) \wedge f(\mathbf{e}-)  = \inf\{ f(u) + \alpha (u - \mathbf{e}) : u \ge \mathbf{s}\} $ 
This certainly gives 
\begin{equation}
\label{E:first_interval}
f(\mathbf{e}) \wedge f(\mathbf{e}-)  \le \inf\{ f(u) + \alpha |\mathbf{e} - u| : u \ge \mathbf{s}\}.
\end{equation}
Combined with the definition of $\mathbf{s}$,
it also gives
\[
\begin{split}
f(\mathbf{e}) \wedge f(\mathbf{e}-)  + \alpha(\mathbf{e} - \mathbf{s})
& \le 
f(\mathbf{s})  + \alpha(\mathbf{s} - \mathbf{s}) \\
& \le 
\inf\{ f(s) - \alpha s : s \leq 0 \} + \alpha \mathbf{s}. \\
\end{split}
\]
Thus,
$f(\mathbf{e}) \wedge f(\mathbf{e}-)  + 2 \alpha(\mathbf{e} - \mathbf{s})  \le \inf\{ f(s) + \alpha (\mathbf{e} - s) : s \leq 0 \} 
$
and hence, {\em a fortiori},
\begin{equation}
\label{E:second_interval}
f(\mathbf{e}) \wedge f(\mathbf{e}-)  
\le \inf\{ f(s) + \alpha |\mathbf{e} - s| : s \leq 0 \}.
\end{equation}

For $0 < s < \mathbf{s}$,
$f(s)  - s
>
\inf\{ f(r) - \alpha r : r \leq 0 \}$,
and so
\begin{equation}
\label{E:third_interval}
\begin{split}
\inf\{f(s) + \alpha |\mathbf{e} - s| : 0 \le s < \mathbf{s}\}
& =
\inf\{f(s) + \alpha (\mathbf{e} - s) : 0 \le s < \mathbf{s}\} \\
& =
\inf\{f(s) - \alpha s : 0 \le s < \mathbf{s}\} + \alpha \mathbf{e} \\
& \ge
\inf\{ f(r) - \alpha r : r \leq 0 \} + \alpha \mathbf{e} \\
& =
\inf\{ f(r) + \alpha(\mathbf{e} - r) : r \leq 0 \} \\
& =
\inf\{ f(r) + \alpha|\mathbf{e} - r| : r \leq 0 \}. \\
\end{split}
\end{equation}

Combining \eqref{E:first_interval}, \eqref{E:second_interval}
and \eqref{E:third_interval} gives
\eqref{E:determin2}.  

The proof of \eqref{E:determin3}
is a straightforward consequence of Lemma~\ref{L:sawtooth}
and we leave it to the reader.
\end{proof}

\begin{cor}
\label{C:t_s_d}
Let
$f: \reals \to \reals$ be a c\`adl\`ag function with
$\alpha$-Lipschitz minorant $m : \reals \to \reals$.
Define 
$\mathbf{d}$, $\mathbf{s}$, and $\mathbf{e}$ as in
Lemma~\ref{lem:recipe}.  Assume that 
$f(\mathbf{s}) \le f(\mathbf{s}-)$, so that
$\mathbf{e}=\mathbf{d}$.  
Put $\mathbf{g}  := \sup \{ t<0 : f(t) \wedge f(t-) = m(t) \}$
and assume that $f(0) \wedge f(0-) > m(0)$, so that
$f(t) \wedge f(t-) > m(t)$ for $t \in (\mathbf{g},\mathbf{d})$.
Let
$\mathbf{t} := (f(\mathbf{d}) \wedge f(\mathbf{d}-) - f(\mathbf{g}) \wedge f(\mathbf{g}-) 
+ \alpha (\mathbf{d} + \mathbf{g}))/(2 \alpha)$
be the point in $[\mathbf{g}, \mathbf{d}]$ at which the function $m$
achieves its maximum.
Then,
$\mathbf{g} \le \mathbf{t} \le \mathbf{s} \le \mathbf{d}$.
Moreover, if $\mathbf{t}=\mathbf{s}$, then 
$\mathbf{t}=\mathbf{s}=\mathbf{d}$.
\end{cor}

\begin{proof}  
We first show that 
$\mathbf{g} \le \mathbf{t} \le \mathbf{s} \le \mathbf{d}$.
We certainly have $\mathbf{g} \le \mathbf{s} \le \mathbf{d}$
and $\mathbf{g} \le \mathbf{t} \le \mathbf{d}$, so it suffices to
prove that $\mathbf{t} \le \mathbf{s}$.  Because $\mathbf{s} \ge 0$, 
this is clear when $\mathbf{t} < 0$,
so it further suffices to consider the case where $\mathbf{t} \ge 0$.  Suppose, then,
that $\mathbf{g} \le 0 \le \mathbf{s} < \mathbf{t} \le \mathbf{d}$. 
 
From Lemma~\ref{L:sawtooth} we have $m(u) = f(\mathbf{g}) \wedge f(\mathbf{g}-) + \alpha(u - \mathbf{g})$
for $\mathbf{g} \le u \le \mathbf{t}$ and
$f(u) \wedge f(u-) \ge f(\mathbf{g}) \wedge f(\mathbf{g}-) + \alpha(u - \mathbf{g})$
for $u \le \mathbf{t}$.  Therefore, 
$\inf\{f(u) \wedge f(u-) - \alpha u : u \le 0\} \ge f(\mathbf{g}) \wedge f(\mathbf{g}-) -  \mathbf{g}$,
and hence 
$\inf\{f(u) \wedge f(u-) - \alpha u : u \le 0\} 
= f(\mathbf{g}) \wedge f(\mathbf{g}-) -  \alpha \mathbf{g}$.
Now, by definition of $\mathbf{s}$,
$f(\mathbf{s}) \wedge f(\mathbf{s-}) - \alpha \mathbf{s} 
\le 
\inf\{f(u) \wedge f(u-) - \alpha u : u \le 0\}$,
and so
\[
\begin{split}
f(\mathbf{s}) \wedge f(\mathbf{s-}) 
& \le f(\mathbf{g}) \wedge f(\mathbf{g}-) -  \alpha \mathbf{g} + \alpha \mathbf{s} \\
& = f(\mathbf{g}) \wedge f(\mathbf{g}-) +  \alpha (\mathbf{s} - \mathbf{g}) \\
& = m(\mathbf{s}), \\
\end{split}
\]
which contradicts $\mathbf{d} = \inf\{u > 0 : f(u) \wedge f(u-) = m(u)\}
= \inf\{u > 0 : f(u) \wedge f(u-) \le m(u)\}$ unless 
$\mathbf{s}=0$ and $f(0) \wedge f(0-) = m(0)$, but we have
assumed that this is not the case.

A similar argument shows that if
$\mathbf{t}=\mathbf{s}$, then 
$\mathbf{t}=\mathbf{s}=\mathbf{d}$.

\end{proof}

\def\cprime{$'$}
\providecommand{\bysame}{\leavevmode\hbox to3em{\hrulefill}\thinspace}
\providecommand{\MR}{\relax\ifhmode\unskip\space\fi MR }
\providecommand{\MRhref}[2]{%
  \href{http://www.ams.org/mathscinet-getitem?mr=#1}{#2}
}
\providecommand{\href}[2]{#2}


\begin{thebibliography}{APRUB11}

\bibitem[APRUB11]{ECP2011-38}
J.~Abramson, J.~Pitman, N.~Ross, and G.~Uribe~Bravo, \emph{Convex minorants of
  random walks and {L}\'evy processes}, Electronic Communications in
  Probability \textbf{16} (2011), 423--434.

\bibitem[Bal77]{MR0452694}
E.~J. Balder, \emph{An extension of duality-stability relations to nonconvex
  optimization problems}, SIAM J. Control Optimization \textbf{15} (1977),
  no.~2, 329--343. \MR{0452694 (56 \#10973)}

\bibitem[Bas84]{MR770946}
Richard~F. Bass, \emph{Markov processes and convex minorants}, Seminar on
  probability, {XVIII}, Lecture Notes in Math., vol. 1059, Springer, Berlin,
  1984, pp.~29--41. \MR{770946 (86d:60086)}

\bibitem[Ber96]{bertoin}
Jean Bertoin, \emph{L\'evy processes}, Cambridge Tracts in Mathematics, vol.
  121, Cambridge University Press, Cambridge, 1996. \MR{1406564 (98e:60117)}

\bibitem[Ber97a]{bertoin_regen}
\bysame, \emph{Regenerative embedding of {M}arkov sets}, Probab. Theory Related
  Fields \textbf{108} (1997), no.~4, 559--571. \MR{1465642 (98h:60013)}

\bibitem[Ber97b]{bertoin97}
\bysame, \emph{Regularity of the half-line for {L}\'evy processes}, Bull. Sci.
  Math. \textbf{121} (1997), no.~5, 345--354. \MR{1465812 (98m:60115)}

\bibitem[Ber00]{MR1747095}
\bysame, \emph{The convex minorant of the {C}auchy process}, Electron. Comm.
  Probab. \textbf{5} (2000), 51--55 (electronic). \MR{1747095 (2001i:60081)}

\bibitem[BS02]{handbook}
Andrei~N. Borodin and Paavo Salminen, \emph{Handbook of {B}rownian
  motion---facts and formulae}, second ed., Probability and its Applications,
  Birkh\"auser Verlag, Basel, 2002. \MR{1912205 (2003g:60001)}

\bibitem[CD01]{MR1891744}
Chris Carolan and Richard Dykstra, \emph{Marginal densities of the least
  concave majorant of {B}rownian motion}, Ann. Statist. \textbf{29} (2001),
  no.~6, 1732--1750. \MR{1891744 (2003c:62030)}

\bibitem[{\c{C}}in92]{MR1145458}
Erhan {\c{C}}inlar, \emph{Sunset over {B}rownistan}, Stochastic Process. Appl.
  \textbf{40} (1992), no.~1, 45--53. \MR{1145458 (93f:60120)}

\bibitem[Die88]{MR955448}
H.~Dietrich, \emph{Zur {$c$}-{K}onvexit\"at und {$c$}-{S}ubdifferenzierbarkeit
  von {F}unktionalen}, Optimization \textbf{19} (1988), no.~3, 355--371.
  \MR{955448 (89j:90268)}

\bibitem[Don07]{doney}
Ronald~A. Doney, \emph{Fluctuation theory for {L}\'evy processes}, Lecture
  Notes in Mathematics, vol. 1897, Springer, Berlin, 2007, Lectures from the
  35th Summer School on Probability Theory held in Saint-Flour, July 6--23,
  2005, Edited and with a foreword by Jean Picard. \MR{2320889 (2008m:60085)}

\bibitem[EN74]{MR0348591}
Karl-Heinz Elster and Reinhard Nehse, \emph{Zur {T}heorie der
  {P}olarfunktionale}, Math. Operationsforsch. Statist. \textbf{5} (1974),
  no.~1, 3--21. \MR{0348591 (50 \#1088)}

\bibitem[Fag09]{faggionato}
Alessandra Faggionato, \emph{The alternating marked point process of
  {$h$}-slopes of drifted {B}rownian motion}, Stochastic Process. Appl.
  \textbf{119} (2009), no.~6, 1765--1791. \MR{2519344 (2010e:60175)}

\bibitem[Fou98]{fourati}
S.~Fourati, \emph{Points de croissance des processus de {L}\'evy et th\'eorie
  g\'en\'erale des processus}, Probab. Theory Related Fields \textbf{110}
  (1998), no.~1, 13--49. \MR{1602032 (99e:60164)}

\bibitem[FT88]{fitztaksar}
P.~J. Fitzsimmons and Michael Taksar, \emph{Stationary regenerative sets and
  subordinators}, Ann. Probab. \textbf{16} (1988), no.~3, 1299--1305.
  \MR{942770 (89m:60176)}

\bibitem[Gro83]{MR714964}
Piet Groeneboom, \emph{The concave majorant of {B}rownian motion}, Ann. Probab.
  \textbf{11} (1983), no.~4, 1016--1027. \MR{714964 (85h:60119)}

\bibitem[GS74]{MR0334335}
R.~K. Getoor and M.~J. Sharpe, \emph{Last exit decompositions and
  distributions}, Indiana Univ. Math. J. \textbf{23} (1973/74), 377--404.
  \MR{0334335 (48 \#12654)}

\bibitem[HJL92a]{MR1168181}
Pierre Hansen, Brigitte Jaumard, and Shi-Hui Lu, \emph{Global optimization of
  univariate {L}ipschitz functions. {I}. {S}urvey and properties}, Math.
  Programming \textbf{55} (1992), no.~3, Ser. A, 251--272. \MR{1168181
  (93f:90149)}

\bibitem[HJL92b]{MR1168182}
\bysame, \emph{Global optimization of univariate {L}ipschitz functions. {II}.
  {N}ew algorithms and computational comparison}, Math. Programming \textbf{55}
  (1992), no.~3, Ser. A, 273--292. \MR{1168182 (93f:90150)}

\bibitem[HT93]{MR1274246}
Reiner Horst and Hoang Tuy, \emph{Global optimization}, second ed.,
  Springer-Verlag, Berlin, 1993, Deterministic approaches. \MR{1274246
  (94m:90004)}

\bibitem[HU80a]{MR593233}
J.-B. Hiriart-Urruty, \emph{Extension of {L}ipschitz functions}, J. Math. Anal.
  Appl. \textbf{77} (1980), no.~2, 539--554. \MR{593233 (83i:58013)}

\bibitem[HU80b]{MR600082}
\bysame, \emph{Lipschitz {$r$}-continuity of the approximate subdifferential of
  a convex function}, Math. Scand. \textbf{47} (1980), no.~1, 123--134.
  \MR{600082 (82c:58007)}

\bibitem[Kal02]{MR1876169}
Olav Kallenberg, \emph{Foundations of modern probability}, second ed.,
  Probability and its Applications (New York), Springer-Verlag, New York, 2002.
  \MR{1876169 (2002m:60002)}

\bibitem[Lev03]{MR2027382}
V.~L. Levin, \emph{Abstract convexity in measure theory and in convex
  analysis}, J. Math. Sci. (N. Y.) \textbf{116} (2003), no.~4, 3432--3467,
  Optimization and related topics, 4. \MR{2027382 (2004k:90163)}

\bibitem[Luc09]{MR2496900}
Yves Lucet, \emph{What shape is your conjugate? {A} survey of computational
  convex analysis and its applications}, SIAM J. Optim. \textbf{20} (2009),
  no.~1, 216--250. \MR{2496900 (2010e:52036)}

\bibitem[Mai71]{maisonregen2}
B.~Maisonneuve, \emph{Ensembles r\'eg\'en\'eratifs, temps locaux et
  subordinateurs}, S\'eminaire de {P}robabilit\'es, {V} ({U}niv. {S}trasbourg,
  ann\'ee universitaire 1969--1970), Springer, Berlin, 1971, pp.~147--169.
  Lecture Notes in Math., Vol. 191. \MR{0448586 (56 \#6892)}

\bibitem[Mai83]{maisonregen1}
\bysame, \emph{Ensembles r\'eg\'en\'eratifs de la droite}, Z. Wahrsch. Verw.
  Gebiete \textbf{63} (1983), no.~4, 501--510. \MR{705620 (85g:60076)}

\bibitem[Mil73]{MR0321198}
P.~W. Millar, \emph{Exit properties of stochastic processes with stationary
  independent increments}, Trans. Amer. Math. Soc. \textbf{178} (1973),
  459--479. \MR{0321198 (47 \#9731)}

\bibitem[Mil77a]{millarzeroone}
\bysame, \emph{Zero-one laws and the minimum of a {M}arkov process}, Trans.
  Amer. Math. Soc. \textbf{226} (1977), 365--391. \MR{0433606 (55 \#6579)}

\bibitem[Mil77b]{millar_77}
P.~Warwick Millar, \emph{Random times and decomposition theorems}, Probability
  ({P}roc. {S}ympos. {P}ure {M}ath., {V}ol. {XXXI}, {U}niv. {I}llinois,
  {U}rbana, {I}ll., 1976), Amer. Math. Soc., Providence, R. I., 1977,
  pp.~91--103. \MR{0443109 (56 \#1482)}

\bibitem[Mil78]{millarpostmin}
P.~W. Millar, \emph{A path decomposition for {M}arkov processes}, Ann.
  Probability \textbf{6} (1978), no.~2, 345--348. \MR{0461678 (57 \#1663)}

\bibitem[NO05]{MR2217473}
V.~I. Norkin and B.~O. Onishchenko, \emph{Minorant methods of stochastic global
  optimization}, Kibernet. Sistem. Anal. \textbf{41} (2005), no.~2, 56--70,
  189. \MR{2217473 (2007a:90053)}

\bibitem[NP89a]{pitmanneveu1}
J.~Neveu and J.~Pitman, \emph{Renewal property of the extrema and tree property
  of the excursion of a one-dimensional {B}rownian motion}, S\'eminaire de
  {P}robabilit\'es, {XXIII}, Lecture Notes in Math., vol. 1372, Springer,
  Berlin, 1989, pp.~239--247. \MR{1022914 (91e:60239)}

\bibitem[NP89b]{pitmanneveu2}
J.~Neveu and J.~W. Pitman, \emph{The branching process in a {B}rownian
  excursion}, S\'eminaire de {P}robabilit\'es, {XXIII}, Lecture Notes in Math.,
  vol. 1372, Springer, Berlin, 1989, pp.~248--257. \MR{1022915 (91e:60240)}

\bibitem[Pit83]{MR733673}
J.~W. Pitman, \emph{Remarks on the convex minorant of {B}rownian motion},
  Seminar on stochastic processes, 1982 ({E}vanston, {I}ll., 1982), Progr.
  Probab. Statist., vol.~5, Birkh\"auser Boston, Boston, MA, 1983,
  pp.~219--227. \MR{733673 (85f:60119)}

\bibitem[PR69]{pechrogo}
E.~A. Pecherskii and B.~A. Rogozin, \emph{On joint distributions of random
  variables associated with fluctuations of a process with independent
  increments}, Theory of Probability and its Applications \textbf{14} (1969),
  no.~3, 410--423.

\bibitem[PS72]{MR0297019}
A.~O. Pittenger and C.~T. Shih, \emph{Coterminal families and the strong
  {M}arkov property}, Bull. Amer. Math. Soc. \textbf{78} (1972), 439--443.
  \MR{0297019 (45 \#6077)}

\bibitem[PU11]{pitmanbravo}
J.~{Pitman} and G.~{Uribe Bravo}, \emph{The convex minorant of a {L}\'evy
  process}, Ann. Probab. (2011), To appear.

\bibitem[Rog68]{rogozin}
B.~A. Rogozin, \emph{The local behavior of processes with independent
  increments}, Teor. Verojatnost. i Primenen. \textbf{13} (1968), 507--512.
  \MR{0242261 (39 \#3593)}

\bibitem[RR98]{MR1619170}
Svetlozar~T. Rachev and Ludger R{\"u}schendorf, \emph{Mass transportation
  problems. {V}ol. {I}}, Probability and its Applications (New York),
  Springer-Verlag, New York, 1998, Theory. \MR{1619170 (99k:28006)}

\bibitem[RW87]{rogerswills}
L.~C.~G. Rogers and David Williams, \emph{Diffusions, {M}arkov processes, and
  martingales. {V}ol. 2}, Wiley Series in Probability and Mathematical
  Statistics: Probability and Mathematical Statistics, John Wiley \& Sons Inc.,
  New York, 1987, It{\^o} calculus. \MR{921238 (89k:60117)}

\bibitem[RW98]{MR1491362}
R.~Tyrrell Rockafellar and Roger J.-B. Wets, \emph{Variational analysis},
  Grundlehren der Mathematischen Wissenschaften [Fundamental Principles of
  Mathematical Sciences], vol. 317, Springer-Verlag, Berlin, 1998. \MR{1491362
  (98m:49001)}

\bibitem[{\v{S}}ta65]{shtatland}
E.~S. {\v{S}}tatland, \emph{On local properties of processes with independent
  increments}, Teor. Verojatnost. i Primenen. \textbf{10} (1965), 344--350.
  \MR{0183022 (32 \#504)}

\bibitem[Sui01]{suidan}
T.~M. Suidan, \emph{Convex minorants of random walks and {B}rownian motion},
  Teor. Veroyatnost. i Primenen. \textbf{46} (2001), no.~3, 498--512.
  \MR{MR1978665 (2004d:60097)}

\bibitem[Vig]{vigon_unpublished}
Vincent Vigon, \emph{D\'eriv\'ees de dini des processus de l\'evy},
  \url{http://www-irma.u-strasbg.fr/~vigon/boulot/pas_publication/fichiers/con%
jecture.pdf}.

\bibitem[Vig02]{MR1875147}
\bysame, \emph{Votre {L}\'evy rampe-t-il?}, J. London Math. Soc. (2)
  \textbf{65} (2002), no.~1, 243--256. \MR{1875147 (2002i:60101)}

\bibitem[Vig03]{abrupt}
\bysame, \emph{Abrupt {L}\'evy processes}, Stochastic Process. Appl.
  \textbf{103} (2003), no.~1, 155--168. \MR{1947963 (2003m:60124)}

\bibitem[Vil09]{MR2459454}
C{\'e}dric Villani, \emph{Optimal transport}, Grundlehren der Mathematischen
  Wissenschaften [Fundamental Principles of Mathematical Sciences], vol. 338,
  Springer-Verlag, Berlin, 2009, Old and new. \MR{2459454 (2010f:49001)}

\end{thebibliography}
\end{document}